\documentclass[11pt, leqno]{amsart}

\usepackage[colorinlistoftodos,prependcaption]{todonotes}

\usepackage{amsfonts, amsmath, 
}

\allowdisplaybreaks
\usepackage{tikz}	
\usepackage{mathtools}
\usepackage{esint}
\usepackage[margin=1.4in]{geometry}
\usepackage{soul}

\usepackage{amssymb,amsthm,
paralist
}

\usepackage{
latexsym,
}

\usepackage{enumitem}

\definecolor{darkgreen}{rgb}{0,0.5,0}
\definecolor{darkred}{rgb}{0.7,0,0}

\usepackage[colorlinks, 
citecolor=darkgreen, linkcolor=darkred
]{hyperref}

\theoremstyle{plain}
\newtheorem{lemma}{Lemma}[section]
\newtheorem{thm}[lemma]{Theorem}
\newtheorem{prop}[lemma]{Proposition}
\newtheorem{cor}[lemma]{Corollary}

\newtheorem{pb}[lemma]{Problem}
\newtheorem{ques}[lemma]{Question}

\theoremstyle{definition}
\newtheorem{defn}[lemma]{Definition}

\newtheorem{rmk}[lemma]{Remark}

\numberwithin{equation}{section}

\newcommand{\D}{{\rm D}}

\newcommand{\of}{{\circ}}

\newcommand{\boundary}{\partial}

\newcommand{\partt}{ {\frac{\partial}{\partial t} } }

\renewcommand{\phi}{\varphi}
\newcommand{\ti}{\tilde}
\newcommand{\parts}{\frac{\partial}{\partial s} }
\newcommand{\al}{\alpha}
\newcommand{\be}{\beta}
\newcommand{\ga}{\gamma}

\newcommand{\de}{\delta}

\newcommand{\Haus}{\mathcal H}

\newcommand{\la}{\lambda}

\newcommand{\si}{\sigma}

\newcommand{\ep}{\varepsilon}

\newcommand{\curlX}{\mathcal X}

\newcommand{\Cone}{\mathcal C}

\newcommand{\R}{\ensuremath{{\mathbb R}}}

\newcommand{\N}{\ensuremath{{\mathbb N}}}
\newcommand{\B}{\ensuremath{{\mathbb B}}}

\newcommand{\downto}{\searrow}

\newcommand{\lap}{\Delta}

\newcommand{\grad}{\nabla}

\newcommand{\vol}{{ \rm Vol}}

\DeclareMathOperator{\AVR}{AVR}

\newcommand{\beq}{\begin{equation}}
\newcommand{\eeq}{\end{equation}}
\newcommand{\beqa}{\begin{equation}\begin{aligned}}
\newcommand{\eeqa}{\end{aligned}\end{equation}}
\newcommand{\brmk}{\begin{rmk}}
\newcommand{\ermk}{\end{rmk}}
\newcommand{\partref}[1]{\hbox{(\csname @roman\endcsname{\ref{#1}})}}

\newcommand{\Rm}{{\mathrm{Rm}}}

\newcommand{\Rc}{{\mathrm{Ric}}}

\newcommand{\Sc}{{\mathrm{R}}}

\usepackage{soul}

\newsavebox\CBox
\newcommand\hcancel[2][0.5pt]{%
  \ifmmode\sbox\CBox{$#2$}\else\sbox\CBox{#2}\fi%
  \makebox[0pt][l]{\usebox\CBox}%
  \rule[0.5\ht\CBox-#1/2]{\wd\CBox}{#1}}

\def\Xint#1{\mathchoice
{\XXint\displaystyle\textstyle{#1}}%
{\XXint\textstyle\scriptstyle{#1}}%
{\XXint\scriptstyle\scriptscriptstyle{#1}}%
{\XXint\scriptscriptstyle\scriptscriptstyle{#1}}%
\!\int}
\def\XXint#1#2#3{{\setbox0=\hbox{$#1{#2#3}{\int}$ }
\vcenter{\hbox{$#2#3$ }}\kern-.6\wd0}}

\def\dashint{\Xint-}

\title[Initial stability estimates for Ricci flow]{Initial stability estimates for Ricci flow and three dimensional Ricci-pinched manifolds}

\author{Alix Deruelle}
\address{Laboratoire de Math\'ematiques d'Orsay, Universit\'e Paris-Saclay, CNRS,  F-91405, Orsay, France}
\email{alix.deruelle@universite-paris-saclay.fr}
\author{Felix Schulze}
\address{Mathematics Institute, Warwick University, Gibbet hill road, Coventry CV4 7AL, UK}
\email{felix.schulze@warwick.ac.uk}
\author{Miles Simon}
\address{Institut f\"ur Analysis und Numerik (IAN), Universit\"at Magdeburg, Universit\"atsplatz~2\\${\ }\ \, $ 39106 Magdeburg, Germany }
\email{msimon@ovgu.de}

\begin{document}

\begin{abstract}
This paper investigates the question of stability for a class of Ricci flows   which start at possibly non-smooth    metric spaces.  We show that if the initial metric space is  Reifenberg   and locally bi-Lipschitz to Euclidean space,   then  two solutions to the Ricci flow whose Ricci curvature is uniformly bounded from below and whose curvature is bounded by $c\cdot t^{-1}$ converge to one another at an exponential rate once they have been appropriately gauged. As an application, we show that smooth three dimensional, complete, uniformly Ricci-pinched Riemannian manifolds with bounded curvature are either compact or flat, thus confirming a conjecture of Hamilton and Lott.
\end{abstract}  

\maketitle

\parskip=0.5pt
\tableofcontents
\parskip=0pt






\section{Introduction}
\subsection{Overview}\label{overview}
In this paper, we consider smooth solutions $(M^n,g(t))_{t\in (0,T)}$ to Ricci flow defined on smooth, connected manifolds satisfying for $t\in(0,T)$,
\begin{equation}\label{CurvatureCond}
\quad |\Rm(g(t))| \leq \frac{c_0}{ t}\ \text{ and }\ \Rc(g(t)) \geq -g(t),
\end{equation} 
where $c_0$ is a positive time-independent constant. Notice that we don not assume the metrics $g(t)$ to be complete a priori.

This   setting has been shown to occur in many situations, one  prominent one being that of expanding solitons with non-negative curvature operator coming out of cones with non-negative curvature operator: see for example \cite{SchulzeSimon}, \cite{Der-Smo-Pos-Cur-Con}, \cite{SiTo2}, \cite{BamCab-RivWil}.

Using Lemma \ref{SiTopThm1} (see also \cite[Lemma 3.1]{SiTo2}), we see that the above  setting guarantees that  the distances    $d_t \coloneqq  d(g(t))$  on $M$ converge locally in a strong sense. More explicitly, assuming \eqref{CurvatureCond},   we have:
\begin{equation}\label{DistCond}
\begin{split}
 \text{ for all } x_0 \in M, \,& \text{there exists a connected open neighbor  hood of $x_0,$ }  U_{x_0} \Subset M\\
 \text{ and } &S  >0  \text{ such that for $0<s\leq t<S ,$}  \\ 
 &e^{t-s}d_s  \geq d_t  \geq d_s - c_0  \sqrt{t-s}\,  \text{ on  }  U_{x_0},
\end{split}
\end{equation}
and hence   there exists a unique limit $d_0 \coloneqq  \lim_{t\downto 0}d_t$ on $U_{x_0}$, which is attained uniformly, such that $(U_{x_0},d_0)$ is a metric space. 
Note that if for different points $x_0,y_0\in M$ we obtain $d_0= \lim_{t\downto 0}d_t$ on $U_{x_0}$ and $\ti d_0= \lim_{t\downto 0}d_t$ on $U_{y_0}$ then we have $d_0 = \ti d_0$ on $ U_{x_0} \cap U_{y_0},$ since they both are obtained as the uniform  limit of $d(g(t))$  as $t\to 0$. For this reason we  do not include a quantifier depending on $x_0$ in the definition of $d_0 \coloneqq  \lim_{t\downto 0}d_t$ on $U_{x_0}.$ 
We are interested in the following problem: 
\begin{pb} \label{main-pb}
Let $(M_i^n,g_i(t))_{t\in(0,T)}, i=1,2$
be two smooth, connected  (possibly incomplete)  Ricci flows satisfying
\eqref{CurvatureCond}  and \eqref{DistCond} and converging locally to the same  metric space, up to an isometry  as $t$  tends to $0$, that is: 
$ \lim_{t\downto 0}d(g_1(t)) = d_{0,1}$ on $U_{p,1}$
 $\lim_{t\downto 0}d(g_2(t)) = d_{0,2},$  on $U_{p,2}$
and $\psi_0(U_{p,1}) =  U_{p,2}$, where $\psi_0:U_{p,1} \to U_{p,2} = \psi_0 (U_{p,1})$ is a homeomorphism with 
$ \psi_0^*(d_{0,2}) = d_{0,1}.$ Then we are concerned with the following problems:\\[1ex]
What further assumptions on the regularity of $d_0$ ensure that\\[-1.5ex]
\begin{enumerate}
\item  there is a suitable {\it gauge} in which we can compare the solutions $g_1$ and $g_2$ effectively?\\[-1ex]
\item  the solutions $(g_1(t))_{t\in(0,T)}$ and $(g_2(t))_{t\in(0,T)}$ remain close  to one another for $t$ close to zero in this gauge?
\end{enumerate}
\end{pb}

Our fundamental regularity assumption on $d_0$ is the  following {\it Reifenberg} property: 
\begin{equation}
\label{ReifenbergCond}
 \text{ for all } p \in M , \text{ for all } x \in U_p,  \text{ every tangent cone at  } x \text{ of } (U_p,d_0)  \text{ is }(\R^n,d(\delta)),
\end{equation} 
where $d(\delta)$ stands for Euclidean distance.
In fact,   if  $(M,d_0)$ is the local limit of a non-collapsing sequence of complete, smooth Riemannian manifolds with bounded curvature and Ricci curvature uniformly bounded from below,       this condition   can be turned into a uniform Reifenberg condition (see Lemma \ref{UniformReifenbergLemma}), 
\begin{equation}
\begin{split}
\label{UniformReifenbergCond}
& \text{for all  $p \in M$}\text{ and  for all $\ep>0$, there exist $r>0$ and a neighborhood $U_p\Subset M$}    \\[-0.5ex]
& \text{such that}\,\, d_{GH}(B_{s^{-1}d_0}(x,1),\B(0,1)) < \ep, \text{ for all } s <r , 
 \text{ and for all } x \in U_p\\[-0.5ex]
   &\text{such that } B_{d_0}(x,s) \Subset U_p. 
\end{split}
\end{equation}

Assumption \eqref{ReifenbergCond} is infinitesimal and means that there are no {\it singular}  points in $(U_p,d_0)$ for any $p \in M$. However, it does not   necessarily mean that the distance is induced by a H\"older continuous Riemannian metric, let alone a Lipschitz Riemannian metric: see \cite[Theorem $1.2$]{Col-Nab-Branching} for such a counterexample.
The same paper does   however show, under the assumption that  $(U_p,d_0)$ is a complete, uniformly Reifenberg, non-collapsed Ricci limit space, that there does exist a bi-Lipschitz embedding into $\R^N$ for a sufficiently large $N$ around each point $x\in U_p$. 

The Reifenberg condition allows us to show that a uniform Pseudolocality estimate holds:  see Lemma \ref{lem:curv_decay}. In the case that the solution  is complete and has bounded curvature, we use  Perelman's Pseudolocality and the uniform Reifenberg result from Lemma \ref{UniformReifenbergLemma} to prove this. In the general case, we use Lemma \ref{pseudorefienberg}, which is valid in the setting we consider here. 

We assume as a further regularity assumption on $d_0$ that there is a bi-Lipschitz chart around each point in $(U_p,d_0)$  given by distance coordinates, and that the Lipschitz constant is $\ep_0$ close to $1$:
\begin{gather}
\label{CoordCond}
\begin{split}
& \text{For any $x_0\in U_p$, there exist   radii  $\ti R> 2R> 0 $ such that $B_{d_0}(x_0,\ti R  ) \Subset U_{p}$ } \\[-0.5ex]
& \text{and points $a_1, \ldots, a_n \in B_{d_0}(x_0,\ti R )$ such that the map } \\[-0.5ex]
& D_0: \, \left\{
   \begin{array}{rcl}
       B_{d_0}(x_0,\ti R)  & \rightarrow &\R^n \cr
       x& \rightarrow & (d_0(a_1,x)-d_0(a_1,x_0), \ldots, d_0(a_n,x)-d_0(a_n,x_0)),
   \end{array}\right.   \\[-0.5ex]
&  \text{is a }   (1+\ep_0)  \text{ bi-Lipschitz homeomorphism from $B_{d_0}(x_0,2R)$ onto its image.  }  \cr  
\end{split}
\end{gather}

Assumption \eqref{CoordCond} is local and always holds true in the case that  $(U_p,d_0)$ is an Alexandrov space with curvature bounded from below  satisfying   \eqref{ReifenbergCond}: See \cite[Theorem $10.8.4$]{BBI}.
For example, this would be the case, if we assume 
\eqref{ReifenbergCond} and replace  the lower bound on the Ricci curvature in \eqref{CurvatureCond}, by ${\rm sec}_{g(t)} \geq -1$, $t\in(0,T)$, where ${\rm sec}$ refers to sectional curvature. \\

We use the following standard notation in the rest of the paper:
given a solution $(M,g(t)))_{t\in [0,T)}$ to Ricci flow and a Riemannian manifold $(N,h),$ we call a map  
  $F: H \times [0,T]  \to N$   a {\it solution to the Ricci Harmonic map heat flow} if
  $\partt F(\cdot,t)= \lap_{g(t),h} F(\cdot,t)$  on $H.$ If $F_t:=F(\cdot,t)$ is a diffeomorphism, then  it is well known that 
  $\ti g(t):= (F_t)_{*}g(t)$ is a solution to the $h$-Ricci DeTurck flow, which is a parabolic evolution equation.
  In the case that $N$ is a subset of $\R^n$ and $h=\de,$ each  of the components of $F= (F^1 ,\ldots, F^n )$ is a solution to the heat flow where the metric evolves by  Ricci flow : $\partt  F^i(\cdot,t)= \lap_{g(t)}  F^i(\cdot,t),$ for $i\in \{1, \ldots,n\}.$  
 See Chapter 6  of \cite{HamFor} for further definitions   and explanations. 
 
We establish the following initial stability estimate which addresses the question posed in Problem \ref{main-pb}.

\begin{thm}\label{thm:main.1}
  For all $n\in \N$ and $\be_0\in (0,1)$ there exists $\ep_0=\ep_0(n,\be_0)>0$, depending only on $\be_0$ and $n$,  such that the following holds. Let $(M_i^n,g_i(t))_{t\in (0,T)}$, $i=1,2$, be smooth solutions to Ricci flow (not necessarily with complete time slices)  both  satisfying \eqref{CurvatureCond} and \eqref{DistCond}, such that  the locally well defined metrics at time zero agree (up to an isometry),  that is  for arbitrary $p\in M_1,$ 
 $ \lim_{t\downto 0}d(g_1(t)) = d_{0,1}$ on $U_{p,1}$, and 
 $\lim_{t\downto 0}d(g_2(t)) = d_{0,2},$  on $U_{\psi_0(p),2}$  
for  an  isometry    $\psi_0:(U_{p,1},d_{0,1}) \to (U_{\hat p,2} =\psi_0 (U_{p,1}), d_{0,2}),$ $\hat p:= \psi_0(p),$ that is
 $\psi_0:U_{p,1} \to U_{\hat p,2} = \psi_0 (U_{ p,1})$ is a homeomorphism with 
$ \psi_0^*(d_{0,2}) = d_{0,1}.$  

We assume \eqref{ReifenbergCond} holds true for the locally defined metric $d_0= {d_{0,1}},$   that    $p\in M_1,$
$x_0 \in U_{p,1},$   and that $\eqref{CoordCond}$  holds for some $\ti R> 2R >0,$ for   the locally defined metric $d_0= {d_{0,1}}$ on $U_p= U_{p,1}$ (and hence for $d_0= {d_{0,1}}$  on  $U_p= U_{p,2}$). 

Then       there exists an $R_0\in(0,1)$ and $ 0<T_0,C_0< \infty $    and solutions $(\ti{g}_1(t))_{t\in(0,T_0)}$ and $(\ti g_2(t))_{t\in(0,T_0)}$ to $\de$-Ricci-DeTurck flow defined on $\B(0,R_0)\times(0,T_0)$ satisfying the following:
 \begin{enumerate}
 \item The metrics $\ti{g}_1(t)$ and $\ti g_2(t)$ are $(1\pm \be_0)$ close to the $\delta$-metric for $t\in(0,T_0)$,\\[-2ex]
 \item On  $\B(0,R_0)$,
  \begin{equation}\label{exp-dec-main-thm}
|\ti g_1(t) -\ti g_2(t)|_{\de}\leq \exp\left(-\tfrac{C_0}{t}\right), \quad\text{for all $t\in(0,T_{0}]$.}
\end{equation}
Moreover, for each $k\geq 0$, there exists $C_k>0$ such that on $\B(0,R_0)$:
\begin{equation}\label{exp-dec-main-thm-higher-cov-der}
|D^k\left(\ti g_1(t) -\ti g_2(t)\right)|_{\de}\leq \exp\left(-\tfrac{C_k}{t}\right), \quad\text{for all $t\in(0,T_{0}]$.}
\end{equation}

\item Moreover, $\ti{g}_1(t)=(F_1(t))_{\ast}g_1(t)$ and $\ti g_2(t)=(F_2(t))_{\ast}g_2(t),$ where $( F_1(t))_{t\in(0,T_0)}$ and $(F_2(t))_{t\in(0,T_0)}$ are smooth families of bi-Lipschitz maps on $B_{d_{0,1}}(x_0, \frac{3}{2}R_0)$ respectively $B_{d_{0,2}}(\psi_0(x_0), \frac{3}{2}R_0)$ that satisfy the following:\\[-2ex]
\begin{enumerate} 
\item The family of maps $(F_1(t))_{t\in(0,T_0)}$ (respectively $(F_2(t))_{t\in(0,T_0)}$) is a solution to the Ricci-harmonic map flow with respect to the Ricci flow solution $(g_1(t))_{t\in(0,T_0)}$ (respectively $(g_2(t))_{t\in(0,T_0)}$).\\[-2ex]
\item The initial value for $F_1$ is   $D_0$ and for $F_2$ is $D_0\, \of\, (\psi_0)^{-1}$ and the   convergence rate at $t=0$ is 
\begin{equation}
\begin{split}
 | F_1(t)- D_0| &\leq C_0\sqrt{t},\quad t\in(0,T_0),\,\,\text{ on $B_{d_{0,1}}\big(x_0,\tfrac{3}{2}R_0\big) $, }\\
 | F_2(t)- D_0\, \of\, (\psi_0)^{-1}| &\leq C_0\sqrt{t},\quad t\in(0,T_0),\,\, \text{ on $B_{d_{0,2}}\big(\psi_0(x_0),\tfrac{3}{2}R_0\big),$}
 \end{split}
 \end{equation}
  where  $D_0$ are the distance coordinates on $B_{d_{0,1}}(x_0,\frac{3}{2}R_0) $ coming from \eqref{CoordCond}.\\[-2ex]
\item The distances $\ti d_{1}(t) \coloneqq d(\ti g_{1}(t))$ and $\tilde{d}_{2}(t)\coloneqq 
d(\tilde{g}_2(t))$ converge, uniformly to the same distance $(D_0)_{\ast}d_{0,1}$ on $\B(0,R_0)$ as $t$ approaches $0$.
\end{enumerate}
\end{enumerate}
\end{thm}

 As an application of this result, we show that the approach of Lott \cite{Lott-Ricci-pinched} leads to a full resolution of a conjecture posed by Hamilton \cite[Conjecture $3.39$]{CLN} (in the setting of bounded curvature) and Lott \cite[Conjecture $1.1$]{Lott-Ricci-pinched}.
 
Recall that a Riemannian manifold $(M^n,g)$ is uniformly Ricci pinched if $\Rc(g)\geq 0$ and there exists a constant $c>0$ such that on $M$, $$\Rc(g)\,\geq\,c\,\Sc_g\,g,$$ in the sense of quadratic forms where $\Sc_g$ denotes the scalar curvature of the metric $g$. Notice that such a condition is invariant under rescalings.

\begin{thm}\label{thm:main.2}
Let $(M^3,g)$ be a smooth complete Riemannian manifold with bounded and uniformly pinched Ricci curvature. Then $(M^3,g)$ is either diffeomorphic  to a spherical space form or flat. In particular, if $M$ is non-compact then $(M^3,g)$ is flat.  
\end{thm}

Hamilton  introduced the Ricci flow in \cite{ham3D}, and in the case that $(M^3,g)$  is compact with non-negative uniformly pinched  Ricci curvature and the scalar curvature is positive at least at one point, the paper shows that the volume preserving Ricci flow of $(M^3,g)$ exists for all time and converges smoothly to a spherical space form. 
In the case that $(M^3,g)$  is compact and has non-negative  uniformly pinched  Ricci curvature and the scalar curvature is zero everywhere, then $M^3$ is Ricci-flat and hence flat. That is, the results of Hamilton imply Theorem \ref{thm:main.2}   immediately in the case that $M^3$ is compact.  

In case $(M^3,g)$ is non-compact with bounded curvature, Lott has proved Theorem \ref{thm:main.2} under the assumption that the sectional curvature of $g$ has a negative lower bound whose absolute value decays at least quadratically in the distance from a fixed point, improving a result of Chen-Zhu \cite{Chen-Zhu} where it is assumed that the metric $g$ has non-negative sectional curvature. Finally, Theorem \ref{thm:main.2} can be interpreted  as   an extension of Myers' theorem in dimension $3$.

Lott's approach to Hamilton's conjecture has two major  parts, which we briefly explain here.   The proof is by contradiction and the first part deals with Ricci flow and it does not use the lower bound on the sectional curvature. Lott shows the following crucial result:
\begin{thm}[\cite{Lott-Ricci-pinched}]\label{thm-lott-ricci-flow}
Let $(M^3, g_0)$ be a smooth, complete non-compact Riemannian 3-manifold which is uniformly Ricci pinched with bounded curvature.
Then there exists a complete Ricci flow solution $(M^3,g(t))$ that exists for all $t\geq 0$ with $g(0) = g_0$ such that:
\begin{enumerate}
\item the solution is Type III, i.e.~$|\Rm(g(t))|\leq \frac{C}{t}$ for all $t>0$ and some uniform positive constant $C$,
\item the solution is uniformly $c$-pinched, i.e.~there exists $c>0$ such that for $t>0$, $\Rc(g(t))\geq c\,\Sc_{g(t)}g(t)$ on $M$,
\item  If $\Sc_{g_0}(x_0)>0$ at some point, then the solution is uniformly non-collapsed for all positive time, i.e. 
\begin{equation*}
\operatorname{AVR}(g(t))\coloneqq \lim_{r\rightarrow+\infty}r^{-3}\vol_{g(t)}(B_{g(t)}(x_0,r))=v_0>0,
\end{equation*}
 for some (hence any) $x_0\in M$, $t>0$ and some time-independent positive constant $v_0$.
\end{enumerate}
\end{thm}
 See Section \ref{sec-comp-ric-pinch} for details. 

\begin{rmk} Shortly after this paper was released on ArXiv, the assumption
of bounded sectional curvature in \cite[Proposition
1.5]{Lott-Ricci-pinched} on the existence of an immortal solution to the Ricci flow,
was removed by Lee and Topping \cite{Lee-Top-3d}. Combining their work
with our Theorem \ref{thm:main.2} leads to a proof of Hamilton’s conjecture
without assuming the initial Ricci-pinched metric has bounded
curvature.
\end{rmk}

Notice that the second part of Lott's work does not use the Ricci flow. 
 In the following exposition, we only consider solutions from Theorem \ref{thm-lott-ricci-flow} where  (3) holds.  If one further assumes that the sectional curvature of the initial space is bounded from below by $-\frac{A}{r^2}$ with $r$ being the distance, and one blows down the solution, one obtains a solution having  the same properties, in some weak sense, coming out of any asymptotic cone of $(M^3,g_0)$. Based on results of Lebedeva-Petrunin \cite{Leb-Pet} using methods of Alexandrov geometry, it is then a static problem to show that such asymptotic cones must be flat which leads to a contradiction.

This is where our approach differs, as we will now explain. 
After blowing down Lott's Ricci flow solution, we first notice that the initial condition is a metric cone with no cone points outside the apex. This follows from a splitting theorem due to Hochard, see Proposition \ref{prop:Hochard-splitting}, for Type III non-collapsed solutions to the Ricci flow with non-negative Ricci curvature coming out of a metric space splitting a line. This result holds in any dimension. In particular, assumption \eqref{ReifenbergCond} is satisfied. Invoking results on \textit{$\operatorname{RCD}$} spaces (see Section \ref{sec-comp-ric-pinch} for references), the initial metric cone turns out to be an Alexandrov space with non-negative curvature. Therefore, Assumption \eqref{CoordCond} is satisfied as well. 
 In particular, the link of the cone is a two dimensional Alexandrov space with curvature not less than one, and may be approximated by smooth spaces with curvature larger than one, in view of results from the theory of Alexandrov spaces: See the references in Section \ref{sec-comp-ric-pinch}. 
Now, a consequence of the work of the first author \cite{Der-Smo-Pos-Cur-Con}   (alternatively the second two authors \cite{SchulzeSimon}, after approximating the cone by a smooth space with non-negative curvature operator, as in  \cite[Section 3.2]{PhDschlichting})    ensures the existence of a self-similarly expanding solution   coming out of the same initial metric cone that will serve as a comparison solution. This solution is an expanding gradient Ricci soliton with non-negative curvature operator which is not necessarily Ricci-pinched. If this would have been the case, we would be done by the work of \cite{Chen-Zhu}.  
Using Theorem \ref{thm:main.1} to compare this solution with the original Ricci pinched solution, we see that this non-negatively curved expanding soliton solution  is almost Ricci pinched. The error to be an exact Ricci-pinched expander is shown to be exponentially decaying at infinity: see Section \ref{alm-ric-pin-sec} for details. This allows us to conclude that the expander is Euclidean which implies that the initial metric cone is flat, leading to a contradiction.

Notice that in Lott's approach and that of the present paper, the main difficulty lies in the problem that one lacks regularity on any of the asymptotic cones of $(M^3,g_0)$, which is partially due to the fact that such an asymptotic cone does not a priori satisfy any kind of (elliptic) equation, and thus one struggles to make sense of the pinching on the initial metric cone. To circumvent this issue, we work directly with Lott's Ricci flow solution coming out of the cone, and we  use the parabolic nature of the Ricci flow equation to derive geometric properties for the initial condition.\\

Based on this result, a higher dimensional counterpart to Hamilton and Lott's conjecture would be:
\begin{ques}\label{ques-DSS}
Let $(M^n,g)$ be a connected non-compact complete Riemannian manifold with $2$-non-negative curvature operator. Assume $(M^n,g)$ is $2$-pinched. Is it true that $(M^n,g)$ is flat?
\end{ques}
Recall that a Riemannian manifold $(M^n,g)$ has $2$-non-negative curvature operator if the sum of the two lowest eigenvalues $\lambda_i(g)$, $i=1,2,$ of the curvature operator is non-negative, i.e. $\lambda_1(g)+\lambda_2(g)\geq 0$ on $M$. Similarly, a Riemannian manifold $(M^n,g)$ is $2$-pinched if there exists a constant $c>0$ such that $\lambda_1(g)+\lambda_2(g)\geq c\,\Sc_g$ on $M$.

An attempt to tackle   Question \ref{ques-DSS} with the help of the Ricci flow,  would most likely require one to first answer the following question:
\begin{ques}\label{ques-DSS-II}
Let $(M^n,g(t))_{t\in(0,+\infty)}$ be a complete non-compact Type III Ricci flow with $2$-non-negative curvature operator. Assume $(M^n,g(t))_{t\in(0,+\infty)}$ is uniformly $2$-pinched and uniformly non-collapsed at all scales, i.e.~there exists $c>0$ such that for $t>0$, $\lambda_1(g(t))+\lambda_2(g(t))\geq c\,\Sc_{g(t)}$ on $M$ and there exists $V_0>0$ such that for $t>0$, $\operatorname{AVR}(g(t))\geq V_0$. Is it true that $(M^n,g(t))_{t>0}$ is isometric to Euclidean space?
\end{ques}

In order to answer Question \ref{ques-DSS-II} and given the strategy we adopt in this paper, one would need to know whether there exists an expanding gradient Ricci soliton coming out of a metric cone that is Reifenberg outside its tip and which has nonnegative Ricci curvature in a weak sense (i.e. which is $\textit{$\operatorname{RCD}$}^*(0,3)$). This is in turn a wide open question in higher dimensions.

\subsection{Outline of paper}
Sections \ref{sec-curv-dec} to \ref{sec-exp-conv} are devoted to the proof of the main result of this paper, Theorem \ref{thm:main.1}. Section \ref{sec-curv-dec} explains how one improves the curvature decay as $t$ approaches $0$ in the setting of Theorem \ref{thm:main.1}. Section \ref{sec-exi-adj-map} is purely static, i.e.~it is independent of the Ricci flow and describes how to adjust two Riemannian metrics originally close in the Gromov-Hausdorff topology with the help of a suitable diffeomorphism in order to make them coincide at the level of metrics. Based on Section \ref{sec-exi-adj-map}, Section \ref{sec-rou-con-rate} is devoted to the adjustment of two solutions to $\de$-Ricci-DeTurck flow coming out of the same metric space so that they are comparable  at the level of {\it Riemannian metrics} along a sequence of times approaching $0$, this is Corollary \ref{cor-mod-DTRF}.  Sections \ref{sec-pol-con-rate} and \ref{sec-exp-conv} are the core of this paper and they establish a polynomial (respectively exponential) decay rate for the difference of the two (adjusted) solutions to $\de$-Ricci-DeTurck flow obtained in the previous section, thereby ending the proof of Theorem \ref{thm:main.1}. Section \ref{alm-ric-pin-sec} proves a rigidity result on expanding gradient Ricci soliton with non-negative Ricci curvature and almost Ricci pinched curvature. Section \ref{sec-comp-ric-pinch} gives a proof of  Theorem \ref{thm:main.2}.
\subsection{Notation} We collect notation used throughout this paper.
\begin{itemize}
\item[(1)] For a connected Riemannian manifold $(M^n,g),$ $x,y \in M,$ $r\in \R^+$:\\[-2ex]
\begin{itemize}
\item[(1a)]  $(M,d(g))$ refers to the associated metric space, 
$$d(g)(x,y) = \inf_{\ga \in G_{x,y}}L_g(\ga),$$ where $G_{x,y}$ refers to the set of smooth regular
  curves $\ga:[0,1] \to M,$ with
$\ga(0) = x,$ $\ga(1) = y$, and $L_g(\ga)$ is the length of $\ga$ with respect to $g$.
\item[(1b)] $B_{g}(x,r) \coloneqq  B_{d(g)}(x,r) \coloneqq  \{ y \in M \ | \ d(g)(y,x) < r\}$.
\item[(1c)] If $g$ is locally in $C^2$:  $\Rc(g)$ is the Ricci Tensor, $\Rm(g)$ is the Riemannian curvature tensor,  and  $\Sc_g$ is the scalar curvature.\\[-2ex]
\end{itemize}
\item[(2)] For a one parameter family $(g(t))_{t\in(0,T)}$ of Riemannian metrics on a manifold $M$, the distance induced by the metric $g(t)$ is denoted either by $d(g(t))$ or $d_t$ for $t\in(0,T)$.\\[-2ex]
\item[(3)] For a metric space $(X,d),$ $A \subseteq X$, $r\in \R^+,$  $B_{d}(A,r) \coloneqq  \{ y \in X \ | \ d(y,A) < r\}.$\\[-2ex]
\item[(4)] $\B(x,r)$ refers to an Euclidean ball with radius $r>0$ and centre $x \in \R^n$.\\[-2ex]
\item[(5)] For a metric space $(X,d),$ $\Haus^n_d(V)$ refers to the $n$-dimensional Hausdorff measure of a subset $V$ of $X$ with respect to $d$.\\[-2ex]
\item[(6)] For a Riemannian manifold $(M^n,g)$, $A \subseteq M$, $\vol_g(A)$ stands for the Riemannian volume of $A$ with respect to the metric $g$.
\end{itemize}

\subsection{Acknowldgements}
The first author is supported by grants ANR-17-CE40-0034 of the French National Research Agency ANR (Project CCEM) and ANR-AAPG2020 (Project PARAPLUI).
The third author is  supported by a grant in the Programm  `SPP-2026: Geometry at Infinity' of the German Research Council (DFG). 

We thank Christian Ketterer, Alexander Lytchak and Stephan Stadler for answering questions related to and explaining to us  their results on Alexandrov and \textit{$\operatorname{RCD}$} spaces.

{\it Note added for journal version}. One year after this paper appeared on Arxiv,  an alternative proof of the Hamilton conjecture using the inverse mean curvature flow,  given  by G. Huisken  and  T. K\"orber, appeared in arxiv   \cite{Huisken-Koerber}. 
\section{Curvature decay}\label{sec-curv-dec}
 To simplify the setup we will in addition to \eqref{CurvatureCond}-\eqref{ReifenbergCond} fix $x_0 \in M$ and $R>0$, reducing $T>0$ if necessary, such that 
\begin{equation}\label{eq:curv_decay.1}
B_{d_0}(x_0,200R + \sqrt{c_0 T}) \Subset M\, .
\end{equation}

 From Lemma \ref{SiTopThm1}, this implies  that $ B_{d_t}(x_0,200R) \Subset M $ for all $t \in (0,T)$ and  $B_{d_0}(x_0,200R)  \Subset M$, where $d_t\coloneqq d(g(t))$.

We record the following consequence of the Bishop-Gromov volume comparison, volume convergence,  and that almost volume cone implies almost metric cone structure.
\begin{lemma}\label{UniformReifenbergLemma}
Let $(  N_i^n,g_i,p_i)_{i\in \N} ,$   be a smooth sequence of complete manifolds, without boundary,   with  $\Rc(g_i) \geq -g_i$ on $N_i$ and let   
$(X,d,x_0)$ be a non-collapsed pointed Gromov-Hausdorff limit  of  $(N^n_i, g_i, p_i)_{i\in \N}.$ Assume $K \subseteq  X$ is compact and that all tangent cones of $X$ for $p\in K$ are Euclidean. Then for every $\varepsilon>0$, there exists $r_\varepsilon > 0$ such that  for any $p\in K$ and  $0<r\leq r_\varepsilon$  the ball   $B_d(p,r) $   is, after scaling $r$ to one, $\varepsilon$-close in Gromov-Hausdorff distance to a Euclidean ball $\B(0,1)\subseteq  \R^n$, i.e.~$(K,d)$ is uniformly Reifenberg.
\end{lemma}
\begin{proof}
Let $\delta_1 > 0$  be given. For all $p\in K$, there exists $0<r_p\leq 1$ such that for all $0<r\leq 2r_p$ we have that $B_d(p,r)$, after scaling $r$ to one, is $\delta_1/2$-close in Gromov-Hausdorff distance to $\B(0,1),$ in view of the  assumption that all tangent cones at points in $ K$ of $X$ are Euclidean. Since $K$ is compact 
we can select finitely many points $p_1, \ldots, p_N$ such that $(B_d(p_i,r_{p_i}))_{i=1}^N$ cover $K$. 
Fix arbitrary  $i \in \{1,\ldots, N\}$ and let $p$ be arbitrary in $B_d(p_i,r_{p_i}).$  Then  
$B_d(p,r_{p_i}) \subseteq B_d(p_i,2r_{p_i}).$  This implies that $B_d(p,r_{p_i})$ is, after scaling $r_{p_i}$ to one, $\delta_1$-close in Gromov-Hausdorff distance to $\B(0,1).$   
 Thus given $\delta_2>0$, by volume convergence under a uniform lower Ricci curvature bound on smooth spaces  (\cite[Theorem 5.9] {Cheeger-Colding}: See also \cite[Theorem 9.31]{Cheeger_notes}),    and the fact that $(N^n_i, g_i, p_i)$ GH-converges to $(X,d,x_0)$ as $i\to \infty$
  we can assume that $\delta_1= \delta_1(\delta_2,n) >0$ is chosen sufficiently small so  that 
\begin{equation*}
r_{p_i}^{-n}\mathcal{H}^n_d(B_d(p,r_{p_i})) \geq (1-\delta_2) \vol_{\delta}(\B(0,1))\, .
\end{equation*}

Now given $\delta_3>0$ we can apply Bishop-Gromov volume comparison and choose $\delta_2= \delta_2(\delta_3,n)>0$ sufficiently small so  that
\begin{equation*}
r^{-n}\mathcal{H}^n_d(B_d(p,r))\geq (1-\delta_3) \vol_{\delta}(\B(0,1)),\, 
\end{equation*}
for all $0<r\leq r_{p_i},$ for arbitrary $p \in B_d(p_i,r_{p_i}).$   Given $\varepsilon>0$ we can thus apply \cite[Theorem 5.9] {Cheeger-Colding} to the smooth approximating spaces  followed by   \cite[Theorem 4.85]{Cheeger-Colding-II}  (see also \cite[Theorem 9.45]{Cheeger_notes}) to see that for $\delta_3(\varepsilon,n)>0$ sufficiently small we have that for all $0<r\leq r_{p_i}$ that $B_d(p,r)$, after scaling $r$ to one, is $\varepsilon$-close in Gromov-Hausdorff distance to $\B(0,1)$ for arbitrary $p    \in B_d(p_i,r_{p_i}).$ 

Now we choose 
\begin{equation*}
r_\varepsilon \coloneqq  \min_{i \in \{1,\ldots, N\}} r_{p_i}\ ,
\end{equation*}
to get the desired statement.
\end{proof}
The following Lemma shows a uniform Pseudolocality type result, in a Reifenberg  setting.
\begin{lemma} \label{lem:curv_decay}
Assume $(M^n, g(t))$ for $t \in (0,T)$ satisfy  \eqref{CurvatureCond}, \eqref{DistCond}, \eqref{ReifenbergCond} and \eqref{eq:curv_decay.1} for a fixed $x_0 \in M$.  Then there exists $\ep:(0,T) \to (0,1)  $   such that $ \ep(t) \to 0 $ as $t\to 0$ and 
\begin{equation*}
 |\Rm(g(t))| \leq \ep(t)^2/t\ \  \text{on} \ \  B_{d_0}(x_0,150R)
 \ \text{ for all } \  t \in (0,T)\, . 
 \end{equation*}
  Furthermore,
\begin{equation*}
 |\nabla^{g(t),k}\Rm(g(t))|^2 \leq C(k,n,\ep(t)) /t^{2+k}\ \  \text{on} \ \  B_{d_0}(x_0,100R)
 \ \text{ for all } \  t \in (0,T),
 \end{equation*}
 where $C(k,n,\ep(t)) \to 0$ as $\ep(t) \to 0$.
 
\end{lemma}
\begin{proof}
 To illustrate the main idea, we present first a proof in the case that there is a complete solution $(N,h(t))_{t\in (0,T)}$ with bounded curvature 
 and $\Rc(h(t)) \geq -h(t)$ for all $t\in (0,T)$ and $M \subseteq N$ is open, with $g(t) = h(t)|_M.$ 
Let $p\in M$ be fixed and $\ep >0$ be given. 
By scaling $d_0$ by a   constant $\la$ sufficiently large, we see that
$(U_p,d_{\la} \coloneqq  \la d_0)$ has the property that $d_{GH}(B_{d_{\la}}(x,1),\B(0,1)) \leq \ep^2$  for all $x \in B_{d_0}(x_0,199R)$, in view of the uniform Reifenberg property, \eqref{UniformReifenbergCond} which holds in view of \eqref{ReifenbergCond}
and Lemma \ref{UniformReifenbergLemma}.  
Hence, using  \eqref{DistCond}, we see that $$  d_{GH}(B_{d(g_{\la}(t) ) }(x,1),\B(0,1)) \leq \ep,$$ for all $x \in B_{d_0}(x_0,199R)$
for all $t\leq \ep^2,$ where $g_{\la}(t) = \la^2 g\left(\frac{t}{\la^2}\right).$ 

Furthermore, $\Rc(g_{\la}(t)) \geq -\ep g_{\la}(t)$  on $U_p$ if $\la$ is large enough.  
Using the volume convergence theorem of Cheeger-Colding for spaces with curvature bounded below, we see that
$\vol_{g_{\la}(t)}(B_{g_{\la}(t)}(x,1)) \geq \omega_n(1-\psi(\ep))$ for all $x \in B_{d_0}(x_0,199R)$
for all $t\leq \ep^2,$ and \cite[Corollary 1.3]{Cav-Mon},  then shows us that
 the isoperimetric profile of $B_{g_{\la}(t)} (x,1)$ is $\psi(\ep,n)$ close to the Euclidean isoperimetric profile of a ball of radius one for all $x \in B_{d_0}(x_0,199R)$, for all $0<t \leq \ep^2,$ where $\psi(\ep,n)$ is a constant depending only on $n$ and $\ep$ and  $\psi(\ep,n) \to 0$ as $\ep \downto 0$ for any fixed $n\in \N.$

Thus Perelman's Pseudolocality theorem \cite[Theorem 10.1]{P1} implies that given $\ep>0$ we can choose $\lambda $ sufficiently large such that
\begin{equation*}
 |\Rm(g_{\lambda}(1))|(x)\leq \ep^2, 
\end{equation*}
for all $x \in B_{d_0}(x_0,150R)$. Scaling back implies the first statement of the lemma. The second statement follows as in \cite[Lemma 3.1]{Der-Sch-Sim}. 

 In the general case,  $|\Rm(g(t))|\leq \frac{\ep(t)}{t}$  on $B_{d_0}(x_0,150R)$ follows from 
 Lemma \ref{pseudorefienberg},  and the  second statement follows as in \cite[Lemma 3.1]{Der-Sch-Sim}. 
\end{proof}

\begin{lemma} \label{lem:dist_est}
Assuming the setting of Lemma \ref{lem:curv_decay}, one has for all $t\in [r,T)$ and $0\leq r<s <T$,
\begin{equation*}
e^{r-s}d_r  \geq d_t  \geq d_r - \ep(t) \sqrt{t-r}, \quad \text{ on }  B_{d_0}(x_0,100R),
\end{equation*}
where $\ep(t) \to 0$ as $t\downto 0$.
\end{lemma}
\begin{proof}
This follows from the previous lemma together with the scaling of the distance estimates as in Lemma \ref{SiTopThm1}. 
\end{proof}

\section{Existence of an adjustment map}\label{sec-exi-adj-map}

 The following lemma shows the existence of a so called adjustment map.  If two Riemannian  metrics, and their derivatives  on a ball in Euclidean space  are uniformly bounded with respect to the standard metric,  and the distances are close to one another, up to an isometry, 
then there is    a diffeomorphism, a so called  {\it adjustment } of the isometry, such that
the pull  back of the Riemannian metric with respect to this adjustment map is close in the $C^0$ sense to the other.

\begin{lemma}[cf.~\cite{Gro-Gree-Boo} and \cite{Fuk-Coll-I}]\label{lem:adj-map} 
Let  $c_0\geq 1$ and an integer $n\geq 1$ be given. Then  for all $\ep>0$, there exists 
a constant $c(\ep,n,c_0)>0$ with the property that $c(\ep,n,c_0) \to 0$ as $\ep \downto 0$ such that  
the following holds. 
 Let $K>1$ and let $(M_i^n,g_i)$, $i=1,2$, be two connected Riemannian manifolds (not necessarily complete), with $M_i \subseteq \R^n$ (where $\R^n$ is given the standard topology and differentiable structure)   such that 
\begin{equation}
\begin{split}
 \frac{1}{c_0} \de &\leq g_i \leq c_0 \de,\quad i=1,2,\\
   |\D^k(g_1)|^2 + |\D^k(g_2)|^2&\leq   c_0 K^{k},\quad \forall\, k \in \{1,2,\ldots,8\},
\label{ass-bded-geo}
\end{split}
\end{equation}
where here $\D$ refers to Euclidean derivatives.
Let $d(g_1)$ be the metric on $M_1$ induced by $g_1$ and 
$d(g_2)$ be the metric on $M_2$ induced by $g_2.$  
Assume furthermore that 
$\varphi:B_{d(g_1)}(0,s) \Subset M_1 \to  \varphi(  B_{d(g_1)}(0,s)) \Subset M_2 $
for  $s \geq  100 K^{-\frac 1 2}$
is  a  homeomorphism which satisfies  
  \begin{equation}
  |\phi_*d(g_1)-d(g_2)|\leq \ep K^{-\frac{1}{2}}. 
  \label{dist-close}\\
  \end{equation}

  Then there exists a smooth bi-Lipschitz diffeomorphism 
  \begin{equation*}
  \tilde{\varphi}:B_{d(g_1)} \big(0, \tfrac{s}{10c_0}\big) \to
   \tilde{\varphi}\big( B_{d(g_1)} \big(0, \tfrac{s}{10c_0}\big)\big) \Subset \B(0,R),
   \end{equation*}
    such that
  \begin{equation}
  \begin{split}
  (1-c(n,\varepsilon,c_0)) g_2\leq \tilde{\varphi}_{\ast }g_1&\leq (1+c(n,\varepsilon,c_0)) g_2, \\
  |d(\ti \varphi_{\ast}g_1)(x,y)-d(g_2)(x,y)|&\leq c(n,\varepsilon,c_0) K^{-\frac{1}{2}}, \\
 d(g_2)(\ti \phi(z),\phi(z)) &\leq c(n,\varepsilon,c_0) K^{-\frac{1}{2}},
 \label{metric-close}
 \end{split}
 \end{equation}
 for all $x,y \in  \tilde{\varphi}( B_{d(g_1)} (0, \frac{s}{10c_0})),$ and all $z\in B_{d(g_1)} (0, \frac{s}{10c_0})$.
 
\end{lemma}

A proof of Lemma \ref{lem:adj-map} can be adapted either from Gromov's book \cite[Section D, Chapter 8]{Gro-Gree-Boo} or Fukaya \cite{Fuk-Coll-I}: there the setting is more general. We provide a somewhat alternative proof in our  setting  based on the notion of almost isometries.

Before we prove Lemma \ref{lem:adj-map}, we define the notion of  {\it  almost isometry} which shall be used in this paper. 

\begin{defn} 
Let $(W,d)$ be a metric space. We call  $Z:(W,d) \to \R^n$ an $\ep_0$ almost isometry if
\begin{equation*} 
(1-\varepsilon_0)d(x,y)  -\ep_0  \leq  |Z(x) -Z(y)| \leq (1+\varepsilon_0)d(x,y) +\ep_0
\end{equation*}
for all $x,y \in W$.
\end{defn}

We state the following useful observation taken from \cite[Lemma 3.5]{Der-Sch-Sim}:

\begin{lemma}\label{almostisomlem}
For all $\si>0$, there exists $\gamma=\ga(\si)\in(0, \si)$ small with the following property: 
if $ L: (\B(0,{\ga}^{-1} ),d(\delta))  \to \R^n$ is a  $\ga$ almost isometry fixing $0$, where $d(\delta)$ denotes the distance induced by the Euclidean metric $\delta$,  then  
there exists an  $S \in O(n)$ such that
$|L-S|_{L^{\infty}(\B(0,\si^{-1})) } \leq \si$. 
\end{lemma}

\begin{proof}[Proof of Lemma \ref{lem:adj-map}]
Let $\al>0$  be given. Without loss of generality, $\frac{1}{\al}\geq c_0$ and $\al \leq 10^{-10}$. Assume $\ep \leq \al^2$. 
Rescaling  $\hat g_1 =  \frac{K}{\al} g_1$ and $\hat g_2 =  \frac{K}{\al} g_2$ and $h = \frac{K}{\al} \de$ we are now in the setting (denoting  $\hat g_1$ by $g_1$  once again and   $\hat g_2$ by $g_2$ once again) that 
\begin{eqnarray}\label{ass-bded-geo.1}
 && \frac{1}{c_0} h  \leq  g_i \leq c_0 h,\nonumber \\
  && |\D^k(g_1)|^2 + |\D^k(g_2)|^2 \leq    \al^k,\quad  \forall \, k \in  \{1,2,\ldots,8\},\\
 &&  |\phi_* d(g_1)-d(g_2)| \leq  \ep \al^{-\frac 1 2}  \leq \al,  \nonumber
  \end{eqnarray}
 on $B_{d(g_1)}(0,s \al^{-\frac 1 2})$
where here now $\D$ refers to the covariant derivative  with respect to $h,$   and $| \cdot |$ is the norm with respect to $h$.  
There is an isometry $Z_i:(M_i,h) \to (K^\frac{1}{2}\al^{-\frac{1}{2} }M_i,\de)$ given by $x \to  \sqrt{\frac{K}{\al}}x.$
Pushing everything forward   by $Z_i$, we see that we may assume that
we are in the setting above, with $h=\de$ and $\D$ is the usual derivative in $\R^n$   and $|\cdot|$ is the norm with respect to $\de$.  
Taking any $p\in B_{d(g_1)}\big(0,\frac{s }{2}\al^{-\frac 1 2}\big),$
let $a_{ij} = (g_1)_{ij}(p)$ be the metric at $p$. From  the estimates 
\eqref{ass-bded-geo.1} we see that 
\begin{equation*}
\sum_{k=0}^6 |\D^k((g_1)_{ij}(\cdot)- a_{ij})|_{C^0(\B(p, \ga^{-1}))} \leq \ga,
\end{equation*}
 where $\ga = \ga(\al,n,c_0) \to 0 $
as $\al \to 0$.  Constants of this type  shall  be denoted by $\ga$  and can change from line to line, but always satisfy
$\ga(\al,n,c_0) \to 0$ as $\al \to 0$.  
We can choose  an affine map  $T_1:\R^n \to \R^n$,   with $T_1(p) = 0$, such that 
in these  coordinates,   that is for  $\ti g_1:= (T_1)_* g_1$,  we have, $(\ti g_1)_{ij}(0) =\de_{ij}$ and 
\begin{equation*}
\sum_{k=0}^6 |\D^k( (\ti g_1)_{ij}(\cdot)- \de_{ij})|_{C^0(\B(0, \ga^{-1}))} \leq \ga,
\end{equation*}
 for some $\ga=   \ga(\al,n,c_0) \to 0$. 
Similarly for $g_2$ at $\phi(p)$ we find an affine map $T_2,$ such that for $\ti g_2:= (T_2)_* g_2,$ we have  $(\ti g_2)_{ij}(0) = \de_{ij}$ and  
\begin{equation*}
\sum_{k=0}^6 |\D^k( (\ti g_2)_{ij}(\cdot)- \de_{ij})|_{C^0(\B(0, \ga^{-1}))} \leq \ga,
\end{equation*}
 for some $ \ga=  \ga(\al,n,c_0) \to 0.$
 
Defining $ L\coloneqq     T_2\,\of\, \phi \,\of\, (T_1)^{-1}$, we see that $ L : (\B(0,\ga^{-1}),d(\ti g_1)) \rightarrow \R^n$ is a $\ga$ almost isometry fixing $0$. Therefore, Lemma \ref{almostisomlem} ensures there must exist an $\ti S\in O(n)$ such that
$|\ti S- L|_{C^0(\B(0, \ga^{-1}))  }\leq   \ga(\al,n,c_0) $ where $\ga \to 0$ as $\al \to 0$.
In particular, setting  $S_p\coloneqq  (T_2)^{-1} \,\of\,  \ti S \,\of\, T_1,$ and using $(T_2)^{-1} \,\of\,  L\, \of\, T_1   = \phi,$ we obtain 
\begin{equation*}
  |  \phi - S_p|_{C^0(\B(p,\ga^{-1}))}  \leq  \ga(\al,n,c_0) \to 0,
\end{equation*}      
as $\al \to 0,$
and for $\ti \de:= (T_2)_*\de$,  (note $\frac{1}{C(c_0,n)} \de \leq \ti \de  \leq C(c_0,n)\de $ ),
\begin{equation*}
\begin{split}
  |   (S_p)_*{g_1} -g_2|_{C^0(\B(\phi(p), 200),\de))}  
&=  |\ti S_*{\ti g_1} - \ti g_2|_{C^0(T_2(\B(\phi(p), 200)), \ti \de))} \\
   &\leq  C(c_0,n)  |\ti S_*(\ti g_1 -\de) + \ti S_*{\de} -\de  -  (\ti g_2 -\de)|_{C^0(\B(0, R(c_0,n)))}\\
& =   C(c_0,n) | \ti S_*(\ti g_1 -\de)   -  (\ti g_2 -\de)|_{C^0(\B(0, R(c_0,n)))}\\
& \leq C(c_0,n) \left(| \ti S_*(\ti g_1 -\de)|_{C^0(\B(0, R(c_0,n)))}   +  |\ti g_2 -\de|_{C^0(\B(0, R(c_0,n)))}\right)\\
& = C(c_0,n) \left(| \ti g_1 - \de|_{C^0(\B(0, R(c_0,n)))}   +  |\ti g_2 -\de|_{C^0(\B(0, R(c_0,n)))}\right)\\
& \leq \ga(\al,n,c_0)\to 0\quad\text{ as $\al \to 0,$ } \label{metricstrong}
\end{split}
\end{equation*}
where we used that fact that $\ti S_* \de = \de$ in the third line.

In particular, 
\begin{equation} 
  (1-c(\al,n,c_0)) (S_p)_*{g_1} \leq g_2 \leq (1+c(\al,n,c_0))(S_p)_*{g_1},\quad\text{on $\B(\phi(p),100)$}, \label{metricnec}
\end{equation} 
with $c(\al,n,c_0) \to 0$ as $\al \to 0$.  

From the definition of $S_p$, we have 
$\D S_p = \D (T_2)^{-1}\, \of\, \D \ti S\, \of \,\D T_1 = V_2^{-1} \,\of \,\ti S\, \of\, V_1 $ and 
 \begin{equation*}
|V_1|_{C^0(\R^n, \R^n)} + |(V_1)^{-1}|_{C^0(\R^n, \R^n)}+|V_2|_{C^0(\R^n, \R^n)} + |(V_2)^{-1}|_{C^0(\R^n, \R^n)} \leq C(n,c_0),
\end{equation*}
so that
$|\D  S_{p}|_{C^0(\R^n)}  \leq c(c_0,n).$
We are ready to define our adjustment map.\\

We take  a covering by balls $(\B(p_i, 1))_{i\in I} $ of radius one of  $B_{d(g_1)}\big(0,s \al^{-\frac 1 2}\big)$ for which
$(\B(p_i,\frac{1}{3}))_{i\in I} $ are pairwise disjoint and $(\B(p_i,\frac 2 3))_{i\in I}$ still covers,  and take a partition of unity  $(\eta_i:M_1 \to \R)_{i\in I}$ subordinate to this covering:
$\eta_i= 1$ on the balls of radius $\B(p_i, \frac 1 3)$,  $\text{supp}(\eta_i)\subset \B(p_i, \frac 2 3)$.
We let $S_i\coloneqq  S_{p_i}$  and define 
\begin{equation}
\ti \phi\coloneqq  \sum_{i\in I} \eta_i S_i. \label{definitionoftiphi}
\end{equation}
As we saw above,   $|  \phi    - S_{i} |_{C^0(\B(p_i,\ga^{-1}))}  \leq  \ga\to 0.$ 
Furthermore, 
\begin{equation*}
|S_i -S_j|_{C^0(\B(p, 100))}  \leq |S_i -\phi|_{C^0(\B(p, 100))}  + |S_j - \phi|_{C^0(\B(p, 100))}  \leq \ga(\al,c_0,n) \to 0,
\end{equation*}
if $p_i, p_j \in \B(p, 100)$ per construction.
Hence 
\begin{equation*}
  | \D   S_j - \D    S_i|^2_{C^0(\B(p, 100))}  
  = |S_j -S_i|^2_{C^0(\B(p, 100))}  \leq  \gamma(\al,c_0,n) \to 0,
  \end{equation*}
   if $p_i, p_j \in \B(p, 100).$ 
Let $x \in \B(p_i, 10),$ and let $J_i$ denote the indices $j\in \N$ for which
$\B(p_j, 1) \cap \B(p_i, 10) \neq \emptyset$. Note that we can assume that there exists a uniform constant $c(n)$ such that $|J_i|\leq c(n)$. Furthermore,  
$\ti \phi|_{\B(p_i, 10)}  = \sum_{j\in J_i} \eta_j S_j$ and hence
\begin{equation*}
\begin{split}
|\D   \ti \phi - \D   S_i|_{C^0(\B(p_i, 10))} & = \Big|\sum_{j\in J_i} \D    (\eta_j \cdot S_j ) - \D  S_i\Big|_{C^0(\B(p_i, 10))}\cr
& =  \Big|\sum_{j\in J_i} \D   ( \eta_j \cdot S_j ) - \D  \Big(\sum_{j\in J_i}\eta_j\Big)S_i-\sum_{j\in J_i}(\eta_j\cdot \D S_i)\Big|_{C^0(\B(p_i, 10))}\\
& =  \Big|\sum_{j\in J_i} (\D    \eta_j) (S_j -S_i)  - \eta_j (\D  S_i-\D  S_j)\Big|_{C^0(\B(p_i, 10))}\\
& \leq \gamma(\al,n,c_0) \to 0,\quad\text{as $\al \to 0.$ }
\end{split}
\end{equation*}
That is
\begin{equation}
|\D   \ti \phi -   \D S_i|_{C^0(\B(p_i, 10))} \leq \gamma(\al,n,c_0) \to 0,\quad\text{as $\al \to 0.$ }
\label{phiSeq}
\end{equation}

Using the fact that $\ti \phi|_{\B(p_i, 10)}  = \sum_{j\in J_i} \eta_j S_j,$  we also obtain
\begin{equation*}
|\D^2(\ti \phi-S_i)|_{C^0(\B(p_i, 10))} +  |\D^3(\ti \phi-S_i)|_{C^0(\B(p_i, 10))} \leq C(n,c_0),
\end{equation*}
and hence, using   standard interpolation inequalities, see for example \cite[Appendix A,  Lemma A.5 and A.6.]{Sch-Sim-Sim-Hyp}, 
we see that 
\begin{equation*}
|\D^2(\ti \phi-S_i)|^2_{C^0(\B(p_i, 10))} \leq |\D(\ti \phi-S_i)|_{C^0(\B(p_i, 10))} |\D^3(\ti \phi-S_i)|_{C^0(\B(p_i, 10))}  \leq \gamma(\al,n,c_0).
\end{equation*}
We perform a Taylor expansion in each component and obtain for $x,y \in \B(p_i, 10)$ and $k=1,..., n$,
\begin{equation*}
 (\ti \phi -S_i)^k(x) =  (\ti \phi -S_i)^k(y)   + \D_{\al}(\ti \phi -S_i)^k(x) (x-y)^{\al} + C^k(x,y),
 \end{equation*}
where $ |C^k(x,y)|   \leq \ga(\al,n,c_0)|x-y|^2.$ Hence 
\begin{equation*}
\begin{split}
  |\ti \phi(x) -\ti \phi(y)| &\geq  |S_i(x) - S_i(y)| -\ga(\al,c,c_0)|x-y|\cr
 & =  |(T_{2,i})^{-1}\,\of\, \ti S_i\, \of\, T_{1,i}(x) - (T_{2,i})^{-1}\,\of\, \ti S_i \,\of\, T_{1,i}(y)| 
 -\ga(\al,c,c_0)|x-y|\cr
 & =  | (T_{2,i})^{-1}\,\of\, \ti S_i\, \of\, T_{1,i}(x-y)| -\ga(\al,c,c_0)|x-y|\cr
 & \geq \be(n,c_0)|x-y|,
 \end{split}
 \end{equation*}
  for $x,y \in \B(p_i, 10),$ 
where $\be(n,c_0) >0.$

In view of \eqref{metricnec} and \eqref{phiSeq} we have
\begin{equation*} 
\label{metricnec2}
  (1-c(\al,n,c_0)) (\ti \phi)_*{g_1} \leq g_2 \leq (1+c(\al,n,c_0))(\ti \phi)_*{g_1},  \quad \text{ on $\B(p_i, 1)$,}
\end{equation*} 
with $c(\al,n,c_0) \to 0$ as $\al \to 0$,
and hence on all of $ B_{d(g_1)}\big(0, s \al^{-\frac 1 2}\big).$
Combining this with the fact that $\ti \phi$ is a diffeomorphism on Euclidean balls of radius $10$,
we see 
\begin{equation*}
 (1-c(\al,n,c_0)) d(g_1)( x, y)\leq  d(g_2)(\tilde \varphi(x), \tilde{\varphi}(y))   \leq (1+c((\al,n,c_0)) d(g_1)( x, y),
\end{equation*}
on any ball of radius $10$. On the other hand, for $x \in \B(p_i, 1)$ and using the notation above, we know 
\begin{equation*}
\begin{split} 
\left|\ti \phi(x) - \phi(x)\right|& = \Big|\sum_{j\in J_i}  \eta_j S_j(x) - \phi(x)\Big| = \Big|\sum_{j\in J_i}  \eta_j ( S_j -\phi)(x)\Big|\leq c(\al,n,\de),
\end{split}
\end{equation*}
that is
\begin{equation*} 
|\ti \phi(x) - \phi(x)| \leq  c(\al,n,\de) \to 0,\quad\text{ for
$\al \to 0$.}
\end{equation*} 
This implies for points $x,y$ with $|x-y|\geq 10,$ 
\begin{equation*}
\begin{split}
  d(g_2)(\tilde \varphi(x), \tilde{\varphi}(y)) &   \leq d(g_2)( \varphi(x), \varphi(y) )+  d(g_2)( \varphi(x), \ti \varphi(x))+  d(g_2)( \varphi(y), \ti \varphi(y))\\   
 & \leq 
   d(g_2)( \varphi(x), \varphi(y) )
   + c(\al,n,\de) \\
   & \leq (1+c(\al,n,\de)) d(g_2)( \varphi(x), \varphi(y) )\\
   & \leq (1+c(\al,n,\de))^2 d(g_1)(x,y),
   \end{split}
   \end{equation*}
   and similarly,
   \begin{equation*}
  d(g_2)(\tilde \varphi(x), \tilde{\varphi}(y) )
  \geq (1-c(\al,n,\de))^2 d(g_1)(x,y).
  \end{equation*}
  
Hence, the map $\ti \phi$  satisfies 
\begin{equation*}
 (1-c(\al,n,c_0)) d(g_1)( x, y)\leq  d(g_2)(\tilde \varphi(x), \tilde{\varphi}(y))   \leq (1+c((\al,n,c_0)) d(g_1)( x, y)
\end{equation*}
on $ B_{d(g_1)}\big(0,\frac{s}{2} \al^{-\frac 1 2}\big)$.
\end{proof}

\section{Rough convergence rate}\label{sec-rou-con-rate}

In this paper we are often interested in the   following setup: $(M^n,g(t))_{t\in (0,T)}$ is a smooth solution to Ricci flow (not necessarily complete) such that
\eqref{CurvatureCond}, \eqref{DistCond}, and \eqref{ReifenbergCond}  
hold, where $d_0 = \lim_{t\downto 0}d(g(t))$ is  locally, uniquely,  well defined on $U_{p}$ for all 
$p \in M$ by  Lemma \ref{SiTopThm1}, as explained in the introduction.
 In particular: for any fixed $x_0 \in M$,  Lemma \ref{lem:curv_decay} 
  implies   that there exists an $\ti R  >0$, $T   \in (0,1]$ such that  
$|\Rm(g(t))| \leq \ep(t)/t\ \  \text{on} \ \  B_{d_0}(x_0,\ti R) \Subset U_p \Subset   M  $
where $0\leq \ep(t) \to 0$ as $t \downto 0$.  
Hence,   we can reduce $T$ if necessary  to  arrive at the setup  
\begin{align} 
& |\Rm(g(t))| \leq \frac{\ep(t)}{t} \ \  \text{on} \ \  B_{d_0}(x_0,\ti R) \Subset U_p \Subset   M  
 \ \text{ for all } \  t \in (0,T),  \tag{${a}$}  \label{a}\\
& \Rc(g(t))\geq -g(t) \ \  \text{on} \ \   B_{d_0}(x_0,\ti R) \Subset U_p \Subset M  \ \text{ for all } \  t \in (0,T),  \tag{${b}$}  \label{b}\\
& |d_{t}-d_0|\leq \ep(t)\sqrt{t} \text{ on }  B_{d_0}(x_0,\ti R )    \ \ \text{ for all }   t\in[0,T)\tag{${c}$} \label{c},
  \end{align}
  where $\ep(t) \to 0$ as $t\downto 0$.
In the following we show that we can construct solutions to  the Dirichlet problem to the Ricci-Harmonic map heat flow  
for a suitable possibly non-smooth  class  of initial data $F_0,$ where the background Ricci flows are of the type considered above. 
  More precisely, we consider initial data $F_0$ and $ \ti R>2 R >0$   satisfying  
\begin{align*}
& F_0: \, \bigg\{
   \begin{array}{rcl}
       B_{d_0}(x_0,\tilde{R})  & \rightarrow &\R^n \\
       x& \rightarrow &  ((F_0)_1(x), \ldots, (F_0)_n(x))
   \end{array} \tag{$\rm d$} \label{d} \\
&  \text{ is a }   (1+\ep_0)  \text{ bi-Lipschitz homeomorphism from $B_{d_0}(x_0,2R)$ onto its image.}    
\end{align*}
We also consider the special case that   $F_0:= D_0$ where $D_0$ are distance coordinates defined on $B_{d_0}(x_0,\ti R),$ which are 
{\bf  $(1+\ep_0)$ bi-Lipschitz}    on   $B_{d_0}(x_0,2R )$ for  a $\ti R >2R >0$  i.e.
\begin{align*}\label{chat}
&   \text{ there are  points } a_1, \ldots, a_n \in B_{d_0}(x_0,\ti R),   \text{  such that 
the map } \cr
 &D_0: \, \bigg\{
   \begin{array}{rcl}
       B_{d_0}(x_0,\tilde{R})  & \rightarrow &\R^n \\
       x& \rightarrow & (d_0(a_1,x)-d_0(a_1,x_0), \ldots, d_0(a_n,x)-d_0(a_n,x_0))
          \end{array} \tag{$\rm \hat{d}$} \\
   & \text{ is a }   (1+\ep_0)  \text{ bi-Lipschitz homeomorphism from $B_{d_0}(x_0,2R)$ onto its image.}  \\[-3ex]
\end{align*}

We recall that we say that two metrics $g,h$ on a set $\Omega \subset M$ are $\ep$-close, for $\ep > 0$, provided
$$ (1+\ep)^{-1}g \leq h \leq (1+\ep)g$$
on $\Omega$.
We notice, for given $\ep_0 \in (0,1)$,  that the condition \eqref{a} implies the condition 
\begin{align}
& |\Rm(g(t))|\leq \frac{\ep_0}{t} \tag{${a_{\ep_0}}$}  \label{aep}\
\end{align}
 for all $t \in (0,T)$ after reducing $T$ if necessary.
 The following existence result  and estimates for the $\de$-Ricci-DeTurck flow can be proved by invoking \cite[Theorem $3.11$]{Der-Sch-Sim} in this setting:  
\begin{thm} \label{RicciDeTurck}
For all $n\in \N,$  $\al_0 \in (0,1),$ there exists an $\ep_0(\al_0,n)>0$ depending only on $n$ and $\al_0$ such that the the following holds. 
Let $(M^n,g(t))_{t\in (0,T]}$, $T\leq 1,$ be  a smooth  solution to Ricci flow satisfying  \eqref{CurvatureCond} and \eqref{DistCond} and let $p \in M,$ and
$d_0$ be  the locally well defined metric coming from Lemma \ref{SiTopThm1}, with $d_0:U_p \times U_p \to \R^+_0.$ Let  $x_0 \in U_p$ and assume that   $(M^n,g(t))_{t\in (0,T]}$ satisfies  \eqref{aep},  and \eqref{c}, where 
$\ti R \geq    R  \geq 200.$ 
 Let $ F_0:B_{d_0}(x_0,R) \to \R^n$ be  a  $(1+\ep_0)$ bi-Lipschitz map with respect to $d_0$  as in $\eqref{d}$,   and  assume that {\bf one  of the following is satisfied}: 
 \begin{enumerate}
  \item \label{item-1st-ass}
  $F_0 = \lim_{i\to \infty} F_i$ uniformly where $F_i: B_{d_0}(x_0,R)\to \R^n$ is a $(1+\ep_0)$ bi-Lipschitz map  with respect to $d_{t_i}$,   for some sequence $t_i>0$ with $t_i\to 0$, {\bf or},\\
\item \label{item-2nd-ass} $(M^n,g(t))_{t\in (0,T]}$ can be smoothly extended to $  (M^n,g(t))_{t\in [0,T]}$,{\bf or},\\
\item \label{item-3rd-ass}$F_0 =D_0$ with $D_0$ as in $\eqref{chat}$. 
\end{enumerate}

Then, for any $m_0 \in  B_{d_0}(x_0,R/2)$,  there exist  an $S= S(n, \al_0)>0$  and  maps  
$$F(t): B_{d_0}(m_0,  R) \to   E_t\coloneqq  F(t)(B_{d_0}(m_0, 3/2)) \subseteq \R^n,$$  
which are solutions to the Ricci-harmonic map flow, i.e.
\begin{equation*}
\partial_t F(x,t)=\Delta_{g(t)}F(x,t) ,\quad \text{for all  $ (x,t) \in B_{d_0}(m_0,3/2)\times(0,\hat S:= \min(S,T/2))$},
\end{equation*}
 with initial condition the map $F_{0}|_{B_{d_0}(m_0,3/2)}$, with $\B(\ti m_0,  1) \subseteq E_{t}$
for $\ti m_0 = F_0(m_0)$, 
which are smooth     bi-Lipschitz  diffeomorphisms onto their images  for all $0< t \leq \hat S,$  satisfying the  following quantitative estimates: \begin{equation}
\begin{split}\label{cauchy-pb-init}
(1-\al_0) d_t (x,y) & \leq |F(x,t) -F(y,t)| \leq (1+\al_0) d_t (x,y),\\
& \ \ \ \ \ \ \ \text{for all  $ (x,t), (y,t)  \in B_{d_0}(m_0,3/2)\times[0, \hat S )$},\\
|F(x,t)-F_0(x)|&\leq c(n)\sqrt{t} ,\quad \text{for all } \quad (x,t) \in B_{d_0}(m_0,3/2)\times[0,  \hat S),
\end{split}
\end{equation}
for some positive constant $c(n)$.
Setting  $\ti g(t) = (F(t))_*(g(t))$ for $t\in (0,\hat S)$,  on  $\B(\ti m_0,1),$   we further have that   $(\ti g(t))_{t \in (0,  \hat S)}$   is a smooth family of metrics which solve the $\de$-Ricci-DeTurck flow and are $\alpha_0$-close to the $\delta$ metric:
\begin{equation} \label{metricin}
  (1+\alpha_0)^{-1} \delta \leq \ti g(t) \leq    (1+\alpha_0) \delta\, ,
 \end{equation}
 satisfying
 \begin{equation}
  |\D^k \ti g(t)|^2 \leq \frac{c(k,n)}{t^k}, \quad k\geq 0, \label{metricin2}
 \end{equation}
 for all   $t \in (0, \hat S),$ where $c(n,k)$ are constants depending only on $k$ and $n$. 
 The metric $\ti d_t: = d(\ti g(t))$ satisfies, 
 $ \ti d_t \to \ti d_0\coloneqq  (F_0)_*d_0$ uniformly on $  \B(\ti m_0, 1)$ as $t \to 0.$ 
\end{thm} 
\begin{proof}
Constants depending on $n$ shall be denoted by $c(n)$ and can change from line to line within the proof.   \\

 \noindent {\it  Proof of  the results  assuming hypothesis \eqref{item-1st-ass}:} \\
 
 The map $F_i:B_{d_{0}}(x_0,R)\rightarrow\mathbb{R}^n$  is $(1+\ep_0)$ bi-Lipschitz with respect to $g(t_i)$ implies that $|\grad^{g(t_i)} F_i|_{g(t_i)} \leq c(n)$. 
 Let   $ Z_{i}: B_{d_0}(m_0,20)\times[t_i,  \min(S(n,c_0),T)) \to \R^n$ be the solution to the Ricci-harmonic map flow $\partial_t Z_i(x,t) = \lap_{ g(t)} Z_i(x,t)$ for $(x,t) \in B_{d_0}(m_0,20)\times (t_i,\min(S(n,c_0),T)]$ with initial and boundary data given by $Z_{i}(t_i) = F_i$ and $Z_i|_{\boundary B_{d_0}(m_0,20)} = F_i|_{\boundary B_{d_0}(m_0,20)}$ coming from \cite[Theorem 2.1]{Der-Sch-Sim}.
The solution satisfies 
$|\grad^{ g(t)} Z_i(t)|\leq c(n)$ for all $t\in [t_i,\min(S(n,c_0),T)]$ 
on $B_{d_0}(x_0,20)$ as shown in
 \cite[Theorem 2.1]{Der-Sch-Sim}.
Hence, all the conditions required to apply \cite[Theorem 3.8]{Der-Sch-Sim}
are satisfied, and so the stated consequences  there (except for the inequalities \eqref{metricin2}) hold: for example 
\begin{equation}
\begin{split}\label{int-est-Z-maps}
  (1-\al_0) d_{t}(x,y) &\leq |Z_i(x,t) - Z_i(y,t)|  \leq  (1+\al_0) d_{t}(x,y), \\
 &|{\grad}^2 Z_i(x,t)| \leq \frac{c(n)}{\sqrt{t}},   \\
 &|Z_i(x,t)-Z_{i}(x,t_i)|\leq c(n)\sqrt{t},
 \end{split}
 \end{equation}
 for all $t\in (2t_i,T/2).$ Taking a limit $i \to \infty$ we obtain a smooth limit
 $F(x,t):= \lim_{i\to \infty} Z_{i}(x,t),$  $F: B_{d_0}(x_0,10)\times (0,T] \to \R^n$
being a  smooth solution to harmonic map heat flow,
 $\partial_t F(x,t) = \lap_{g(t)} F(x,t)$ for $(x,t) \in  B_{d_0}(x_0,10)\times (0,T]$
 satisfying
 all the stated inequalities 
except \eqref{metricin2}.  The inequalities \eqref{metricin2} follow from \cite[Lemma 4.2]{MilesC0paper}. \\

\noindent {\it  Proof of  the results  assuming hypothesis \eqref{item-2nd-ass}:} \\

  Since $(g(t))_{t\in [0,T]}$ is smooth, we know $ (1-\ep_0)g(s) \leq  g(0) \leq (1+\ep_0)g(s)$  for all
$s \in [0,\si)$  for $\si>0$ sufficiently small.
Hence, since $F_0$ is $(1+\ep_0)$ bi-Lipschitz with respect to $g(0),$
we must have $F_i:= F_0$ is $(1+2\ep_0)$ bi-Lipschitz with respect to $g(t_i)$ for a sequence 
$t_i \downto 0$.
Hence we may apply (i).\\

\noindent {\it  Proof of  the results  assuming hypothesis \eqref{item-3rd-ass}:} \\

 The results in this  setting are obtained in \cite[Theorem  3.11]{Der-Sch-Sim} and the proof thereof: 
the $\alpha_0$-closeness of $\ti g(t):= F(t)_* g(t)$ to $\delta$ follow from \cite[Lemma 4.2]{MilesC0paper} 
as explained in the proof of \cite[Theorem
3.11]{Der-Sch-Sim}. 
The smoothness of the solution $F$ (for $t>0$)  to the harmonic map heat flow follows from the
smoothness of $\ti g$ and $g$ (for $t>0$) and the fact that $F(t):B_{d_0}(m_0,3/2) \to F(t)(B_{d_0}(m_0,3/2))$ is a  $C^1$ diffeomorphism: see  \cite[Theorem 3.11]{Der-Sch-Sim} and the proof thereof. 
 
\end{proof}

The main result of this section is   Proposition \ref{prop-dec-2-sol-HMF}, which is used in   the proof of the main theorem, Theorem \ref{thm:main.1}.

 \begin{prop}\label{prop-dec-2-sol-HMF}
 For $n\in \N$, $\al_0 \in \R$ let $\ep_0(n,\al_0)$ be the constant from     Theorem \ref{RicciDeTurck}. 
 We assume  the following setting in this proposition, and the corollary, Corollary \ref{cor-mod-DTRF} that follows: 
 \begin{enumerate}
 \item\label{S} 
Assume  $(M_1^n,g_1(t))_{t\in (0,T]}$ and $(M_2^n,g_2(t))_{t\in (0,T]}$, $T\leq 1$,  are  two solutions to Ricci flow satisfying \eqref{a}, \eqref{b}, \eqref{c},  \eqref{CurvatureCond} and \eqref{DistCond}  for an   $\ti R \geq R \geq 200$. Let  $p\in M_1$ and $\hat p\in M_2$ and let $d_{0,1}:U_{p,1} \times U_{p,1} \to [0,\infty)$ and $d_{0,2}:U_{\hat p,2} \times U_{\hat p,2} \to [0,\infty)$   are  the locally well defined metrics at time zero     on  $M_1$, respectively on  $M_2$   given by  Lemma  
 \ref{SiTopThm1}.  Let  $x_0 \in U_{p,1}$, $\hat x_0 \in U_{\hat p,2}$,   
 and  assume that there is  an isometry $\psi_0: U_{p,1} \to \psi_0(U_{p,1}) \subseteq U_{\hat p,2}$, $\psi(x_0) = \hat x_0.$ \\
 \item
  Assume that both solutions satisfy \eqref{aep},  and    that  $  D_{0}: B_{d_{0,1}}(x_0,4R) \to \R^n$ 
satisfies   \eqref{chat} for $d_0= d_{0,1}$ and the $\ep_0$ specified at the beginning of the statement of this theorem.     
 Let 
  $F_1:B_{d_{0,1}}(x_0,\frac 3 2)\times(0,\hat S)\rightarrow \R^n$, $F_2:B_{d_{0,2}}(\hat x_0,\frac 3 2)\times(0,\hat S)\rightarrow \R^n$,
  be the two smooth, $(1 + \al_0)$ bi-Lipschitz solutions to the Ricci-harmonic map heat flow provided by [(iii), Theorem \ref{RicciDeTurck}] with initial value the map $D_0$, respectively $ D_0\, \of\,  (\psi_0)^{-1}.$ 
   Denote the corresponding solutions to the $\de$-Ricci-DeTurck flow by $\tilde{g}_i(t)$, $i=1,2$. Let $(t_k)_k$ be any sequence of positive times with $t_k \to 0$ as $k\to \infty$.
    \end{enumerate}
   Then there exists a family of diffeomorphisms $\tilde{\varphi}_k$ defined on $B_{d(\ti g_1(t_k))}\left(E_{t_k}^1,\sqrt{t_k}\right)$ where $E_{t_k}^1\coloneqq F_1(t_k)\left(B_{d_0}(x_0,3/2)\right)$ such that 
\begin{align}
 (1-\varepsilon(t_k))\tilde{g}_2(t_k)&\leq(\tilde{\varphi}_k)_{\ast}(\tilde{g}_1(t_k))\leq (1+\varepsilon(t_k))\tilde{g}_2(t_k),\label{metricin-adj-map}\\
 |{\ti \phi}_k - \operatorname{Id}|  &\leq c(n)\sqrt{t_k},  \label{phiesty}
\end{align}
where $\varepsilon(t_k) \to 0$ as $t_k \to 0$.
\end{prop}

\begin{proof}
Let $(t_k)_k$ be any sequence of positive times with $t_k \to 0$ as $k\to \infty$ that we fix once and for all. Let $\varphi_k\coloneqq F_2(t_k)\,\of\,    \psi_0 \,\of\, (F_1(t_k))^{-1}$ defined on $E_k^1\coloneqq F_1(t_k)\left(B_{d_0}\left(x_0,\frac{3}{2}\right)\right)$. 
We first note that if $q = F_1(t_k)^{-1}(p)$,
\begin{equation}
\begin{split}
   d(\ti g_2(t_k))(  \phi_k(p), \operatorname{Id}(p))&=
  d(\ti g_2(t_k))(F_2(t_k) \circ    \psi_0(q) , F_1(t_k)(q))\\
&\leq  d(\ti g_2(t_k))(F_2(t_k) \circ    \psi_0(q), D_0(q))\\
&\quad + d(\ti g_2(t_k))(F_1(t_k)(q), D_0(q))\\
&\leq 2|F_2(t_k) \circ    \psi_0(q) -D_0(q)| + 2|F_1(t_k)(q) -D_0(q)| \\
& \leq c\sqrt{t_k} + 2 |F_2(0) \circ \psi_0(q) -D_0(q)|\\
&\quad + 2|F_1(0)(q) -D_0(q)| + c\sqrt{t_k}  \\
& =   c\sqrt{t_k} \to 0 \text{ as } t_k \to 0\, , \label{firstc0}
\end{split}
\end{equation}
  which yields 
$$ |\phi_k(\cdot) - \operatorname{Id}|_{C^0(\B(0,1))} \leq c \sqrt{t_k}\, .$$ 
Define $\ti g_1^k\coloneqq \ti g_1(t_k), $   and $\ti g_2^k\coloneqq \ti g_2(t_k)$, $k\geq 0$. Then, the sequences $(\ti g_1^k)_{k\geq 0}$ and $(\ti g_2^k)_{k\geq 0}$  satisfy  the estimates
\eqref{metricin} and \eqref{metricin2} with $t=t_k$. Also,  the distance distortion estimates \eqref{c} imply:
\begin{equation*}
\begin{split}
&\left| (\phi_k)_*d(\ti g_1^k) (x,y)-d(\ti g_2^k)(x,y)\right|=\left| (F_2(t_k)\, \of\, \psi_0)_* d(g_1(t_k))(x,y)-d(\ti g_2^k)(x,y)\right|\\
&=| d(g_1(t_k))(  (F_2(t_k) \,\of\, \psi_0)^{-1}(x),(F_2(t_k) \,\of\, \psi_0)^{-1}(y))\\
&\quad\quad\quad\quad\quad\quad\quad\quad\quad\quad\quad\quad\quad\quad-d(g_2(t_k))((F_2(t_k))^{-1}(x),(F_2(t_k))^{-1}(y))|\\
& \leq  \left| d_{0}((F_2(t_k)\, \of\, \psi_0)^{-1}(x),(F_2(t_k) \,\of\, \psi_0)^{-1}(y))    -  {\psi_0}_{\ast}d_{0}((F_2(t_k))^{-1}(x),(F_2(t_k))^{-1}(y))\right|  \\
&\quad +\ep(t_k)\sqrt{t_k} \\
& =  \left| {\psi_0}_* d_{0}( (F_2(t_k)  )^{-1}(x),(F_2(t_k) )^{-1}(y) )   -  {\psi_0}_{\ast}d_{0}((F_2(t_k))^{-1}(x),(F_2(t_k))^{-1}(y))\right| \\
&\quad  +\ep(t_k)\sqrt{t_k} \\
& =  \ep(t_k)\sqrt{t_k} .
\end{split}
\end{equation*}
These facts let us apply Lemma \ref{lem:adj-map} to the sequences of metrics $\ti g_1^k,$ and $ \ti g_2^k,$   defined on a Euclidean ball of radius $1$, with   $K\coloneqq 100 t_k^{-1}$, $c_0= 100/99$, $s = 1$ and $\phi = \varphi_k$. 
We obtain 
 the existence of a family of diffeomorphisms $\ti\varphi_k$ defined on $\B(0,\frac 1 2)$ having the required properties: 
 the estimate \eqref{phiesty} follows from  the third estimate of \eqref{metric-close} and \eqref{firstc0}. 
\end{proof}
As a consequence of Proposition \ref{prop-dec-2-sol-HMF}, we can measure the difference of the two corresponding solutions to $\de$-Ricci-DeTurck flow:
\begin{cor}\label{cor-mod-DTRF}
We assume the setting \eqref{S} of the previous Proposition \ref{prop-dec-2-sol-HMF}. 
For   $\al_0 \in (0,1),$  there exists an   $0<\ep_1= \ep_1(\al_0,n)$ where  $\ep_1$ is bounded from above by the  $\ep_0(\al_0,n)$ of Theorem \ref{RicciDeTurck},    such that 
 if   $  D_{0}: B_{d_{0,1}}(x_0,2R) \to \R^n$ 
satisfies   \eqref{chat} for $d_{0,1}$ in place of $d_0$, and $\ep_1$ in place of $\ep_0,$ 
and  $|\Rm(\cdot,t)|\leq \frac{\ep_1}{t}$ for all $t\in (0,T)$ for both solutions $g_1,g_2,$
then the following holds. 

Let  $\ti g_i(t)\coloneqq (F_i(t))_{\ast}g_i(t)$, $t\in(0,\hat S)$, $i=1,2$ be  the associated solutions to $(g_i(t))_{t\in(0,T)}$ to $\de$-Ricci-DeTurck flow coming out of $ (D_0)_{\ast}d_{0,1}$   respectively $ D_0\, \of\,  (\psi_0)^{-1}$ in the distance sense provided by [(iii), Theorem \ref{RicciDeTurck}], and constructed in the proof thereof.  Let $(t_k)_{k\in \N}$ be any sequence of positive times with $t_k \to 0$ as $k\to \infty$. 

Then there exists an $   S =   S(n,\al_0)  $    and   a  solution  $(\hat{g}_1^k(t)=(\hat{F}_1^k(t))_{\ast}g_1(t))_{t\in[t_k, \ti S )}$, $\ti S := \min(S,T/4)$, to $\de$-Ricci-DeTurck flow associated to $(g_1(t))_{t\in[t_k,T)}$ which is $\al_0$-close to the $\delta$ metric on $\B(0,\frac 1 2)$   such that 
\begin{equation}
\begin{split}\label{est-app-sequ}
\lim_{t_k\rightarrow 0^+}|\hat{ g}_1^k(t_k)-\ti g_2(t_k)|_{\de}&=0,\quad \text{uniformly on $\B(0,\frac 1 2)$},\\
|\hat{F}_1^k(t)- D_0|&\leq c(n)\sqrt{t},\quad \text{on $B_{d_0}\big(x_0,\tfrac{3}{2}\big)$, for $t\in[t_k,\ti S) $}.
\end{split}
\end{equation}
\end{cor}

\begin{proof}
For $\alpha_0\in (0,1),$ let $\ep_0$ be the constant provided in Theorem \ref{RicciDeTurck}. 
We now   choose $\ep_1 $ so that if  $  D_{0}: B_{d_{0,1}}(x_0,4R) \to \R^n$ 
satisfies   \eqref{chat} for $d_{0,1}$ in place of $d_0$, and $\ep_1$ in place of $\ep_0,$
 that then the solutions 
  $F_1:B_{d_{0,1}}(x_0,\frac 3 2)\times(0,\hat{S})\rightarrow \R^n$, $F_2:B_{d_{0,2}}(\hat x_0,\frac 3 2)\times(0,\hat{S})\rightarrow \R^n$,
  constructed in the proof of  Theorem \ref{RicciDeTurck}  remain
$1+\ep_0/2$ Bi-Lipschitz on $B_{d_{0,1}}(x_0,3/2)$ respectively  
$B_{d_{0,2}}(x_0,3/2)$ for $t\in [0,\hat S].$ This means 
\begin{equation} \label{metricin2b}
  \left(1+\frac{\ep_0}{2}  \right)^{-1} \delta \leq \ti g(t) \leq    \left(1+\frac{\ep_0}{2}\right) \delta\, ,
 \end{equation}
for $t\in [0,\hat S].$

Let $\hat{F}_1(t_k)\coloneqq \tilde{\varphi}_k\circ F_1(t_k)$ where the family of maps $\tilde{\varphi}_k$ is obtained from  Proposition \ref{prop-dec-2-sol-HMF}.   Then \eqref{metricin-adj-map} implies the first statement in \eqref{est-app-sequ}. 
Furthermore,  $\hat{F}_1(t_k)$ is $(1+\ep_0)$ bi-Lipschitz with respect to $g_1(t_k)$ in view of
\eqref{metricin2b} and \eqref{metricin-adj-map}. Thus we can apply Theorem \ref{RicciDeTurck} to  $\hat{F}_1(t_k)$ and $(M_1^n,g_1(t))_{t\in [t_k,T)}$ to 
obtain solutions $\hat{F}^k_1(t)$ for $t\in [t_k,\ti S)$ to   Ricci harmonic map heat flow   which are  $(1+\al_0)$ bi-Lipschitz and diffeomorphisms onto their image. Hence  $(\hat g^k_1(t):= \hat{F}^k_1(t)_*(g_1(t)))_{t\in [t_k,\ti S]}$ 
satisfy   $(1-\al_0) \de \leq \hat g^k_1(t) \leq (1+\al_0) \de$ as required.
 The second estimate in \eqref{est-app-sequ} is a direct consequence of previously established estimates: if $t\in[t_k,\ti{S})$,
\begin{equation*}
\begin{split}
|\hat{F}_1^k(t)- D_0|&\leq |\hat{F}_1^k(t)- \hat{F}_1^k(t_k)|+|\hat{F}_1^k(t_k)- D_0|\\
&\leq c\sqrt{t-t_k}+|\hat{F}_1^k(t_k)- D_0|\\
&\leq c\sqrt{t}+|(\ti\varphi_k-\operatorname{Id})(F_1(t_k))|+|F_1(t_k)-D_0|\\
&\leq c\sqrt{t},
\end{split}
\end{equation*}
which is the second inequality of \eqref{est-app-sequ}. 
Here we have used \eqref{cauchy-pb-init} from Theorem \ref{RicciDeTurck}  in the second and the  last inequality, and  \eqref{phiesty} in the last inequality.
\end{proof}

 \section{Polynomial convergence rate}\label{sec-pol-con-rate}
  We start by establishing a (faster than) polynomial convergence rate for the difference of two solutions to $\de$-Ricci-DeTurck flow in the $L^2_{loc}$ sense.
\begin{lemma}[Faster-than-polynomial $L^2$ convergence rate]\label{lemma-fast-L-2}
Let $n\geq 2$ be an integer. Then there exists an $\ep_0(n)>0$  depending only on $n$ such that the following is true.  Let $\ep\in (0,\ep_0(n)]$ and  $(\ti g_i(t))_{t\in[0,T)}$, $i=1,2$, be two smooth solutions to $\de$-Ricci-DeTurck flow on $\B(0,R)\times [0,T)$ which are $\ep$-close to the Euclidean metric, 
 \begin{equation*}
 \begin{split}\label{basic-ass-conv-rate}
 (1+\ep)^{-1} \delta \leq \ti g_i(t) \leq (1+\ep)\delta \quad\text{on $\B(0,R)\times [0,T)$},
 \end{split}
  \end{equation*} 
for $i=1,2$. For each $k\in\N,$ $r>0$, there exists $C_k(n)=C(n,k)>0$  and   $V(n,k,r)$ which  is non-decreasing in $r,$ 
    such that if $|\ti g_2(0)-\ti g_1(0)|^2_{\de}\leq \tau_0$ on $\B(0,R)$, then
  \begin{equation*}
\dashint_{\B(0,2^{-k} R)}|\ti g_1(t) -\ti g_2(t)|^2_{\de}\,dx\leq   \tau_0 \cdot V\left(n,k,\tfrac{t}{R^2}\right) + \Big(\frac{t}{R^2}\Big)^{\frac{k}{2}} \cdot C_k(n)  \quad t\in[0,T). 
\end{equation*}
 
\end{lemma}

\begin{proof}[Proof of Lemma \ref{lemma-fast-L-2}]
Following the statement and the proof of \cite[Lemma $6.1$]{Der-Sch-Sim}, consider the function 
\begin{equation*}
v(t)\coloneqq  |h(t)|^2 \left(1+ \lambda ( |\ti g_1(t)-\delta|^2 + |\ti g_2(t)-\delta|^2)\right),\quad \lambda >0\, ,
\end{equation*}
  where $h(t) = \tilde g_1(t) - \tilde g_2(t)$.  Define  
$$\tilde{h}^{ab}(t)\coloneqq  \frac 12 \left(\ti g_1^{ab}(t)+\ti g_2^{ab}(t)\right) \quad \text{and} \quad \hat{h}^{ab}(t)\coloneqq  \frac 12 \left(\ti g_1^{ab}(t)-\ti g_2^{ab}(t)\right)\, .$$ We record the fact that $v$ satisfies the following differential inequalities for $\lambda =  \frac{1}{ \sqrt{\ep}},$ if $\ep\in (0,\ep_0(n)]$ where $\ep_0(n)>0$ is sufficiently  sufficiently small ($\ep_0(n)$ is a constant depending only on $n$):
\begin{equation}
\begin{split}\label{ineq-v-lovely}
\partt v &\leq  \tilde{h}^{ab}\partial_a \partial_b v -   \lambda |h|^2 \left(|\D \ti g_1|^2 + |\D \ti g_2|^2\right)  -   |\D h|^2\\
 &\ \ \ + \left(1+ \lambda ( |\ti g_1-\delta|^2 + |\ti g_2-\delta|^2)\right) \Big( h * \hat{h} *\D^2 (\ti g_1+\ti g_2)\Big)\,  \\[-1ex] 
 & \ \ \ + \lambda |h|^2 \sum_{l=1}^2\big(\hat h *  (\ti g_{l} -\de)* \D^2\ti  g_l\big).
\end{split}
\end{equation}
 and    $ v(t) \leq 2|h(t)|^2$.    Observe that we always have $|h(t)|^2 \leq v(t)$ by definition of $v(t)$.

Now, choose a smooth cut-off function $\eta: \B(0,R)\rightarrow[0,1]$ such that $\eta\equiv 1$ on $\B(0,R/2)$ with support in $\B(0,R)$ and such that $|\D\eta|\leq c\cdot R^{-1}$ for some positive universal constant $c$. Then multiplying \eqref{ineq-v-lovely} by $\eta$ and integrating over $\B(0,R)$ gives:

\begin{equation*}
\begin{split}
\partt \int_{\B(0,R)} \eta v &\leq   \int_{\B(0,R)}\eta\tilde{h}^{ab}\partial_a \partial_b v   - \lambda \int_{\B(0,R)}\eta |h|^2 \left(|\D \ti g_1|^2 + |\D \ti g_2|^2\right)  -  \int_{\B(0,R)} \eta |\D h|^2\\
& \ \ \ + \int_{\B(0,R)} \eta \left(1+ \lambda ( |\ti g_1-\delta|^2 + |\ti g_2-\delta|^2)\right) ( h * \hat{h} *\D^2 (\ti g_1+\ti g_2)) \\
& \ \ \ + \lambda \int_{\B(0,R)}  \eta  \sum_{l=1}^2|h|^2 (\hat h *  (\ti g_{l} -\de)* \D^2 \ti g_l).
\end{split}
\end{equation*}

Integrating the first and last two terms on the right hand side of the above inequality by parts, we get:
\begin{equation*}
\begin{split}
\partt \int_{\B(0,R)} \eta v &\leq   -\int_{\B(0,R)}\partial_a \left(\eta\tilde{h}^{ab}\right)\partial_b v   - \lambda \int_{\B(0,R)}\eta |h|^2 \left(|\D \ti g_1|^2 + |\D \ti g_2|^2\right) - \int_{\B(0,R)} \eta |\D h|^2\\
& \ \ \ + \int_{\B(0,R)} \D \Big( \eta\big(1+ \lambda ( |\ti g_1-\delta|^2 + |\ti g_2-\delta|^2)\big) * h * \hat{h} \Big )    *  \D (\ti g_1+\ti g_2) \\
& \ \ \ + \lambda \int_{\B(0,R)} \sum_{l=1}^2 
\D  \Big(\eta |h|^2( \hat h \ast  (\ti g_{l} -\de)) \Big) * \D\ti  g_l\\
& =: A+ B +C +D + E\, . 
\end{split}
\end{equation*}
 We can now mimick the estimates of each integral quantity as in the proof of \cite[Lemma $6.1$]{Der-Sch-Sim}: recalling that $|\D\tilde g_i| \leq c(n) t^{-1/2}$ for $i=1,2$ and   taking into account the additional presence of the cut-off function $\eta$, we see  
\begin{equation*}
\begin{split}
 |A|  \leq&\, c(n)(1+2\lambda\ep^2)\int_{\B(0,R)}\eta( |\D \ti g_1| + |\D\ti g_2|)|h||\D h|
\\
&\quad+c(n)\lambda\ep\int_{\B(0,R)} \eta |h|^2 ( |\D \ti g_1|^2 + |\D \ti g_2|^2) +c(n)(1+\ep)\int_{\B(0,R)}|\D v||\D\eta|\\
\leq &\,c(n)(1+2\lambda\ep^2)\int_{\B(0,R)}\eta( |\D \ti g_1| + |\D\ti g_2|)|h||\D h|
+c(n)\lambda \ep\int_{\B(0,R)}  \eta |h|^2 ( |\D \ti g_1|^2 + |\D \ti g_2|^2)\\
&+ c(n)\int_{\B(0,R)}\left[2(1+2\lambda \ep^2)|h||\D h| 
+ \lambda \ep|h|^2  ( |\D \ti g_1| + |\D \ti g_2|)\right]|\D\eta|\\
\leq &\,c(n)(1+2\lambda\ep^2)^2\int_{\B(0,R)}\eta |h|^2( |\D \ti g_1|^2 + |\D\ti g_2|^2)  + \frac{1}{10} \int_{\B(0,R)} \eta   |\D h|^2\\
&+c(n)\lambda  \ep \int_{\B(0,R)} \eta |h|^2 ( |\D \ti g_1|^2 + |\D \ti g_2|^2)  \\
&+ \frac{c(n)\varepsilon\lambda}{\sqrt{t}}\int_{\B(0,R)}|\D \eta||h|^2+c(n)(1+2\lambda\ep^2)^2\bigg(\int_{\B(0,R)}|\D\eta||h|^2\bigg)^{\frac{1}{2}}\bigg(\int_{\B(0,R)}|\D\eta||\D h|^2\bigg)^{\frac{1}{2}}\, ,
\end{split}
\end{equation*}
 
A similar analysis can be performed on the integrals $D$ and $E$ to show:
\begin{equation}
\begin{split}\label{gal-est-v}
\partt \int_{\B(0,R)} &\eta v + \frac{1}{2}\int_{\B(0,R)}\lambda \eta |h|^2 \left(|\D \ti g_1|^2 + |\D \ti g_2|^2\right)+\eta |\D h|^2 \leq\\
&\frac{1}{\sqrt{t}}\int_{\B(0,R)}|\D\eta||h|^2+c(n)\bigg(\int_{\B(0,R)}|\D\eta||h|^2\bigg)^{\frac{1}{2}}\bigg(\int_{\B(0,R)}|\D\eta||\D h|^2\bigg)^{\frac{1}{2}},
\end{split}
\end{equation}
where the constants $c(n)$ may have changed from line to line but depend still only on $n$.

By dividing by  $\mathcal{H}^n_{d(\delta)}(\B(0,R))$ and using the facts that 
\begin{equation*}
|\D\eta|\leq c(n)R^{-1},\quad |h|^2 \leq \ep_0(n),\quad |Dh|^2 \leq \frac{c(n)}{t},
\end{equation*}
 one gets a first rough convergence rate at $t=0$:
\begin{equation}
\begin{split}\label{first-est-v-int}
\partt \dashint_{\B(0,R)} \eta v + &\frac{1}{2}\dashint_{\B(0,R)}\lambda \eta |h|^2 \left(|\D \ti g_1|^2 + |\D \ti g_2|^2\right)+\eta|\D h|^2 \leq\frac{c(n)R^{-1}}{\sqrt{t}},
\end{split}
\end{equation}
 for $t\in(0,T)$. By integrating in time \eqref{first-est-v-int} and by the assumption on the initial condition:
\begin{equation}
\begin{split}\label{first-est-v}
 \dashint_{\B(0,2^{-1} R)}  |h(t)|^2 + &\int_0^t\dashint_{\B(0,2^{-1}R)}\lambda  |h|^2 \left(|\D \ti g_1|^2 + |\D \ti g_2|^2\right)+|\D h|^2 \leq c(n)\tau_0+c(n)R^{-1}\sqrt{t},
\end{split}
\end{equation}
for $t\in(0,T)$, where we have used 
that $|h(t)|^2\leq   v(t) $  as we noted  at the beginning of the proof. \\


 {\bf Claim:} For every $k\in \mathbb{N}$, we have
\begin{equation*}
\dashint_{\B(0,2^{-k}R)} \!\!\! |h(t)|^2 + \int_0^t\dashint_{\B(0,2^{-k}R)}\!\!\!   |\D h|^2
\leq V\left(n,k,\tfrac{t}{R^2}\right)\tau_0+C_kR^{-k} t^{\frac{k}{2}},
\end{equation*}
for  $t\in(0,T)$, where $V\left(n,k,\frac{t}{R^2}\right)$ denotes a non-negative function which is non-decreasing in the third variable. 
 
This follows by induction, assuming that the statement holds for $k$,  by using the induction assumption in \eqref{gal-est-v} applied to a cut-off function $\eta$ such that $\eta\equiv 1$ on   $\B(0,2^{-k}R)$ with support in $\B(0,2^{-k+1}R)$, and  $|D\eta|\leq \frac{c_k}{R}$:
\begin{equation*}
\begin{split}\label{sec-est-v}
& \dashint_{\B(0,2^{-k-1}R)}  |h(t)|^2 + \int_0^t\dashint_{\B(0,2^{-k-1}R)}  |\D h|^2\\
 & \leq c \tau_0 + \int_0^t\frac{c_k R^{-1}}{\sqrt{s}}\dashint_{\B(0,2^{-k}R)}|h|^2 +c_kR^{-1}\bigg(\int_0^t\dashint_{\B(0,2^{-k}R)}|h|^2\bigg)^{\frac{1}{2}}\bigg(\int_0^t\dashint_{\B(0,2^{-k} R)}|\D h|^2\bigg)^{\frac{1}{2}}\\
 &\leq c\tau_0 + \int_0^t\frac{cR^{-1}}{\sqrt{s}}\left(c_kV\left(n,k,\tfrac{s}{R^2}\right)\tau_0+c_kC_k\left(\tfrac{s}{R^2}\right)^{\frac{k}{2}}\right)\,ds\\
 &\quad+ c_kR^{-1}\left(\int_0^tV\left(n,k,\tfrac{s}{R^2}\right)\tau_0+C_k\left(\tfrac{s}{R^2}\right)^{\frac{k}{2}}\,ds\right)^{\frac{1}{2}}\left(V\left(n,k,\tfrac{t}{R^2}\right)\tau_0+C_k\left(\tfrac{t}{R^2}\right)^{\frac{k}{2}}\right)^{\frac{1}{2}}\\
 &\leq c \tau_0+ c_k R^{-1}\sqrt{t}\,V\left(n,k,\tfrac{t}{R^2}\right)\tau_0+c_kC_kR^{-(k+1)}\int_0^ts^{\frac{k-1}{2}}\,ds\\
 &\quad +c_k\left(V\left(n,k,\tfrac{t}{R^2}\right)\tfrac{\sqrt{t}}{R}\sqrt{\tau_0}+C_k\left(\tfrac{\sqrt{t}}{R}\right)^{\frac{k}{2}+1}\right)\left(V\left(n,k,\tfrac{t}{R^2}\right)\sqrt{\tau_0}+C_k\left(\tfrac{\sqrt{t}}{R}\right)^{\frac{k}{2}}\right)\\
 &\leq c\tau_0 + c_kR^{-1}\sqrt{t}V\left(n,k,\tfrac{t}{R^2}\right)\tau_0+c_kC_kR^{-(k+1)}t^{\frac{k+1}{2}}+c_kV\left(n,k,\tfrac{t}{R^2}\right) \sqrt{\tau_0}\left(\tfrac{\sqrt{t}}{R}\right)^{\frac{k}{2}+1}\\
 &\leq V\left(n,k+1,\tfrac{t}{R^2}\right)\tau_0+C_{k+1}R^{-(k+1)} t^{\frac{k+1}{2}},
\end{split}
\end{equation*} 
for $t\in(0,T)$, where $V\left(n,k+1,\frac{t}{R^2}\right)$ denotes a non-negative function which is non-decreasing in the third variable and which may vary from line to line. A similar remark applies to $C_k$. Here we have used the induction assumption in the second inequality together with the elementary inequality $\sqrt{a+b}\leq \sqrt{a}+\sqrt{b}$ for real numbers $a,b\geq 0$ in the third inequality. Finally, Young's inequality is invoked in the last inequality.
\end{proof}

We are now in a position to prove the main result of this section:

\begin{prop}[Almost faster-than-polynomial $L^{\infty}$ convergence rate]\label{lemma-fast-C-0}
  Under the same assumptions of Lemma \ref{lemma-fast-L-2} and for each $k\in\N,$ and $r>0$, there exists $C(n,k)>0$  and $ V\left(n,k,r\right)>0$ where $V(n,k,\cdot)$  is non-decreasing in $r,$ such that if $\B(p,\sqrt{t})\times (2^{-1}t,2t)\subset \B(0, 2^{-k}R)\times (0,T)$, then:
  \begin{equation}\label{c0-bd-prop-pol}
 |\ti g_1(t) -\ti g_2(t)|^2_{\de}\leq \tau_0 \cdot R^n t^{-\frac n 2} \cdot V \left(n,k,\tfrac{t}{R^2}\right) +  \Big(\frac{t}{R^2}\Big)^{\frac{k-n}{2}} \cdot  C(n,k),\quad \text{on $\B(p,\sqrt{ t}/2)$.}
\end{equation}
In particular, if $k\in\mathbb{N}$, $j\in\mathbb{N}\setminus \{0\},$ and $r>0$,  there exists $ C(n,k,j)>0$ and
$V\left(n,k,j,r \right)$ which  is non-decreasing in $r>0,$ 
such that if $\B(p,\sqrt{t})\times (2^{-1}t,2t)\subset \B(0,2^{-k}R)\times   (0,T)$, then on a smaller ball $\B(p,\sqrt{t}/4)$,
 \begin{equation*}
t^{\frac{j}{2}}|\D^j\left(\ti g_1(t) -\ti g_2(t)\right)|_{\de}\leq   \left(R^{\frac n 2}t^{-\frac{n}{4}} V\left(n,k,j,\tfrac{t}{R^2}\right)\sqrt{\tau_0}+C(n,k,j)\left(\tfrac{t}{R^2}\right)^{\frac{k-n}{4}}\right)^{1-\frac{1}{\max\{j,2\}}}.
\end{equation*}
\end{prop}

\begin{rmk}
The estimates on the covariant derivatives of the difference of two solutions to $\de$-Ricci-DeTurck flow satisfying the setting of Lemma \ref{lemma-fast-L-2} obtained in Proposition \ref{lemma-fast-C-0} are not sharp but they will be sufficient for the proof of Theorem \ref{thm:main.1}. 
\end{rmk}
\begin{proof}[Proof of Proposition \ref{lemma-fast-C-0}]
With the same notations as those of the proof of Lemma \ref{lemma-fast-L-2}, and according to the proof of \cite[Lemma $6.1$]{Der-Sch-Sim}, recall that the function $|h|^2$ satisfies the following differential inequality:
\begin{equation*}
\begin{split}
\partt |h|^2 &\leq \tilde{h}^{ab}\partial_a \partial_b |h|^2 - \frac{2}{1+\ep} |\D h|^2 + h * \hat{h} *\D^2 (\ti g_1+\ti g_2)\\
&\ \ \ + h*\hat{h}* \ti g_1^{-1} * \D \ti g_1*\D \ti g_1 + h*\ti g_2^{-1}* \hat{h}* \D\ti g_1*\D \ti g_1 \\
&\ \ \ + h*\ti g_2^{-1}* \ti g_2^{-1} *\D h *\D \ti g_1 + h* \ti g_2^{-1}* \ti g_2^{-1} * \D \ti g_2 *\D h\, .
\end{split}
\end{equation*}
In particular, since there exists $C>0$ such that $t|\D\ti g_i|^2+t|\D^2\ti g_i|\leq C$, $i=1,2$, for $t\in(0,T)$, one gets:
\begin{equation}
\begin{split}\label{diff-inequ-norm-h-prelim}
\partt |h|^2 &\leq \tilde{h}^{ab}\partial_a \partial_b |h|^2- \frac{2}{1+\ep} |\D h|^2+\frac{C}{t}|h|^2+\frac{C}{\sqrt{t}}|h||\D h|\\
&\leq \tilde{h}^{ab}\partial_a \partial_b |h|^2- |\D h|^2+\frac{C}{t}|h|^2,
\end{split}
\end{equation}
where $C$ is a positive constant depending on $n$ that may vary from line to line. Here we have used Young's inequality in the second line $\frac{C}{\sqrt{t}}|h||\D h|\leq \frac{1-\varepsilon}{1+\varepsilon}|Dh|^2+\frac{C'}{t}|h|^2$ to absorb the gradient term $\frac{1-\varepsilon}{1+\varepsilon}|\D h|^2$.

In particular, based on the definition of the coefficients $\tilde{h}^{ab}$ in terms of the two solutions, given at the beginning of the proof of Lemma \ref{lemma-fast-L-2},
\begin{equation*}
\begin{split}\label{diff-inequ-norm-h-bis}
\partt |h|^2 &\leq \partial_a\left(\tilde{h}^{ab} \partial_b |h|^2\right)-\partial_a\tilde{h}^{ab}\partial_b |h|^2-|\D h|^2+\frac{C}{t}|h|^2\\
&\leq\partial_a\left(\tilde{h}^{ab} \partial_b |h|^2\right)+C|h||\D h|(|Dg_1| + |Dg_2|) -|\D h|^2+\frac{C}{t}|h|^2\\
&\leq\partial_a\left(\tilde{h}^{ab} \partial_b |h|^2\right)-\frac{1}{2}|\D h|^2+\frac{C}{t}|h|^2,\quad \text{on $\B(0,R)\times(0,T)$.}
\end{split}
\end{equation*}

 As a first conclusion, there exists a positive constant $C$ such that the function $t^{-C}|h|^2$ satisfies:
\begin{equation}
\begin{split}\label{diff-inequ-norm-h}
\partt \left(t^{-C}|h|^2\right) &\leq \partial_a\big(\tilde{h}^{ab} \partial_b\big( t^{-C}|h|^2\big)\big).
\end{split}
\end{equation}

Choose $k>2C+n$ and perform a local Nash-Moser iteration on each ball $\B(p,\sqrt{t})\times\left(t,2t\right)\subset \B(0,2^{-k}R)\times(0,T)$ to get for each $\theta\in(0,1)$,  
\begin{equation}\label{loc-nash-moser}
\sup_{\B(p,\sqrt{\theta t})\times\left(t(1+\theta),2t\right)}t^{-C}|h|^2\leq C(n,\theta)\, \dashint_t^{2t}\dashint_{\B(p,\sqrt{t})}\left(s^{-C}|h|^2\right)\,dxds.
\end{equation}
 See for instance \cite[Theorem $6.17$]{Boo-Lieberman} with `$k=0$' for a proof.

Now apply Lemma \ref{lemma-fast-L-2} so that if $t<2^{-2k}R^2$, the previous inequality \eqref{loc-nash-moser} leads to the pointwise bound:
\begin{equation*}
\begin{split}
\sup_{\B(p,\sqrt{\theta t})\times\left(t(1+\theta),2t\right)}t^{-C}|h|^2&\leq C(n,k,\theta)\big(\tfrac{R^2}{t}\big)^{\frac{n}{2}}\dashint_t^{2t}s^{-C}\dashint_{\B(0,2^{-k} R)}|h(s)|^2\,ds\\
&\leq C(n,k,\theta)\big(\tfrac{R^2}{t}\big)^{\frac{n}{2}}\dashint_t^{2t}\left(V\left(n,k,\tfrac{s}{R^2}\right)\tau_0+\left(\tfrac{s}{R^2}\right)^{\frac{k}{2}}\right)\cdot s^{-C}\,ds\\
&\leq C(n,k,\theta) R^nt^{-(C + \frac{n}{2})}  \dashint_t^{2t}\left(V\left(n,k,\tfrac{s}{R^2}\right)\tau_0+\left(\tfrac{s}{R^2}\right)^{\frac{k}{2}}\right)\cdot  \,ds\\
&\leq  C(n,k,\theta) R^nt^{-(C + \frac{n}{2})} \Big( V\left(n,k,\tfrac{t}{R^2}\right)\tau_0+\left(\tfrac{t}{R^2}\right)^{\frac{k}{2}}\Big),\\
& =  C(n,k,\theta) R^nt^{-(C + \frac{n}{2})}   V\left(n,k,\tfrac{t}{R^2}\right)\tau_0 +  C(n,k,\theta)t^{-C}\left(\tfrac{t}{R^2}\right)^{\frac{k-n}{2}}
\end{split}
\end{equation*}
for $t\in(0, 2^{-2k}R^2).$

Here we have used that $V\left(n,k,r\right)$ is non-decreasing in $r$. 
This estimate in turn implies the desired pointwise convergence rate, i.e.
\begin{equation*}
\sup_{\B(p,\sqrt{\theta t})\times\left(t(1+\theta),2t\right)}|h|^2\leq    C(n,k,\theta) R^nt^{-\frac{n}{2}}   V\left(n,k,\tfrac{t}{R^2}\right)\tau_0 +  C(n,k,\theta)\left(\tfrac{t}{R^2}\right)^{\frac{k-n}{2}} 
\end{equation*}
as long as $\B(p,\sqrt{t})\times(t,2t)\subset \B(0,2^{-k} R)\times(0,T).$

In order to prove the bounds on the derivatives, we recall the following standard local interpolation inequalities on Euclidean space:
\begin{equation}
\sup_{\R^n}|\D^ju|\leq C(n,j,m)\sup_{\R^n}|u|^{1-\frac{j}{m}}\cdot \sup_{\R^n}|\D^mu|^{\frac{j}{m}},\label{int-inequ-easy}
\end{equation}
where $u$ is any smooth function on $\R^n$ with compact support and $0\leq j\leq m$. See for example \cite[Theorem $3.70$]{Aub-Boo}, or inductively apply Lemma A.5 of \cite{Sch-Sim-Sim-Hyp} for a proof.
Now, by interior Bernstein-Shi's estimates on solutions to $\delta$-Ricci DeTurck flow as stated in \cite[Corollary 5.4]{Sch-Sim-Sim-Euc} together with \eqref{c0-bd-prop-pol}, the previous interpolation inequalities applied to the coordinates of the tensor $\eta\cdot h$ where $\eta$ is a smooth cut-off function such that $\eta\equiv 1$ on $\B(p,\sqrt{\theta' \cdot t})$ with support in $\B(p,\sqrt{\theta\cdot t})$ for $0<\theta'<\theta<1$ and such that $t^\frac{k}{2}|\D^k\eta|$ is bounded on $\R^n$ for all $k\geq 0$, imply for $0\leq j\leq m $:
 \begin{equation}
 \begin{split}
&\sup_{\B(p,\sqrt{\theta' t})}|\D^j h((3+\theta')t/2)|\leq\sup_{\R^n}|\D^j(\eta \cdot h((3+\theta')t/2))|\leq \\
&C(n,j,m)\sup_{\R^n}|\eta \cdot h((3+\theta')t/2)|^{1-\frac{j}{m}}\cdot \sup_{\R^n}|\D^m(\eta\cdot h((3+\theta')t/2))|^{\frac{j}{m}}\\
&\leq C(n,j,m,\theta)\sup_{\B(p,\sqrt{\theta t})}|h((3+\theta')t/2)|^{1-\frac{j}{m}}\left(\frac{C(n,m)}{t^{\frac{m}{2}}}\right)^{{\frac{j}{m}}}\\
&=  C(n,j,m,\theta)\sup_{\B(p,\sqrt{\theta t})}|h((3+\theta')t/2)|^{1-\frac{j}{m}}\left(\frac{C(n,m)}{t^{\frac{m}{2}}}\right)^{{\frac{j}{m}}}\\
&\leq C(n,j,m,\theta)\cdot t^{-\frac{j}{2}}\cdot \left(R^{\frac n 2}t^{-\frac{n}{4}} V\left(n,k,\theta,\tfrac{t}{R^2}\right)\sqrt{\tau_0}+C(n,k,\theta)\left(\tfrac{t}{R^2}\right)^{\frac{k-n}{4}}\right)^{1-\frac{j}{m}},
\end{split}
\end{equation}
where $C(n,j,m,\theta)$ is  a positive constant that may vary from line to line. This ends the proof of the desired estimate once one fixes $j\geq 1$ and let $m$ be $j^2$ if $j\geq 2$ and $m=2$ if $j=1$, and $\theta' = \frac{1}{4} < \theta = \frac{1}{2}<1$.
 \end{proof}

As a first consequence of Proposition \ref{lemma-fast-C-0}, we derive the following result which is an intermediate but essential step towards the proof of Theorem \ref{thm:main.1}.

\begin{cor}\label{cor-pol-main-thm-version}

For all $n\in \N$ and $\be_0\in (0,1)$ there exists $\ep_0=\ep_0(n,\be_0)>0$, depending only on $\be_0$ and $n$,  such that the following holds. 

Let $(M_i^n,g_i(t))_{t\in (0,T)}$, $i=1,2$, be smooth solutions to Ricci flow (not necessarily with complete time slices)  both  satisfying \eqref{CurvatureCond} and \eqref{DistCond}, such that  the locally well defined metrics at time zero agree (up to an isometry),  that is  for arbitrary $p\in M_1,$ 
 $ \lim_{t\downto 0}d(g_1(t)) = d_{0,1}$ on $U_{p,1}$, and 
 $\lim_{t\downto 0}d(g_2(t)) = d_{0,2},$  on $U_{\psi_0(p),2}$  
for  an  isometry    $\psi_0:(U_{p,1},d_{0,1}) \to (U_{\hat p,2} =\psi_0 (U_{p,1}), d_{0,2}),$ $\hat p:= \psi_0(p),$ that is
 $\psi_0:U_{p,1} \to U_{\hat p,2} = \psi_0 (U_{ p,1})$ is a homeomorphism with 
$ \psi_0^*(d_{0,2}) = d_{0,1}.$  

We assume \eqref{ReifenbergCond} holds true for the locally defined metric $d_0= {d_{0,1}},$   that    $p\in M_1,$
$x_0 \in U_{p,1},$   and that $\eqref{CoordCond}$  holds for some $\ti R> 2R >0,$ for   the locally defined metric $d_0= {d_{0,1}}$ on $U_p= U_{p,1}$ (and hence for $d_0= {d_{0,1}}$  on  $U_p= U_{p,2}$). 

Then there exists an $R_0\in(0,1)$ and $T_0>0$  and  solutions $(\hat{g}_1(t))_{t\in(0,T_0)}$ and $(\ti g_2(t))_{t\in(0,T_0)}$ to $\de$-Ricci-DeTurck flow defined on $\B(0,R_0)\times(0,T_0)$ satisfying the following:
 \begin{enumerate}
 \item \label{close-sec-5}The metrics $\hat{g}_1(t)$ and $\ti g_2(t)$ are $\beta_0$-close to the $\delta$-metric for $t\in(0,T_0)$.
  \\[-2ex]
 \item  \label{pol-close-sec-5} For each $j\geq 1$ and $k\in\N$, there exist  $C_k>0$ and $C_{j,k} >0$ and $T_k >0$  such that if $t\in(0,T_k]$:
  \begin{equation*}
  \begin{split}
  |\hat g_1(t) -\ti g_2(t)|_{\de}&\leq C_{k} \,t^{k},\quad \text{on $\B(0,\sqrt{ T_k})$}, \\
t^{\frac{j}{2}}|D^j\left(\hat g_1(t) -\ti g_2(t)\right)|_{\de}&\leq C_{j,k} \,t^{k\left(1-\frac{1}{\max\{2,j\}}\right)},\quad \text{on $\B(0,\sqrt{ T_k})$.}
\end{split} 
\end{equation*}
\item \label{part-3-sec-5}   One has $ \hat{g}_1(t)= (\hat F_1(t))_{\ast}g_1(t)$ and $\ti g_2(t)=(F_2(t))_{\ast}g_2(t),$ where $(\hat F_1(t))_{t\in(0,T_0)}$ and $(F_2(t))_{t\in(0,T_0)}$ are smooth families of bi-Lipschitz maps on $B_{d_{0,1}}(x_0, \frac{3}{2}R_0)$ respectively $B_{d_{0,2}}(\psi_0(x_0), \frac{3}{2}R_0)$ that satisfy the following:
\begin{enumerate}
\item The family of maps $(\hat F_1(t))_{t\in(0,T_0)}$ (respectively $(F_2(t))_{t\in(0,T_0)}$) is a solution to the Ricci-harmonic map flow with respect to the Ricci flow solution $(g_1(t))_{t\in(0,T_0)}$ (respectively $(g_2(t))_{t\in(0,T_0)}$).
\item \label{est-F-maps}
$  | \hat F_1(t)- D_0| \leq C_0\sqrt{t},\quad t\in(0,T_0)$ on $B_{d_{0,1}}(x_0,\frac{3}{2}R_0) $\\
and $  | F_2(t)- D_0 \,\of\, (\psi_0)^{-1}| \leq C_0\sqrt{t},\quad t\in(0,T_0)$ on $B_{d_{0,2}}((\psi_0)(x_0),\frac{3}{2}R_0),$ where here $D_0$ are the distance coordinates on $B_{d_{0,1}}(x_0,\frac{3}{2}R_0) $ coming from \eqref{CoordCond}.
\item \label{est-dist-rate} The distances ${\hat  d}_{1}(t)\coloneqq d(\hat  g_{1}(t))$ and $\tilde{d}_{2}(t)\coloneqq 
d(\tilde{g}_2(t))$ converge, uniformly to the same distance $(D_0)_{\ast}d_{0,1}$ on $\B(0,R_0)$ as $t$ approaches $0$.
\end{enumerate}
\end{enumerate}
\end{cor}
\begin{rmk}
We have  denoted the solution  to the Ricci-harmonic map heat flow from $(M_1^n,g_1(t))_{t\in (0,T_0)}$ to $(\R^n,\de)$ appearing in the statement of  this theorem  by  $\hat{F_1}$ in order to make it clear that is not necessarily the same solution to the one considered in 
Proposition \ref{prop-dec-2-sol-HMF} (which comes from Theorem \ref{RicciDeTurck}). The solution $\hat F_1$ appearing here is  obtained by modifying  the  
$F_1$ from Theorem \ref{RicciDeTurck} with the help of Corollary \ref{cor-mod-DTRF}, as is explained in the proof below. 
\end{rmk}

\begin{proof}
  Let $p,x_0 \in M_1$, $\hat p = \psi_0(p)$, $\hat x_0 = \psi_0(x_0) \in M_2$ be as in the statement of the theorem.
By scaling  once, we may assume that $R(x_0,n) = R(\hat x_0,n) >200$.  
We prove the estimates in this setting; scaling back to the original setting implies the  estimates for the original solution. 

Let us start by proving part \eqref{part-3-sec-5}. To this end, let $(t_k)_{k\in \N}$ be an arbitrary sequence of positive times with $t_k \downto 0$ for $k\to \infty$. According to Corollary \ref{cor-mod-DTRF}, there exists a smooth family of $(1+\be_0)$-bi-Lipschitz maps $( \hat F^k_1(t))_{t\in[t_k,\hat{S})}$ on $B_{d_{0,1}}(x_0, \frac{3}{2}R_0)$ for some $\be_0\in(0,1)$ solving the Ricci-harmonic map heat flow equation   such that  $|\hat{F}_1^k(t)- D_0|\leq c(n)\sqrt{t}$ on $B_{d_0}\big(x_0,\tfrac{3}{2}\big)$, for $t\in[t_k,\hat{S}),$
and, 
$\lim_{t_k\rightarrow 0^+}|\hat{ g}_1^k(t_k)-\ti g_2(t_k)|_{\de} =0,$  on $\B(0,R_0),$ 
for $\hat{g}_1^k(t_k) = (\hat F_1^k)_*(g_1)(t_k),$
in view of 
\eqref{est-app-sequ}.

 The Arz\'ela-Ascoli theorem guarantees the existence of a subsequence of $( \hat F^k_1(t))_{t\in[t_k,\hat{S})}$ that converges locally uniformly on $B_{d_{0,1}}(x_0, \frac{3}{2}R_0)\times[t_k,\hat{S})$ to a continuous family of $(1+\be_0)$-bi-Lipschitz maps $( \hat F_1(t))_{t\in[0,\hat{S})}$.   For the rest of the proof we fix this subsequence. By interior parabolic regularity (see for instance \eqref{int-est-Z-maps} in the proof of Theorem \ref{RicciDeTurck} and the proof of \cite[Theorem 3.8]{Der-Sch-Sim}), the convergence takes place in the $C^{\infty}_{loc}$ topology (for positive times) and this ensures that the family of maps $(\hat F_1(t))_{t\in(0,\hat{S})}$ is a solution to the Ricci-harmonic map heat flow equation with respect to the Ricci flow solution $(g_1(t))_{t\in(0,\hat{S})}$.   Passing to the limit in estimate  \eqref{est-app-sequ} of  Corollary \ref{cor-mod-DTRF} gives us the second estimate \eqref{est-F-maps} for $\hat F_1(t)$. 

Now, let  $( F_2(t))_{t\in(0,\hat{S})}$ on $B_{d_{0,2}}(\psi_0(x_0), \frac{3}{2}R_0)$ be the smooth family of bi-Lipschitz maps solving the Ricci-harmonic map heat flow equation  with initial value the map $ D_0\, \of\,  (\psi_0)^{-1} $ provided by Theorem \ref{RicciDeTurck}. Estimate \eqref{est-F-maps} for $ F_2(t)$ follows from \eqref{cauchy-pb-init}. 

In order to prove \eqref{est-dist-rate}, let us notice that both distances $\hat{d}_1(t)\coloneqq d(\hat{g}_1(t))$ and $\tilde{d}_2(t)\coloneqq d(\tilde{g}_2(t))$ converge to the same distance $(D_0)_{\ast}d_0$ as $t$ approaches $0$ thanks to condition \eqref{DistCond} and estimates \eqref{est-F-maps}.

We are now in a position to prove part \eqref{close-sec-5} and \eqref{pol-close-sec-5}.

According to Corollary \ref{cor-mod-DTRF}, the family of metrics $(\hat{g}_1^k(t)=(\hat{F}_1^k(t))_{\ast}g_1(t))_{t\in[t_k,\hat{S})}$ is a solution to $\de$-Ricci-DeTurck flow associated to $(g_1(t))_{t\in[t_k,\hat{S})}$ which is $(1\pm \beta_0)$ close to the $\delta$ metric on $\B(0,R_0)$. Moreover, if $\ti g_2(t)\coloneqq (F_2(t))_{\ast}g_2(t)$, $t\in(0,T)$, is the associated solution to $(g_2(t))_{t\in(0,\hat{S})}$ to $\de$-Ricci-DeTurck flow provided by Theorem \ref{RicciDeTurck}, 
\begin{equation}
\lim_{t_k\rightarrow 0^+}|\hat{ g}_1^k(t_k)-\ti g_2(t_k)|_{\de}=0.
\end{equation}
Denote an upper bound on $|\hat{ g}_1^k(t_k)-\ti g_2(t_k)|_{\de}$ on $\B(0,R_0)$ by $\varepsilon_k$. 
Now, let us apply Proposition \ref{lemma-fast-C-0} to $(\hat{g}_1^k(\tau+t_k))_{\tau\in[0,\hat{S}/2)}$ and $(\ti g_2(\tau+t_k))_{\tau\in[0,\hat{S}/2)}$ for $t_k$ sufficiently small to get for each $l\geq 0$ and $R=1$,
  \begin{equation}
  \begin{split}
|\hat g^k_1(t) -\ti g_2(t)|^2_{\de}&\leq V\left(n,l, \hat{S}\right)\varepsilon_k (t-t_k)^{-\frac n 2} +C(n,l)\left(t-t_k\right)^{l},\quad \text{on $\B\left(p,\tfrac{\sqrt{t-t_k}}{2}\right)$,}\label{intermed-est-conv-rate}
\end{split}
\end{equation}
in view of 
\eqref{est-app-sequ}, as we pointed out above, whenever $\B(p,\sqrt{t-t_k})\times (2^{-1}(t-t_k),2(t-t_k))\subset \B_{2^{-2l-n}}(0)\times (0,\hat{S})$. Interior Bernstein-Shi's estimates for solutions to the $\de$-Ricci-DeTurck flow $(1 \pm \beta_0)$ close to $\delta$ metric as stated in \cite[Corollary 5.4]{Sch-Sim-Sim-Euc} let us pass to the limit as $t_k$ tends to $0^+$, i.e. for each $l\geq 1$, there exists $\hat{T}_l>0$ such that the solution 
$({\hat{g}_1}^k(t))_{t\in(t_k,\hat{T}_l)}$ converges in $C^{\infty}_{loc}(\B_{2^{-2l-n-1}}(0)\times(0,\hat{T}_l))$ to a solution  $\check g_1(t)$ to $\de$-Ricci-DeTurck flow 
which is $(1\pm \beta_0)$ close to the $\delta$ metric and which satisfies thanks to \eqref{intermed-est-conv-rate} the expected $C^0$-closeness to the solution $\ti g_2(t)$ as stated in  (\ref{pol-close-sec-5})  of Corollary \ref{cor-pol-main-thm-version} for $j=0$. The cases $j\geq 1$ follows similarly. Finally,  by construction,  
$\check g_1(t) = \lim_{k\to \infty} ({\hat{F}_1}^k)_{*} g_1(t)  = (\hat{F}_1)_{*}g_1(t) =: \hat{g}_1(t),$  which completes the proof of   \eqref{pol-close-sec-5} and the theorem.
\end{proof}

\section{Exponential convergence rate}\label{sec-exp-conv}

Before stating the main result of this section, we take the opportunity of explaining the behaviour at $t=0$ of a solution to the heat equation on Euclidean space starting from an initial data vanishing on a ball.
Although  this result (and the short proof thereof) will not be used in the rest of the paper, we  present it here,
as it indicates   what one may expect in more general cases.
 More precisely,
\begin{prop}\label{prop-baby-conv-rate}
Let $u_0\in L^1(\R^n)$ with compact support. Assume $u_0$ vanishes on $\B(0,R)$ fro some $R>0$. Then the \textbf{standard} solution $u$ to the heat equation with initial data $u_0$, obtained by convolution with the Euclidean heat kernel, satisfies:
\begin{equation*}
|u(x,t)|\leq (4\pi t)^{-\frac{n}{2}} \exp\left(-\frac{(R-|x|)^2}{4t}\right)\|u_0\|_{L^1(\R^n)},\quad |x|\leq R,\quad t>0.
\end{equation*}
\end{prop}
\begin{proof}
The proof is a simple consequence of the fact that $u(\cdot ,t)=K_t\ast u_0$ on $\R^n$ for $t>0$ where $K_t$ is the Euclidean heat kernel defined by $K_t(x):=(4\pi t)^{-\frac{n}{2}}\exp\left(-\frac{|x|^2}{4t}\right)$ for $t>0$ and $x\in \R^n$. Indeed, if $|x|\leq R$,
\begin{equation*}
\begin{split}
|u(x,t)|&=|K_t\ast u_0(x)|\leq \int_{\R^n}K_t(x-y)|u_0(y)|\,dy\\
&\leq\int_{\R^n\setminus \B(0,R)}(4\pi t)^{-\frac{n}{2}} \exp\left(-\frac{|x-y|^2}{4t}\right)|u_0(y)|\,dy\\
&\leq (4\pi t)^{-\frac{n}{2}}\int_{\R^n\setminus \B(0,R)}\exp\left(-\frac{(R-|x|)^2}{4t}\right)|u_0(y)|\,dy\\
&\leq(4\pi t)^{-\frac{n}{2}} \exp\left(-\frac{(R-|x|)^2}{4t}\right)\|u_0\|_{L^1(\R^n)},\quad  t>0.
\end{split}
\end{equation*}
Here we have used the assumption on the support of $u_0$ in the second line and the triangular inequality in the third line.
\end{proof}

The following theorem can be interpreted as a non-linear version of Proposition \ref{prop-baby-conv-rate} and as such, it proves a sharp qualitative convergence rate.
 \begin{thm}\label{dream-conv-rate}
 Let $n\geq 2$ be an integer. Then there exists $\varepsilon(n)>0$ such that the following holds true.
 Let $R>0$, $T>0$ and $(\ti g_i(t))_{t\in(0,T)}$, $i=1,2$, be two smooth solutions to $\de$-Ricci-DeTurck flow on $\B(0,R)\times (0,T)$ which are $\ep(n)$-close to the Euclidean metric, 
 \begin{equation*}
 \begin{split}\label{basic-ass-conv-rate.2}
 (1+\ep(n))^{-1} \delta \leq \ti g_i(t) \leq (1+\ep(n))\delta \quad\text{on $\B(0,R)\times (0,T)$},
 \end{split}
  \end{equation*}
for $i=1,2$.  Assume furthermore that $\lim_{t\rightarrow 0^+}\sup_{\B(0,R)}|\ti g_2(t)-\ti g_1(t)|_{\delta}=0$.
  
  Then there exist positive constants $T_0$, $C_0$ and $R_0$ depending only on $n,R,T$ such that
  \begin{equation*}
|\ti g_1(t) -\ti g_2(t)|_{\de}\leq  \exp\left(-\frac{C_0}{t}\right),\quad \text{on $\B(0,R_0)\times (0,T_0)$.}
\end{equation*}
 \end{thm}
The proof of Theorem \ref{dream-conv-rate} relies on the following crucial result due to Aronson \cite{Aronson-fund-sol}:

\begin{thm}[\cite{Aronson-fund-sol}]\label{lemma-fund-sol-est}
Let $a^{ij}\in L^{\infty}(\R^n\times(0,T))$, $1\leq i,j\leq n$, such that on $\R^n\times(0,T)$,
\begin{equation*}
\lambda^{-1} |\xi|^2\leq a^{ij}(x,t)\xi_i\xi_j\leq \lambda |\xi|^2,\quad \forall \,\xi=(\xi_1,...,\xi_n)\in\R^n,
\end{equation*}
for some positive uniform $\lambda$. Then there exist positive constants $C=C(n,T,\lambda)$ and $\alpha=\alpha(n,T,\lambda)$ such that the fundamental solution of the divergence structure parabolic equation $\partial_tu-\partial_i\left(a^{ij}\partial_ju\right)=0$ on $\R^n\times(0,T)$ satisfies:
\begin{equation}
K(x,t,y,s)\leq C(t-s)^{-\frac{n}{2}}\exp\left\{-\alpha\frac{|x-y|^2}{t-s}\right\},\quad 0\leq s<t,\quad x,y\in\R^n.
\end{equation}
\end{thm}
\begin{rmk}
In \cite{Aronson-fund-sol}, a lower bound for the fundamental solution is also provided. However, we will only use an upper bound in the sequel.
\end{rmk}

Let us now prove Theorem \ref{dream-conv-rate}.

\begin{proof}[Proof of Theorem \ref{dream-conv-rate}]
Recall that equation \eqref{diff-inequ-norm-h} from the proof of Proposition \ref{lemma-fast-C-0} implies that the weighted norm of the difference of the solutions $u(x,t)\coloneqq t^{-C}|h(t)|^2\coloneqq t^{-C}|\ti g_2(t)-\ti g_1(t)|^2_{\delta}$ satisfies for some large positive constant $C$ on $\B(0,R)\times(0,T)$,
\begin{equation}\label{diff-inequ-aux-u}
\partial_tu\leq \partial_a\left(\tilde{h}^{ab} \partial_b u\right).
\end{equation}
The next step consists in localizing the previous differential inequality \eqref{diff-inequ-aux-u} in order to apply Theorem \ref{lemma-fund-sol-est}. To do so, let $\psi:\R^n\rightarrow[0,1]$ be any smooth cut-off function with support in $\B(0,R)$ such that $\psi\equiv1$ on  $\B(0,R/2)$ and $|D\psi|\leq c/R$ on $\B(0,R)$. Then the function $\psi u$ satisfies on $\R^n\times(0,T)$:
\begin{equation*}
\partial_t\left(\psi u\right)\leq \partial_a\left(\tilde{h}^{ab} \partial_b \left(\psi u\right)\right)+\underbrace{C|D\psi||Du|+C\left(\frac{|D\psi|}{\sqrt{t}}+|D^2\psi|\right)|u|}_{\text{compactly supported in $\B(0,R)\setminus\B(0,R/2)$}},
\end{equation*}
where the coefficients $\tilde{h}^{ab}$ are extended arbitrarily by continuity on $\R^n\times(0,T)$ so that $(1-\varepsilon(n))|\xi|^2\leq \tilde{h}^{ab}(x,t)\xi_a\xi_b\leq (1+\varepsilon(n))|\xi|^2$ for all $\xi\in\R^n$. Here we have used the fact that $|D\ti g_i|$, $i=1,2$ (and hence $|D\ti h|$) is uniformly bounded by $C/\sqrt{t}$.\\

The parabolic maximum principle applied to $\psi u  $ shows that for $0<t_0<t$ and $x\in\R^n$, $(\psi u)(x,t)\leq \ell(x,t)$ where $\ell$ is the standard solution to $\partial_t\ell= \partial_a\left(\tilde{h}^{ab} \partial_b \ell\right) + S(t)$ with $\ell(t_0) = u(t_0)$ and where the source term is defined as $S(t)\coloneqq C|D\psi||Du|+C(t^{-\frac{1}{2}}|D\psi|+|D^2\psi|)|u|$: The maximum principle can be applied as $\ell(t) \geq 0$ and $\psi u(t) =0$ outside the ball $\B(0,R)$.   In particular,
\begin{equation*}
(\psi u)(x,t)\leq \int_{\R^n}K(x,t,y,t_0) (\psi u) (y,t_0)\,dy+\int^t_{t_0}\int_{\R^n}K(x,t,y,s)S(y,s)\,dy ds, \quad x\in\R^n.
\end{equation*}

     
Taking into account the definition of $S$ and the support of $\psi$ leads to:
\begin{equation}
\begin{split}\label{mean-value-u}
u(x,t)&\leq \int_{\B(0,R)}K(x,t,y,t_0) u (y,t_0)\,dy\\
&+C(R)\int^t_{t_0}\int_{\B(0,R)\setminus \B(0,R/2)}K(x,t,y,s)\left(|Du|(y,s)+\left(1+\frac{1}{\sqrt{s}}\right)|u|(y,s)\right)\,dy ds, 
\end{split}
\end{equation}
if $x\in\B(0,R/2)$ and $t_0<t$.

 Now,
  Proposition \ref{lemma-fast-C-0} applied to the time interval $[r_i,T)$ for any sequence of positive times $r_i$ with $r_i \to 0$ as $i\to \infty$, and the fact that $\tau_i := \sup_{\B(0,R)} |\ti g_2(r_i) - \ti g_1(r_i)|  \to 0$  tells us that
$|\ti g_2(t) - \ti g_1(t)| \leq   C(n,k,R) t^k$ for $t\in (0,T)$  and hence $u(\cdot,t)$ satisfies $\lim_{t\rightarrow 0^+}\left(1+t^{-\frac{1}{2}}\right)u(\cdot,t)=0$ by replacing $R$ with $R/2^k=:R_0$ with $k\in\mathbb{N}$ large enough if necessary. Similarly,  Proposition \ref{lemma-fast-C-0},  ensures that the gradient of $u$ is uniformly bounded on $\B(0,R_0)\times(0,T)$ and approaches zero uniformly as $t\downto 0$. In particular, by letting $t_0$ go to $0^+$, Inequality \eqref{mean-value-u} together with Theorem \ref{lemma-fund-sol-est} applied to $u(\cdot,t)$ imply that there exist $R_0$, $C_0$ and $\alpha_0$ depending on $n$, $\varepsilon(n)$ and $T$ such that if $x\in\B(0,R_0/4)$ and $t\in(0,T)$:
\begin{equation*}
\begin{split}
u(x,t)&\leq C_0\int_0^t\int_{\B(0,R_0)\setminus \B(0,R_0/2)}(t-s)^{-\frac{n}{2}}\exp\left\{-\alpha_0\frac{|x-y|^2}{t-s}\right\}\,dy ds\\
&\leq C_0\int_0^t\int_{\B(0,R_0)\setminus \B(0,R_0/2)}(t-s)^{-\frac{n}{2}}\exp\left\{-\alpha_0\frac{3R_0^2}{16(t-s)}\right\}\,dy ds\\
&\leq C_0R_0^n\int_0^t\tilde{s}^{-\frac{n}{2}}\exp\left\{-\alpha_0\frac{3R_0^2}{16\tilde{s}}\right\}\, d\tilde{s}=C_0R_0^2\int_{\frac{R_0^2}{t}}^{+\infty}\tau^{\frac{n}{2}-2}\exp\left\{-\alpha_0\frac{3\tau}{16}\right\}\,d\tau\\
&\leq \frac{32}{3\alpha_0}C_0R_0^2\left(\frac{R_0^2}{t}\right)^{\frac{n}{2}-2}\exp\left\{-\alpha_0\frac{3R_0^2}{16 t}\right\},
\end{split}
\end{equation*}
if $t\leq T_0=\min\{T,\varepsilon_0(n)\alpha_0R_0^2\}$, where $\varepsilon_0(n)$ denotes a positive constant depending on $n$ only.
Indeed, an integration by parts shows that:
\begin{equation*}
\begin{split}
\int_{\frac{R_0^2}{t}}^{+\infty}\tau^{\frac{n}{2}-2}\exp\left\{-\alpha_0\frac{3\tau}{16}\right\}\,d\tau&=\frac{16}{3\alpha_0}\left(\frac{R_0^2}{t}\right)^{\frac{n}{2}-2}\exp\left\{-\alpha_0\frac{3	R_0^2}{16t}\right\}\\
&\quad+\frac{8(n-4)}{3\alpha_0}\int_{\frac{R_0^2}{t}}^{+\infty}\tau^{\frac{n}{2}-3}\exp\left\{-\alpha_0\frac{3\tau}{16}\right\}\,d\tau\\
&\leq \frac{16}{3\alpha_0}\left(\frac{R_0^2}{t}\right)^{\frac{n}{2}-2}\exp\left\{-\alpha_0\frac{3R_0^2}{16t}\right\}\\
&\quad+\frac{8t|n-4|}{3\alpha_0R_0^2}\int_{\frac{R_0^2}{t}}^{+\infty}\tau^{\frac{n}{2}-2}\exp\left\{-\alpha_0\frac{3\tau}{16}\right\}\,d\tau,
\end{split}
\end{equation*}
which implies the expected result by absorption if $t$ is chosen small enough compared to $n$, $\alpha_0$ and $R_0$.

By unravelling the definition of $u$ in terms of the norm of the difference of the solutions to $\delta$-Ricci DeTurck flow, this proves the expected exponential decay. 
\end{proof}

Combining Corollary \ref{cor-pol-main-thm-version} and Theorem \ref{dream-conv-rate} leads to the proof of Theorem \ref{thm:main.1}: 

\begin{proof}[Proof of Theorem \ref{thm:main.1}]
Thanks to  Corollary \ref{cor-pol-main-thm-version}, there exist solutions $(\tilde{g}_1(t))_{t\in(0,\hat{T})}$ and $(\ti g_2(t))_{t\in(0,\hat{T})}$ to $\de$-Ricci-DeTurck flow associated to $(g_i(t))_{t\in(0,\hat{T})}$, $i=1,2$, which satisfy all the conditions stated in Theorem \ref{thm:main.1} but the exponential decays stated in \eqref{exp-dec-main-thm} and \eqref{exp-dec-main-thm-higher-cov-der}.

Now, we can apply Theorem \ref{dream-conv-rate} to the two aforementioned solutions to $\de$-Ricci-DeTurck flow in order to guarantee an exponential convergence rate as expected in \eqref{exp-dec-main-thm}.
As for \eqref{exp-dec-main-thm-higher-cov-der}, we invoke interpolation inequalities \eqref{int-inequ-easy} as used in the proof of Proposition \ref{lemma-fast-C-0}
 together with [\eqref{pol-close-sec-5}, Corollary \ref{cor-pol-main-thm-version}] and the previously established $C^0$ exponential bound on $|\ti g_1(t) -\ti g_2(t)|_{\de}$.
\end{proof}

\section{Almost Ricci-pinched expanding gradient Ricci solitons} \label{alm-ric-pin-sec}
 
In this section, we prove that if an expanding gradient Ricci soliton has non-negative Ricci curvature and if it is almost Ricci-pinched in a suitable sense then it is Euclidean.
Recall that a triple $(M^n,g,\nabla^gf)$ is an expanding gradient soliton if $(M^n,g)$ is a smooth Riemannian manifold, $f:M \to \R$ is  smooth and
$ \nabla^{g,2}f= \Rc(g) +\frac{1}{2}g.$
We use the following definition, compare \cite[Definition $1.1$]{Der-Asy-Com},  to describe exponential closeness of an  expanding gradient Ricci soliton to a smooth cone at infinity. 

We gather well-known soliton identities holding on a self-similar expander:
\begin{lemma}\label{id-EGS}
Let $(M^n,g,\nabla^g f)$ be an expanding gradient Ricci soliton. Then the trace and first order soliton identities are:
\begin{equation*}
\begin{split}
\Delta_g f &= \Sc_g+\frac{n}{2}, \\
\nabla^g \Sc_g+ 2\Rc(g)(\nabla^g f)&=0, \\
|\nabla f|^2_g+\Sc_g&=f+cst .
\end{split}
\end{equation*}
\end{lemma}

See for instance \cite[Chapter $4$, Section $1$]{CLN} for a proof.

\begin{defn}\label{defn-egs-norm}
An expanding gradient Ricci soliton $(M^n,g,\nabla^gf)$ is said to be \textit{normalized} if $f$ is normalized such that $|\nabla f|^2_g+\Sc_g=f$ on $M$. 
\end{defn}

\begin{rmk}
According to Lemma \ref{id-EGS}, a normalized expanding gradient Ricci soliton $(M^n,g,\nabla^gf)$ with nonnegative scalar curvature has a nonnegative potential function $f$.
\end{rmk}

\begin{defn}\label{defn-coni-exp}
A normalized expanding gradient Ricci soliton $(M^n,g,\nabla^g f)$ is asymptotic to a metric cone  $(C(\Sigma),g_{C(\Sigma)}:=dr^2+r^2g_{\Sigma},r\partial r/2)$ over a smooth Riemannian link $(\Sigma,g_{\Sigma})$ at an exponential rate if there exists a compact $K\subset M$, a positive radius $R$ and a diffeomorphism $\phi:M\setminus K\rightarrow C(\Sigma)\setminus \overline{B_{d(g_{C(\Sigma)})}(o,R)}$ such that
\begin{equation*}
\begin{split}
&\sup_{\partial B(o,r)}\arrowvert\nabla^k(\phi_*g-g_{C(\Sigma)})\arrowvert_{g_{C(\Sigma)}}=\textit{O}(r^{-n+k}e^{-\frac{r^2}{4}}),\quad \forall k\in \mathbb{N},\\
&f(\phi^{-1}(r,x))=\frac{r^2}{4},\quad\forall \, (r,x)\in C(\Sigma)\setminus \overline{B_{d(g_{C(\Sigma)})}(o,R)},
\end{split}
\end{equation*}
 as $r\rightarrow+\infty$.
\end{defn}

We are now in a position to state the main rigidity result of this section.

\begin{prop}\label{prop-rig-alm-ric-pin-sol}
Let $(M^n,g,\nabla^gf)$ be a normalized complete expanding gradient Ricci soliton with non-negative Ricci curvature. Assume there exist $c>0$ and $\beta\in(0,1)$ such that 
\begin{equation}
c\Sc_gg\,\leq\,\Rc(g)+\exp\left({-f^{\beta}}\right)g,\quad\text{on $M$.}\label{alm-ric-pin}
\end{equation}
Then,
\begin{enumerate}
\item \label{id-asy-cone}$(M^n,g,\nabla^g f)$ is asymptotic to a Ricci flat metric cone $\left(C(\Sigma),g_{C(\Sigma)},\frac{1}{2}r\partial_r\right)$ at an exponential rate in the sense of Definition \ref{defn-coni-exp}, where $\Sigma$ is diffeomorphic to $\mathbb{S}^{n-1}$.\\[-2ex]
\item \label{rig-low-dim}If $ n\in \{3,4\}$, $(M^n,g,\nabla^g f)$ is isometric to the Gaussian soliton $(\mathbb{R}^n,\delta,\frac{1}{2}r\partial_r)$.\\[-2ex]
\item \label{rig-2-pos-cur}If $g$ has $2$-non-negative curvature operator then $(M^n,g,\nabla^g f)$ is isometric to the Gaussian soliton $(\mathbb{R}^n,\delta,\frac{1}{2}r\partial_r)$.
\end{enumerate}
\end{prop}

\begin{rmk}
(1) This proposition was proved in \cite[Proposition $3.5$]{Der-Asy-Com} under the assumption that the metric is Ricci pinched, i.e.~$c \,\Sc_gg\leq \Rc(g)$ on $M$ for some positive constant $c$. See the references therein for related results.\\[1ex]
(2) The proof in dimension $3$ can be simplified from the one given in \cite{Der-Asy-Com} since in this particular dimension, the Ricci curvature controls pointwise the whole curvature tensor. Indeed, as soon as the Ricci curvature is shown to decay exponentially fast, then it follows that the norm of the full curvature tensor, as well as   the norms of the higher covariant derivatives of the curvature tensor also decay exponentially fast, in view of  Bernstein-Shi type estimates applied to the curvature tensor:  
See for example \cite[Chapter $6$]{CLN}. 
\end{rmk}
\begin{rmk}
 We also mention the paper \cite{Chan-Pak} that shows that any expanding gradient Ricci soliton coming out of Euclidean space is Euclidean and which relies on the positive mass theorem due to \cite{Sch-Yau-Pos-Mass} in all dimensions.
\end{rmk}
\begin{rmk}
The assertion \eqref{rig-2-pos-cur} from Proposition \ref{prop-rig-alm-ric-pin-sol} supports Question \ref{ques-DSS}.
\end{rmk}
We underline the fact that the rigidity part of the proof of Proposition \ref{prop-rig-alm-ric-pin-sol} relies on the rigidity of the Bishop-Gromov inequality only.

\begin{proof}[Proof of Proposition \ref{prop-rig-alm-ric-pin-sol}]
Recall that the potential function of an expanding gradient Ricci soliton with non-negative Ricci curvature satisfies:
\begin{equation*}
\nabla^{g,2}f\geq \frac{g}{2},\quad\text{on $M$.}
\end{equation*}
In particular, $f$ is strictly convex and by integrating along geodesics, $f$ is proper. More precisely, 
\begin{equation}
f(x)\geq \min_Mf+\frac{d_g(p,x)^2}{4},\quad \forall\, x\in M,\label{low-bd-f}
\end{equation}
where $p$ is the unique critical point of $f$ which necessarily corresponds to the minimum of $f$.   In particular, by the Morse lemma, this implies that the level sets $\{f=t\}$ of $f$, for $t>\min_Mf$, are all diffeomorphic to $\mathbb{S}^{n-1}$ and that $M$ is diffeomorphic to Euclidean space $\R^n$.

 By the first order identity $|\nabla^gf|^2_g+\Sc_g=f$ that holds on $M$ by Definition \ref{defn-egs-norm}, one also gets $|\nabla^gf|^2_g\leq f$ on $M$ which implies:
\begin{equation}
f(x)\leq \left(\sqrt{\min_Mf}+\frac{d_g(p,x)}{2}\right)^2,\quad \forall\, x\in M.\label{upp-bd-f}
\end{equation}

Let us consider the Morse flow associated to $f$
  which is well-defined outside a compact set thanks to the first order identity mentioned above and the lower bound \eqref{low-bd-f} on $f$, i.e. 
\begin{equation}
\partial_t\phi_t=\left(\frac{1}{|\nabla^gf|^2_g}\nabla^gf\right)\circ\phi_t,\quad \phi_{t_0}=\operatorname{Id}_M,\quad \text{for $t\geq t_0\geq \min_Mf+1$.}
\end{equation}

Then, $f(\phi_t(x))=f(x)+t-t_0=t$, $t\geq t_0$ and $x\in f^{-1}(\{t_0\})$. Now, thanks to the contracted Bianchi identity, $\nabla^g \Sc_g+2\Rc(g)(\nabla^gf)=0$ on $M$ from Lemma \ref{id-EGS}. Contracting this identity with $\nabla^gf$ and   invoking \eqref{alm-ric-pin},  we see
\begin{equation*}
2c \Sc_g|\nabla^gf|^2_g+g(\nabla^g\Sc_g,\nabla^gf)\leq 2e^{-f^{\beta}}|\nabla^gf|^2_g,\quad\text{on $M$}.
\end{equation*}
Reinterpreting the previous differential inequality in terms of the Morse flow associated to $f$, one gets:
\begin{equation*}
\frac{\partial}{\partial t}\left(e^{2ct} \Sc_g(\phi_t(x))\right)\leq 2 e^{-t^{\beta}+2c t},\quad t\geq t_0. 
\end{equation*}
In particular, by integrating from $t_0$ to $t$:
\begin{equation*}
\begin{split}
e^{2ct} \Sc_g(\phi_t(x))&\leq e^{2ct_0} \Sc_g(\phi_{t_0}(x))+2 \int_{t_0}^te^{-s^{\beta}+2c s}\,ds\\
&= e^{2ct_0} \Sc_g(x)+2 \int_{t_0}^te^{-s^{\beta}+2c s}\,ds\\
&\leq e^{2ct_0} \Sc_g(x)+ \frac{1}{c}e^{2ct-t^{\beta}},
\end{split}
\end{equation*}
if $t_0$ is chosen sufficiently large with respect to $c$ and $\beta$. Indeed, since $\beta\in(0,1)$ and $c>0$, an integration by parts shows,
\begin{equation*}
\begin{split}
\int_{t_0}^te^{-s^{\beta}+2c s}\,ds&=\left[\frac{e^{2cs-s^{\beta}}}{2c}\right]_{t_0}^t+\frac{\beta}{2c}\int_{t_0}^ts^{\beta-1}e^{-s^{\beta}+2cs}\,ds\\
&\leq \frac{e^{2cs-s^{\beta}}}{2c}+\frac{\beta t_0^{\beta-1}}{2c}\int_{t_0}^te^{-s^{\beta}+2c s}\,ds,
\end{split}
\end{equation*}
which leads by absorption to $\int_{t_0}^te^{-s^{\beta}+2c s}\,ds\leq \frac{1}{c}e^{2ct-t^{\beta}}$ if $t_0$ is large enough so that $\beta/(2c)t_0^{\beta-1}\leq 1/2.$

 As a first conclusion, one gets $\Sc_g(\phi_t(x))\leq e^{-2c(t-t_0)} \Sc_g(x)+  \frac{1}{c}e^{-t^{\beta}}$ for any $x\in f^{-1}(\{t_0\})$. By definition of the Morse flow, this gives 
\begin{equation*}
\Sc_g(y)\leq \max_{x \in f^{-1}(\{t_0\})} \Sc_g(x) \cdot  e^{2c(t_0-f(y))} + \frac{1}{c}e^{-f(y)^{\beta}},
\end{equation*}
 for any $y\in M$ such that $f(y)\geq t_0\geq \min_Mf+1$.
 
This shows that $ \Sc_g$ decays as fast as $e^{-f^{\beta}}$ at infinity, i.e. there exists a positive constant $C$ such that for all $y\in M$, $\Sc_g\leq Ce^{-f(y)^{\beta}}$. So does the Ricci curvature since $g$ has non-negative Ricci curvature.\\

The rest of the proof of (\ref{id-asy-cone}) and (\ref{rig-low-dim}) is \cite[Proposition $3.5$]{Der-Asy-Com}   verbatim. \\

In order to prove (\ref{rig-2-pos-cur}), we invoke (\ref{id-asy-cone}) to ensure that the   link is diffeomorphic to $\mathbb{S}^{n-1}$ and is endowed with an Einstein metric such that the cone over it has $2$-non-negative curvature operator. The results of \cite{Bohm-Wilking} then shows the expected rigidity result: the cone is isometric to Euclidean space which forces the soliton to be the expanding Gaussian soliton by Bishop-Gromov volume comparison.
\end{proof}

\section{Compactness of Ricci-pinched manifolds in dimension three}\label{sec-comp-ric-pinch}

We consider $(M^3,g_0)$ a smooth, complete Riemannian manifold with uniformly pinched, non-negative, bounded Ricci curvature. Assuming by contradiction that $(M^3,g_0)$ is non-compact and non-flat from now on, the work of Lott \cite{Lott-Ricci-pinched} ensures there exists a smooth, complete Ricci flow $(M^3, g(t))_{t\in [0,\infty)}$ with $g(0)=g_0$, satisfying the following properties:
\begin{enumerate}
\item \label{type-iii-lott}\cite[Propositions $2.1$, $2.13$]{Lott-Ricci-pinched}   there exists $C>0$ such that 
\begin{equation*}
 |\Rm(g(t))| \leq \frac{C}{t}\,, \text{ for all } t > 0,\   
\end{equation*}
\item \label{ric-pinch-lott}\cite[Theorem $9.4$]{ham3D}, \cite[Proposition $2.1$]{Lott-Ricci-pinched} the flow has $\Rc(g(t)) >0$ and is uniformly Ricci pinched, i.e.~there exists $\lambda_0>0$ such that 
 \begin{equation*}
 \Rc(g(t)) \geq \lambda_0 \Sc_{g(t)} \cdot g(t)\,,  \text{ for all } t \geq 0, 
 \end{equation*}
\item \label{unif-non-coll-lott}  \cite[Propositions $1.5$, $3.1$ and $4.1$]{Lott-Ricci-pinched} the flow is uniformly volume non-collapsed, i.e.~ there exists $v_0>0$ such that 
 \begin{equation*}
 \vol_{g(t)} (B_{g(t)}(x,r)) \geq v_0 r^3\,,
 \end{equation*}
  for all $x \in M$, $r>0$ and $t>0$. 
  \end{enumerate}



By performing a blow-down we can instead assume that the flow $(M^3,g(t))$ is only smooth for $t>0$ and converges in the metric  sense uniformly as $t \to 0$ to a metric space  $(M, d_0)$,  since $d_0  \geq d_t \geq d_0 -c_0\sqrt{t}$ for all $t>0$ by Lemma \ref{SiTopThm1}. 
In fact  $(M, d_0) $ is a  metric cone  over a 2-dimensional metric space   $(\Sigma^2,\tilde d_0)$  from Cheeger-Colding \cite[Theorem $9.79$]{Cheeger_notes}, it is furthermore uniformly volume non-collapsed. 
Finally, by Lemma \ref{SiTopThm1}, $(M,d_0)$ is homeomorphic to $(M,d_t)$ for all $t>0$.


\begin{lemma}\label{lemma-no-cone-pt}
Let $o$ be the tip of the cone $(M, d_0).$  
Then for $r'>r>0$, $(A_{d_0}(o,r,r')\coloneqq  B_{d_0}(o,r') \backslash \overline{B_{d_0}(o,r)},d_0) \subseteq M$ is Reifenberg (and hence uniformly Reifenberg due to Lemma \ref{UniformReifenbergLemma}).
\end{lemma}

\begin{proof}

We first show that every tangent cone to $(M, d_0)$ at any point away from the vertex $o$ is 3-dimensional Euclidean space.

 Let $p \in M \setminus \{o\} $ and consider a blow-up sequence of the flow $(M^3,g(t))_{t\in(0,\infty)}$ around $(p,0)$. Since $(M,d_0)$ is a cone, this converges to a flow $(M',(g'(t))_{t\in (0,\infty)}),$ with the  same properties \eqref{type-iii-lott}, \eqref{ric-pinch-lott}, \eqref{unif-non-coll-lott} 
  that the original solution had, 
   coming out of a cone $(\Cone', d_0')$ which splits a line, i.e.~$(\Cone', d_0')$ is isometric to $(\Cone'' \times \mathbb{R}, d_0'' \oplus ds^2)$. But then Proposition \ref{prop:Hochard-splitting} implies that $(M',(g'(t))_{t\in (0,\infty)})$ splits a  Euclidean factor. Since $(M',(g'(t))_{t\in (0,\infty)})$ is still uniformly Ricci pinched and 3-dimensional, this implies that $(M',(g'(t))_{t\in (0,\infty)})$ is static flat Euclidean $(\mathbb{R}^3, \delta)$ and thus $(\Cone', d_0')$ is isometric to $(\mathbb{R}^3, d(\delta))$.
 
  Considering $\overline{A_{d_0}(o,r,r')} \subseteq X$, Lemma \ref{UniformReifenbergLemma} then shows that $(A_{d_0}(o,r,r'),d_0)$ is uniformly Reifenberg.
 \end{proof}

We consider the induced distance $\tilde d_0$ on the link $\Sigma$.

\begin{lemma}\label{lemma-Alex-cone-0}
 The metric space $(\Sigma, \tilde d_0)$ is an Alexandrov space of curvature at least $1$.
\end{lemma}

The proof of Lemma \ref{lemma-Alex-cone-0} uses the notions of $\operatorname{RCD}^*$ and $\operatorname{RCD}$ spaces. For an overview of these notions, we invite the reader to read the introduction of \cite{Ketterer15} and the references therein together with the recent survey \cite{Gigli-Degiorgi-Gromov}.

\begin{proof}
Observe first that according to \cite[Theorem $7.2$]{Gig-Mon-Sav}, $(\Cone(\Sigma), d_0,\Haus^3_{d_0})$ is  an $\textit{$\operatorname{RCD}$}^*(0,3)$ space, as it is the limit (in the Gromov-Hausdorff and distance sense) of smooth spaces with non-negative Ricci curvature and uniform Euclidean volume growth, that is $\vol_{g(t)}(B_{g(t)}(p_0,r) )\geq v_0\,r^3$, for all $r>0$ and $t>0$.
Therefore, the  result of Ketterer \cite{Ketterer15} implies that $(\Sigma, \tilde d_0,\Haus^{2}_{\ti d_0})$ is an $\textit{$\operatorname{RCD}$}^*(1,2)$ space. Now, by \cite[Corollary $13.7$]{Cav-Mil}, $(\Sigma, \tilde d_0,\Haus^{2}_{\ti d_0})$ is an $\textit{$\operatorname{RCD}$}\,(1,2)$ space. But then the result of Lytchak-Stadler \cite{Lyt-Sta-Ric-2-dim} implies that $(\Sigma, \tilde d_0)$ is a two dimensional  Alexandrov space of curvature at least $1$.
\end{proof}

\begin{lemma}
 There exists a smooth structure on $\Sigma$ and a sequence of  metrics $(g_i)_{i\in \N}$ on $\Sigma$, which are smooth with respect to this structure,  such that $(\Sigma, g_i)$ has sectional curvature greater than one, and $(\Sigma, d(g_i))$ converges to  $(\Sigma, \tilde{d}_0)$  as $i\to \infty$ in the    Gromov-Hausdorff sense. In particular, $\Sigma$  is a $2$-sphere.
\end{lemma}

\begin{proof}
According to \cite[Corollary $10.10.3$]{BBI} and the references therein and thanks to Lemma \ref{lemma-Alex-cone-0}, $\Sigma$ is a topological surface without boundary which therefore admits a unique smooth structure: compare also the proof of \cite[Proposition $6.4$]{Lott-Ricci-pinched}. Moreover, since $(M, d_0) $ is a manifold and is a metric cone over $\Sigma$, $\Sigma$ is orientable. Now, \cite[Lemmata $2.3$, $2.4$]{Itoh-Rouyer-Vilcu} ensure that the set of smooth Riemannian metrics on $\Sigma$ with curvature bounded from below by $1$ is dense in the space of Alexandrov spaces of curvature at least $1$ with respect to the Gromov-Hausdorff topology. In particular, the Gauss-Bonnet formula applied to one member of the sequence guarantees that $\Sigma$ is a $2$-sphere and that there exists a sequence of smooth metrics $(g_i)_{i\in \N}$ on $\Sigma$ with $K_{g_i}\geq 1$ such that $(\Sigma, d(g_i))_{i\in \N}$ converges to  $(\Sigma, \tilde{d}_0)$ in the Gromov-Hausdorff sense. Moreover, by stretching the sequence $(g_i)_{i\in \N}$ by a factor $(1-\varepsilon_i)$ where $\varepsilon_i$ tends to $0$, we can assume that $K_{g_i}>1$. This ends the proof of the lemma.
\end{proof}

\begin{rmk}
Some of the proofs of the results cited in the paper  \cite[Lemmata $2.3$, $2.4$]{Itoh-Rouyer-Vilcu} may be replaced  by  alternative  proofs from \cite{Creutz-Romney-Tri}.
For example, an alternative proof to the  
  the result of Alexandrov and Zalgaller  from \cite{Alek-Zalg-Book}, of the fact that 
   any 2-dimensional manifold with bounded integral curvature
can be decomposed into a finite union of geodesic triangles whose interiors are
disjoint, can be found in \cite{Creutz-Romney-Tri}.
\end{rmk}
\begin{lemma}\label{exist-barrier-exp}
  There exists a smooth self-similar expanding solution $(\mathbb{R}^3, g_e,\nabla^{g_e}f_e)$ with bounded non-negative curvature operator which is uniformly volume non-collapsed, such that its tangent cone at infinity is $\Cone = (\Cone(\Sigma), d_0)$, i.e. for some (hence all) point $p\in\R^3$, 
$(\R^3, d_{\ep g_e},p)   \to  (\Cone(\Sigma), d_0,o)$ in the pointed Gromov-Hausdroff sense as $\ep \downto 0$. 
 \end{lemma}

\begin{proof}
 By \cite{Der-Smo-Pos-Cur-Con}, alternatively \cite{SchulzeSimon}, after approximating the cone by a smooth space with non-negative curvature operator as in \cite[Proposition 3.2.6]{PhDschlichting},  we know that there exists smooth expanding gradient Ricci solitons  $(\mathbb{R}^3, g^i_{e}(t))_{t\in (0,\infty)}$ with positive curvature operator with asymptotic cone  $(\Cone(\Sigma), d^i_C,o)$,  where $d_C^i $ denotes the cone metric induced by the Riemannian distance $d(g_i)$ on $\Sigma$. 
 By continuity of Hausdorff measure under pointed Gromov-Hausdorff convergence for sequences of Alexandrov spaces with curvature uniformly bounded from below (\cite[Theorem $10.10.10$]{BBI}), the asymptotic volume ratios 
 \begin{equation*}
 \AVR(g^i_e)\coloneqq \lim_{r\rightarrow+\infty}\vol_{g^i_e}(B_{g^i_e}(p_i,r))/r^3=\frac{1}{2}\vol_{g_i}(\Sigma),\quad p_i\in \R^3,
 \end{equation*}
   are uniformly bounded from below by a positive constant $V_0$ say. Now, this implies that there is a uniform bound, $|\Rm(g^i_e(t))|_{g^i_e(t)} \leq \frac{B_0(V_0)}{t}$  on the curvature of the metrics $g^i_e(t),$ where $B_0(V_0)$ is a constant depending on $V_0,$ as explained in \cite{SchulzeSimon} or in the proof of \cite[Theorem $4.7$]{Der-Asy-Com}. This allows us to take a limit in the smooth Cheeger-Gromov topology as $i\to \infty$ to obtain an expanding gradient Ricci soliton $(\R^3,g_e)$. Moreover, as explained in the same aforementioned references, the asymptotic cone of $(\R^3,g_e)$ must be $(\Cone(\Sigma),d_0)$. 
\end{proof}
We are now in a position to prove Hamilton's conjecture in dimension $3$:

\begin{proof}[Proof of Theorem \ref{thm:main.2}]
 Assume $(M^3,g)$ is a complete Riemannian manifold with non-negative uniformly pinched bounded Ricci curvature. Assume furthermore that $M^3$ is non-compact and non-flat. 
  As explained in the beginning of this section this implies that there exists an immortal solution to Ricci flow, $(M^3_1,g_1(t))_{t\in (0,\infty)}$  satisfying \eqref{type-iii-lott}, \eqref{ric-pinch-lott} and \eqref{unif-non-coll-lott}, whose singular initial condition, achieved in the distance sense, is by Lemma \ref{lemma-Alex-cone-0} a metric cone $(\Cone,d_1(0),o)$ over an Alexandrov space $(\Sigma, \ti d_1(0))$ with curvature larger than or equal to $1$. By Lemma \ref{lemma-no-cone-pt}, the cone $(\Cone,d_1(0),o)$ satisfies \eqref{ReifenbergCond} away from $o$. Moreover, Lemma \ref{exist-barrier-exp} ensures the existence of a self-similar solution $(M_2^3,g_2(t),\hat o)_{t\in (0,\infty)}$ converging to  $(\Cone,  d_2(0),\hat o)$ in the distance sense as $t$ goes to $0$ and which is an expanding gradient Ricci soliton with non-negative curvature operator. Here $\hat o$ is the unique critical point of the potential function associated to the self-similar solution $(M_2^3,g_2(t))_{t\in (0,\infty)}$ which corresponds to the minimum of this function. Pointing at $\hat o$ is crucial as $(M_2^3,g_2(t)=t\Phi_t^*g_2,\hat o)$ is isometric to $(M_2^3,t\cdot g_2,\hat o)$ as pointed metric spaces for each $t>0$ through the diffeomorphisms $(\Phi_t)_{t>0}$ generated by $-\nabla^gf/t$ such that $\Phi_1=\operatorname{Id}_M$. Moreover, the cone $(\Cone,  d_2(0),\hat o),$ is isometric to $(\Cone,d_1(0),o)$ through some map $\psi_0$, where  $(\Cone,   d_2(0),\hat o)$ is  a cone over $(\Sigma,   \ti d_2(0))$  where
  $(\Sigma,  \ti d_2(0) )$ is isometric to  $(\Sigma,  \ti d_1(0) )$.
 
 Let $x_0\in \Cone$ such that $d_1(0)(o,x_0)\geq 2$ so that $B_{d_1(0)}(x_0,1)\subseteq \Cone\setminus\{o\}$. Then Theorem \ref{thm:main.1} guarantees that there exist $0<T_0<1,$ $0<R_0<1$ and solutions $F_1: B_{d_1(0)}(x_0,\frac{3}{2}  R_0) \times (0,T_0) \to \R^3,$  $F_2: B_{d_2(0)}(\psi_0(x_0),\frac{3}{2} R_0) \times (0,T_0) \to \R^3$ such that $ \ti g_1(t)=  (F_1(t))_{*}g_1(t)$  and $ \ti g_2(t)=  (F_2(t))_{*}g_2(t)$, $t\in(0,T_0]$, are solutions to the  $\de$-Ricci-De Turck flow
 on  $\B(0,R_0)$
   $i=1,2$, which are  $\varepsilon_0$-close to the $\delta$ metric and such that there exist $C_{0}>0$  such that if $t\in(0,T_{0}]$ and $k=0,1,2$:
  \begin{equation}
|D^k\left(\ti g_1(t) -\ti g_2(t)\right)|_{\de}\leq \exp\left(-\frac{C_0}{t}\right),\quad \text{on $\B(0,R_0).$}\label{conv-rate-exp-two-sol-ham}
\end{equation}
In the following $C_0$ denotes a   constant, which  may change from line to line, but remains positive. 
Remembering that $(\ti g_1(t))_{t\in(0,T_0)}$ is uniformly Ricci-pinched, we see that     \eqref{conv-rate-exp-two-sol-ham} implies: 
\begin{equation*}
\Rc(\ti g_2(t))\geq \lambda_0\Sc_{\ti g_2(t)}\ti g_2(t)-\exp\left(-\frac{C_0}{t}\right)\ti g_2(t),\quad t\in(0,T_{0}],
\end{equation*}
 on $\B(0,R_0)$ 
for some uniform $\lambda_0>0$.

Since $\ti g_2(t)=(F_2(t))_{\ast}g_2(t)$ with $F_2(t)$ $(1\pm\varepsilon_0)$ bi-Lipschitz, the same estimate holds on $B_{d_2(0)}\left(\psi_0(x_0) = \hat x_0,R_0/2\right):$
\begin{equation*}
\Rc(g_2(t))\geq \lambda_0\Sc_{g_2(t)} g_2(t)-\exp\left(-\frac{C_0}{t}\right) g_2(t),\quad t\in(0,T_0].
\end{equation*}

 Since the solution $g_2(t)=t\Phi_t^*g_2(1)$ with $\partial_t\Phi_t=-t^{-1}\nabla^{g_2}f\circ \Phi_t$ and $\Phi_t|_{t=1}=\operatorname{Id}_M$,   where $f:M\rightarrow\R$ is the associated potential function, the previous estimate can be further simplified to:
\begin{equation*}
\Rc(g_2(1))\geq \lambda_0\Sc_{g_2(1)} g_2(1)-\exp\left(-\frac{C_0}{t}\right) g_2(1),
\end{equation*}
on $\Phi_t\left(B_{d_2(0)}\left(\hat x_0,R_0/2\right)\right)$, $t\in(0,T_{0}]$. 

Without loss of generality, we can assume $\hat x_0  $ to lie in the annulus $A_{d_2(0)}(\hat o,2,4)=\{x\in\Cone\,|\, 2\leq d_2(0)(\hat o,x)\leq 4\}$ so that the previous estimate holds on $\Phi_t(A_{d_2(0)}(\hat o,1,5))$, $t\in(0,T_{0}]$, by reducing $T_{0}$ if necessary by a compactness argument. We can also assume $\hat o$ to be the unique critical point of $f$. 

 We claim that there exists $c_1>0$ and $c_2>0$ such that for $t\in(0,T_{0}]$, 
\begin{equation}\label{inclusion-final}
\left[\frac{c_1^{-1}}{t},\frac{c_1}{t}\right]\subseteq f\left(\Phi_t\left(A_{d_2(0)}(\hat o,1,5)\right)\right)\subseteq\left[\frac{c_2^{-1}}{t},\frac{c_2}{t}\right].
\end{equation}

 Indeed, by invoking the bounds of $f$ in terms of $d(g_2(1))(\hat o,\cdot)$ given in \eqref{low-bd-f} and \eqref{upp-bd-f} and since $\hat o$ is fixed by the diffeomorphisms $\Phi_t$, $t\in(0,T_0]$, we see, using $g_2(t)=t\Phi_t^*g_2(1),$ that 
\begin{equation*}
\min_M f+\frac{d(g_2(t))(\hat o,x)^2}{4t}\leq f(\Phi_t(x))\leq \left(\sqrt{\min_Mf}+\frac{d(g_2(t))(\hat o,x)}{2\sqrt{t}}\right)^2,\quad t>0.
\end{equation*}

On the other hand, thanks to distance distortion estimates from \eqref{DistCond},
\begin{equation*}
\begin{split}
\sqrt{f(\Phi_t(x))}&\leq \sqrt{\min_Mf}+\frac{d_2(0)(\hat o,x)}{2\sqrt{t}},\\
\frac{d_2(0)(\hat o,x)}{2\sqrt{t}}&\leq \frac{d(g_2(t))(\hat o,x)+c_0\sqrt{t}}{2\sqrt{t}}\leq \sqrt{f(\Phi_t(x))-\min_M f}+\frac{c_0}{2},\quad t>0,
\end{split}
\end{equation*}
where $c_0$ is a positive constant uniform in space and time.

In particular, if $\gamma\in f\left(\Phi_t\left(A_{d_2(0)}(\hat o,1,5)\right)\right)$ then, for all $t>0$,
\begin{equation*}
\frac{1}{2\sqrt{t}}-\frac{c_0}{2}\leq \sqrt{\gamma-\min_M f}\leq \sqrt{\gamma}\leq \sqrt{\min_Mf}+\frac{5}{2\sqrt{t}}.
\end{equation*}
This implies the second inclusion in \eqref{inclusion-final} for a suitable positive $c_2$ for all $t\in(0,T_0]$ by reducing $T_0$ if necessary.

Now, if $\gamma\in \left[\frac{c_1^{-1}}{t},\frac{c_1}{t}\right]$ for some positive constant $c_1$, since $f$ is proper and bounded from below, there must be some point $y\in M$ such that $f(y)=\gamma$. Moreover, since $\Phi_t$ is a diffeomorphism for all $t>0$, $y=\Phi_t(x)$ for some unique $x\in M$ if $t>0$ is fixed and it must satisfy:
\begin{equation*}
\begin{split}
\sqrt{\gamma}- \sqrt{\min_Mf}&\leq\frac{d_2(0)(\hat o,x)}{2\sqrt{t}}\leq \sqrt{\gamma-\min_M f}+\frac{c_0}{2}.
\end{split}
\end{equation*}
This implies the first expected inclusion in \eqref{inclusion-final} for a suitable positive constant $c_1$ for all $t\in(0,T_0]$ by reducing $T_0$ again if necessary.
 This shows that \eqref{alm-ric-pin} in Proposition \ref{prop-rig-alm-ric-pin-sol} holds, i.e.
\begin{equation*}
\Rc(g_2(1))\geq \lambda_0\Sc_{g_2(1)} g_2(1)-\exp\left(-C_0f\right) g_2(1),\quad \text{on $M$.}
\end{equation*}
As a conclusion, [\ref{rig-low-dim}, Proposition \ref{prop-rig-alm-ric-pin-sol}] implies that $(M^3,g_2)$ is isometric to Euclidean $3$-space which in turn implies that the cone $(\Cone,d_2(0))$ is itself isometric to Euclidean $3$-space. Therefore, the asymptotic volume ratio $\operatorname{AVR}(g_2)\coloneqq \lim_{r\rightarrow +\infty}r^{-3}\vol_{g_2}(B_{g_2}(p,r))$ (independent of the base point $p\in M$) equals the volume $\omega_3$ of the unit ball $\B(0,1)\subseteq\R^3$. Now $\operatorname{AVR}(g_2)=\operatorname{AVR}(g_1)$ since the two corresponding solutions share the same cone at $t=0$ (together with Colding's continuity of the volume), one gets $\operatorname{AVR}(g_1)=\omega_3$ which implies that $(M^3,g_1)$ is isometric to Euclidean $3$-space by the rigidity part of Bishop-Gromov volume comparison. This ends the proof of Theorem \ref{thm:main.2}.
\end{proof}
\newpage 

\begin{appendix}
 \section{Volume and distance convergence}
 We recall the following result of Simon-Topping. 

\begin{lemma}[\textrm{Simon-Topping, \cite[Lemma 3.1]{SiTo2}}] \label{SiTopThm1}
Let $(M^n,g(t))_{t\in (0,T)}$, $T\leq 1$, be a smooth Ricci flow, satisfying $\Rc(g(t)) \geq -g(t) , \quad |\Rm(g(t))| \leq c_0/t,$ where 
 $M$ is connected but $(M^n,g(t))$ not  necessarily complete. Assume furthermore that $B_{g(t)}(x_0,1)$ is compactly contained in $M$ for all $t\in (0,T)$.    
 
Then $X\coloneqq  (\cap_{s \in (0,T)} B_{g(s)}(x_0,\frac 1 2))$ is non-empty and there is a well defined limiting metric $d_t \to d_0$ as $t \downto 0$, where 
\begin{eqnarray}
&&e^td_ 0  \geq d_t  \geq d_0 -\ga(n)  \sqrt{c_0 t} \ \ \text{ for all } t\in [0,T)  \text{ on } X.  \label{distestinit} 
\end{eqnarray}

Furthermore, there exists   $R= R(c_0,n)>0, S = S(c_0,n) >0$ such that
$B_{d_0}(x_0,r) \Subset \curlX \subseteq X$  
and $B_{g(t)}(x_0,r) \Subset \curlX \subseteq X $ for all $r \leq R(c_0,n)$ and $t \leq S$ where $\curlX$ is the connected component of $X$ which contains $x_0$, and the topology 
of $ B_{d_0}(x_0,r)$ induced by $d_0$ agrees with that of the set $  B_{d_0}(x_0,r)\subseteq M$ induced by the topology of $M$. 
 \end{lemma}
 
 In the setting of solutions $(M^n,g(t))_{t\in (0,T)}$ with $\Rc(g(t)) \geq -1$ and curvature bounded by $c_0/t$
 which are Reifenberg, at time zero,  one    has  a  Pseudolocality type theorem, even when $(M^n,g(t))$ is not necessarily complete.
 This is in contrast to the general setting, where completeness is required, as can be seen by the example of Giesen-Topping \cite{Gie-Top-RF-Unbded}.
 \begin{lemma}[Pseudolocality for Reifenberg regions]{\label{pseudorefienberg}}
  Let $(M^n,g(t))_{t\in (0,T)}$, $T\leq 1$, be a smooth Ricci flow, satisfying $\Rc(g(t)) \geq -g(t) , \quad |\Rm(g(t))| \leq c_0/t$, where 
 $M$ is connected but $(M^n,g(t))$ not  necessarily complete. Assume furthermore that $B_{g(t)}(x_0,1)$ is compactly contained in $M$ for all $t\in (0,T),$ and let $X\coloneqq  (\cap_{s \in (0,T)} B_{g(s)}(x_0,\frac 1 2))$ be as in \ref{SiTopThm1} endowed with the metric $d_0$.
 Assume that $ B_{d_0}(x_0,r)$ is as in Lemma \ref{SiTopThm1}, and furthermore that all tangent cones of $ B_{d_0}(x_0,r)$ are isometric to  $(\R^n,d(\de)).$ 
Let $\ep>0$ be given. Then there exists a $\hat \de>0$ depending on $\ep,n,x_0$ and   the solution, such that if 
$0<t\leq \hat\de$ then 
\begin{equation*}
|\Rm(g(t))| \leq  \frac{\ep}{t},\quad\text{ on  $ B_{d_0}\left(x_0,\tfrac r 4\right)$.}
\end{equation*} 
 \end{lemma}
 \begin{proof}
 Let $ \ep >0$ be fixed. 
Assume we find a sequence of points $y_i \in    B_{d_0}(x_0,r/4)$   
and times  $0<t_i  \to 0$ as $i \to \infty$  so that 
$t_i|\Rm(g(t_i))|(y_i)  \in [\ep,c_0].$

After taking a subsequence  we may assume without loss of generality that  $y_i \to z \in  B_{d_0}(x_0,r/2).$
We consider different  cases, Case 1.2 being the most difficult. \\[2ex]
{\bf Case 1: $d^2_0(y_i,z) \geq A t_i$ for some $A >0$ for all $i\in \N$.}\\[1ex]
We scale solutions and initial data so that $\ti d_i(0) :=  \la_i   d_i(0),$ for $\la_i := (d_0(y_i,z))^{-1}   \to \infty $ so that
$\ti d_i(0)(y_i,z) =1,$ and $(M,\ti d_i(0), z)) \to (\R^n,d(\delta),0)$    in the Gromov-Hausdorff sense  as $i\to \infty.$ 
Then  the new times  $\ti t_i:= t_i \la^2_i$ for the scaled solution  $ (M, \ti g_i(\ti t) = \la^2_i g(\la^{-2}_i\ti t))_{\ti t\in (0,\lambda_i^{2})}$ satisfy $0<\ti t_i = 
 (\ti d_i(0)(y_i,z))^{-2}\ti t_i =   (d_i(0)(y_i,z))^{-2} t_i  \leq A^{-1} < \infty.$
 In particular,  $(M,\ti d_i(0),z)  $ approaches  $(\R^n,d(\delta),0)$ in the Gromov-Hausdorff sense as $i\to \infty,$ and 
the solution $ (M, \ti g_i(  t),z)_{t\in (0,\lambda_i^{2})}$  (we denote  the variable $\ti t$ by $t$ again, for ease of reading) approaches   
a  solution $(X,h(t),z)_{t\in (0,\infty)}$  with $\Rc(h(t)) \geq 0,$ $|\Rm(h(t))|\leq c_0/t,$ and  
 \begin{equation}\label{dist-est-app-A}
 d_s \geq d_t \geq d_s  -c\sqrt{t-s}, \quad 0\leq s\leq t,
 \end{equation}  
such that  $(X,d(h(t)),z) \to  (\R^n,d(\delta),0)$ in the Gromov-Hausdorff sense  as $t\downto 0$.


Scaling the solution by $\eta_i\downto 0$ leads to solutions   also coming out of  $(\R^n,d(\de))$, and hence the volume convergence theorem of Cheeger-Colding,   \cite[Theorem 9.45]{Cheeger_notes}  with the distance estimates \eqref{dist-est-app-A}, shows for the solution $(X,h(t))$  for any  fixed $t>0$, that  
 $ r^{-n}\vol_{h(t)}(B_{h(t)}(x,r)) \to \omega_n$ as $ r\to \infty,$ and so Bishop-Gromov volume comparison tells us 
  that  $r^{-n}\vol_{h(t)}(B_{h(t)}(x,r))= \omega_n$   for all $r>0$ which implies the solution is isometric to $(\R^n,d_{\de})$ (by the equality case of Bishop-Gromov volume comparison).  
  From the construction, we know that $\ti t_i \leq A^{-1}$.

Note that  $\ti t_i|\Rm(\ti g_i(\ti t_i))|(y_i)  = t_i|  \Rm(  g( t_i))|(y_i) \in [\ep_0,c_0].$\\[2ex] 
{\bf Case 1.1: For a subsequence of $ \ti t_i$ we have   $ \ti t_i \to \tau>0,$ as $i\to \infty$. }\\[1ex]
Then, 
\begin{equation*}
 |  \Rm(\ti g_i(\ti t_i))|(y_i)=   \frac{1}{\ti t_i} ( \ti t_i |  \Rm(\ti g_i(\ti t_i))|(y_i) ) \in   \left[ \ep_0/\ti t_i, c_0/ \ti t_i\right]
 \to  \left[\ep_0/\tau, c_0/\tau \right],
 \end{equation*}
leads to a contradiction,  since  the limiting solution is flat.\\[2ex]
{\bf Case 1.2: $\ti t_i \to 0$ as $i\to \infty$. }\\[1ex] 
For $\eta>0$, we can find $ N(\eta)\in \N$ large  so that 
 \begin{equation*}
 \vol_{\ti g_i(\eta)}(B_{\ti g_i(\eta)}(y_i,1)) \geq \omega_n (1-\eta),\quad\text{ for all $i \geq N(\eta)$.}
 \end{equation*}
  In the following, $\gamma(\eta)>0$ can change from line to line, but always satisfies  $\gamma(\eta) \to 0$ as $\eta \to 0$.
 Hence, with the help of Bishop-Gromov volume comparison and the fact that  $\Rc(\ti g_i(t)) \geq-\ep(i)$ for all $t>0$  (in particular for $t=\eta$), where $\ep(i) \to 0$ as $i\to \infty,$ we see 
 $\vol_{\ti g_i(\eta)}(B_{\ti g_i(\eta)}(y_i,\ell)) \geq \omega_n (1-\gamma(\eta))\ell^n$ for all $i \geq N(\eta)$, $\ell \in (0,1),$ after adjusting $N(\eta)$ if necessary.
 Using the distance estimates \eqref{dist-est-app-A},  we know for arbitrary $r\in (0,\eta)$ that 
 $B_{\ti g_i(r)}(y_i,1 ) \supseteq B_{\ti g_i(\eta)}(y_i,1-\gamma(\eta))$
 and hence, 
 \begin{eqnarray*}
 &&   \vol_{\ti g_i(\eta)}(B_{\ti g_i(r)}(y_i,1)) \geq \vol_{\ti g_i(\eta)}(B_{\ti g_i(\eta)}(y_i,1-\gamma(\eta) ))   \geq \omega_n (1-\gamma(\eta)) 
  \end{eqnarray*}
   for   $i\geq   N$.
For arbitrary $r\in (0,\eta)$, let $\Omega:= B_{\ti g_i(r)}(y_i,1) $ for $i\geq   N(\eta).$ 
Using $\Sc_{\ti g_i(t)} \geq -\ep(i),$  with $\ep(i) \to 0$ as $i\to \infty$ and 
\begin{equation*}
\frac{d}{dt} \int_{\Omega}\,d\mu_{\ti g_i(t)}= -\int_{\Omega} \Sc_{\ti g_i(t)}\,d\mu_{\ti g_i(t)} \leq \ep(i) \int_{\Omega} \,d\mu_{\ti g_i(t)},
\end{equation*}
  we see that 
  \begin{equation*}
  \omega_n (1-\gamma(\eta)) \leq \int_{\Omega}\,d\mu_{\ti g_i(\eta)} \leq e^{\ep(i)\eta} \int_{\Omega}\,d\mu_{\ti g_i(r)} ,
  \end{equation*}
   that is
$   \vol_{\ti g_i(r)} (B_{\ti g_i(r)}(y_i,1)) \geq  \omega_n (1-\gamma(\eta))$ for arbitrary $r\in (0,\eta)$, $i \geq N(\eta).$  
Once again, $\Rc(\ti g_i(t)) \geq -\ep(i)$ for all $t>0$ coupled with Bishop-Gromov volume comparison tells us that 
 $\vol_{\ti g_i(r)} (B_{\ti g_i(r)}(y_i,\ell)) \geq  \omega_n (1-\gamma(\eta))\ell^n$ for all $\ell \leq 1,$ for all $i\geq N(\eta),$ where we assume without loss of generality, $\ep(i) \leq \eta$  for all $i\geq N(\eta).$
 Scaling the solution once more by $\hat g_i(t) = \ti t_i^{-1}\ti g_i(\ti t_i t)$, we see on the one hand that $| \Rm(\hat g_i(1))|(y_i) \in [\ep,c_0]$ and on the other hand,   $\vol_{\hat g_i(t)} (B_{\hat g_i(t)}(y_i,\ell)) \geq  \omega_n (1-\gamma(\eta))\ell^n$ for all  $t\in (0,\eta{\ti t_i}^{-1}),$   for all $i \geq N(\eta).$  Taking a subsequence if necessary,   we can assume that
   $(M, \hat g_i(t) ,y_i)_{t\in (0,T_i)}  \to (Y,\hat g(t),p)_{t\in (0,\infty)}$ as $i\to \infty$ where  $T_i:=\eta \ti t_i^{-1}$. Moreover, $ (Y,\hat g(t),p)$ satisfies $\Rc(\hat g(t)) \geq 0$, $|\Rm(\hat g(t))|\leq c_0/t$ for all $t>0$ and   $   \vol_{\hat g(t)} (B_{ \hat g(t)}(p,\ell))  \geq  \omega_n (1-\gamma(\eta))\ell^n$ for all $l>0$, $t>0$. 
Thus letting $\eta \to 0$ we have  $\vol_{\hat g(t)} (B_{ \hat g(t)}(p,\ell))  \geq  \omega_n\ell^n,$ and hence Bishop-Gromov volume comparison implies 
   $   \vol_{\hat g(t)} (B_{ \hat g(t)}(p,\ell))  =\omega_n \ell^n,$ for all $l>0$ and so  
 $ (Y,\hat g(t),p)$ is flat and isometric to $(\R^n,\de,0)$ for each   $t\in (0,\infty)$. 
 This is a contradiction  to $|\Rm(\hat{g_i}(1))|(y_i) \in [\ep,c_0]$.\\[2ex]
{\bf Case 2: There exists a subsequence  of $((y_i,t_i))_{i \in \N}$  so that $d^2_0(y_i,z)/ t_i\to 0$ as $i\to \infty$.   }\\[1ex]
We scale solutions and initial data so that $\ti d_i(0) :=  \la_i   d(0),$ for $\la_i := t_i^{-1/2}  $ so that
$\ti d^2_i(0)(y_i,z) = d^2_0(y_i,z)/t_i\to 0 $. Since by assumption, $(M,\ti d_i(0), z)_i \to (\R^n,d(\delta),0)$    in the Gromov-Hausdorff sense  as $i\to \infty,$ so does the sequence $(M,\ti d_i(0), y_i)_i$. Then  the new times  $\ti t_i:=  t_i  \la^2_i  $   for the scaled solution  $ (M, \ti g_i(\ti t) = \la^2_i g(\la^{-2}_i\ti t ))_{\ti t \in (0, \lambda_i^2)}$ satisfy $\ti t_i=1,$ and as before,    $ (M, \ti g_i(\ti t),y_i)_{\ti t\in (0,\lambda^2_i)}  \to (X,h(t),p)_{t\in (0,\infty)}$
as $i\to \infty,$  where $(X,h(t),p)_{t\in (0,\infty)}$ is a solution to Ricci flow satisfying $\Rc(h(t)) \geq 0$ and $ |\Rm(h(t))| \leq c_0/t$  for all $t>0,$ which approaches $(\R^n,d(\delta),0)$ in the Gromov-Hausdorff sense as $t\downto 0,$ in view of the distance estimates \eqref{dist-est-app-A}. 

 The argument presented in the proof above before considering Case $1.1$ shows us that  $(X,h(t),p)$ is isometric to $(\R^n,d(\delta),0)$ for each $t>0.$ 
But $| \Rm(\ti g_i(\ti t_i))|(y_i) =  |\Rm(\ti g_i(1))|(y_i) \in  [\ep,c_0]$
    which is   a contradiction for large enough $i$.
 \end{proof}

\newpage
 \section{A splitting theorem of R.~Hochard for  the Ricci flow} 
 
 We recall an unpublished result due to Hochard \cite{HochardThesis}.
 
 \begin{prop}[\text{\cite[Lemma I.3.13]{HochardThesis}}]\label{prop:Hochard-splitting}
 Let $(M^n,g(t))_{t\in(0,T)}$ be a complete Ricci flow such that for $t\in(0,T)$ and for some $c_0>0$,
 \begin{equation}\label{eq:hochard_conditions}
 |\Rm(g(t))|\leq \frac{c_0}{t},\quad \operatorname{inj}_x(g(t))\geq\sqrt{\frac{t}{c_0}},\quad \Rc(g(t))\geq  0.
   \end{equation}
   Let $(M,d_0)\coloneqq \lim_{t\rightarrow 0^+}(M,d_{g(t)})$ denote the limit metric space obtained from Lemma \ref{SiTopThm1} and assume it is non-collapsed at all scales, i.e. $\mathcal{H}^n(B_{d_0}(x,r))\geq vr^n$ for all $r>0$ and $x\in M$ for some $v>0$. Assume furthermore that there exists a metric space $(X,d_X)$ and for some $1\leq m\leq n$ an isometry  
 $$ \phi: (X\times\R^m,d_{X\times \R^m}) \to (M,d_0),$$ 
 where $d_{X\times \R^m}$ denotes the product metric on $X\times \R^m$. Then  
 \begin{enumerate} 
\item  $X$ can be given a smooth manifold structure, that is there exists an isometry $\Psi: (N,d_N) \to (X,d_X)$   where $N$ is a smooth manifold, the metric $d_X$ not necessarily being smooth.\\
\item    For the resulting isometry 
 $\si : (N \times \R^m,d_{N \times \R^m} ) \to (M,d_0),$  $\si(y,r) := \phi(\Psi(y),r) $, the following holds:\\
\   a) There exists a smooth complete Ricci flow solution   $(N,h(t))$   such that $(N,d_{h(t)})$ converges to  $(N,d_N)$ in the distance sense as $t\downto 0$,\\
 \   b) $\si^{*}g(t) = h(t) \oplus \de_{\R^m}$ for all $t\in (0,T),$ where $\de_{\R^m}$ is the standard Euclidean metric on $\R^m.$
\end{enumerate}
 \end{prop}
 
 The proof of Proposition \ref{prop:Hochard-splitting} uses the notion of a $\delta$-static point.
  In our setting it is sufficient to consider a slightly stronger property, which we 
call {\it strongly-$\delta$ static}. For the sake of completeness we define also
{\it $\de$-static points}, as they were defined originally in the thesis of Hochard, 
see \cite[Definition I.3.14]{HochardThesis}. 

 \begin{defn}[\text{\cite[Definition I.3.14]{HochardThesis}}]\label{defn-static-pt}
Let $(M^n,g(t))_{t\in[0,T)}$ be a solution to Ricci flow (which is not assumed to be complete for each time slice a priori) and let $\delta>0$. A point $x\in M$ is said to be {\it strongly $\delta$-static} (respectively {\it  $\delta$-static}) on $[0,t]$ if for all $s\in[0,t]$, the geodesic ball $B_{g(s)}\big(x,\frac{\gamma(n)}{\delta}\sqrt{s}\big)$ is relatively compact in $yM$  and if $ |\Rm(g(s))|_{g(s)}\leq \frac{\delta^2}{\gamma(n)^2s}$ (respectively $ |\Rc(g(s))|_{g(s)}\leq \frac{\delta^2}{\gamma(n)^2s}$)
on $B_{g(s)}\big(x,\frac{\gamma(n)}{\delta}\sqrt{s}\big)$, where $\gamma(n)$ is the constant from \eqref{distestinit} 
 \end{defn} 
  {\bf Remark}:  By definition, $p$ is {\it strongly $\delta$-static} on $[0,t]$ implies 
  $p$ is {\it $c(n)\delta$-static} on $[0,t].$ 
 \begin{lemma}\label{lemma-existence-static-pt}
 Let $(M^n,g(t))_{t\in[0,T]}$  be a complete solution to Ricci flow satisfying \eqref{eq:hochard_conditions} and assume it is non-collapsed at all scales, i.e. $\mathcal{H}^n(B_{d_0}(x,r))\geq vr^n$ for all $r>0$ and $x\in M$ for some $v>0$. 
 
Then for all $\de>0$ and  all   $x_0\in M$ such that $B_{d_0}(x_0,1) \Subset M,$ 
there exists an $x\in B_{d_0}(x_0,\delta)$ and $r(\de,v,n)>0$, $S=S(\de,v,n)>0$ such that
$$|\Rm(g(t)) |_{g(t)} \leq \frac{\de^2}{t},\quad\text{ on $B_{d_0}(x,r)$ for all $t\in   [0,T]\cap[0,S]$.}$$
  In particular, for each $x_0\in M$ and $\delta>0$, there exists an $\ep=\ep(\delta,n) >0$ and a strong $\delta$-static point $x\in B_{d_0}(x_0,\delta)$ on $[0, \ep(\de,n)^2]$.
 \end{lemma}
 
 \begin{rmk}
 Lemma \ref{lemma-existence-static-pt} proves a stronger statement than \cite[Corollary I.3.17]{HochardThesis} in the sense that the full curvature tensor is shown to be smaller than expected on an appropriate ball. However, \cite[Lemma I.3.16]{HochardThesis} 
 proves  the existence of $\de$-static points  in a more general setting.
 \end{rmk}
 \begin{proof}
 Let $\de>0$ and $n\in \N$ be fixed, and assume the lemma  is incorrect.
 Then we get a sequence of solutions $(M_i,g_i(t))_{t\in [0,T_i]}$ and points $x_{0,i}$ such that 
 $B_{d_{0}}(x_{0,i},1)   \Subset M$ but the statement fails.
 A subsequence of the   starting balls $(B_{d_{0,i}}(x_{0,i},1) ,d_{0,i}, x_{0,i})$ (respectively starting spaces $(M_i ,g_i(0), x_{0,i})$)  converges  in the pointed Gromov Hausdorff sense to $(B_{d_{0}}(x_{0},1),d_{0}, x_{0})$ (respectively $(X,d_0,x_{0})).$ 
 Take a point $y \in B_{d_{0}}(x_{0},\frac 1 2)$ which is a regular point: such a point exists due to \cite[Theorem $6.1$]{Cheeger-Colding}.
Scale $(X,d_0)$  by a large constant $C$ so that (after scaling)  $B_{d_0}(y,1)$ is $\ep$-Gromov-Hausdorff close to a Euclidean ball.
Then we have that $B_{d_{0,i}}(y_i,1)$ is $2\ep $ Gromov-Hausdorff close to 
a Euclidean ball, for $i$ large enough, where we have scaled all solutions and starting spaces by the same large constant $C$. 
Since $\Rc(g_i(0))\geq 0$, the volume convergence theorem \cite[Theorem $5.9$]{Cheeger-Colding} implies that the volume of the balls
$B_{d_{0,i}}(y_i,1)$  are $\ti \ep(\ep,n)$ close to the volume of the Euclidan ball. 
Now, using \cite[Corollary 1.3]{Cav-Mon},  then shows us that
 the isoperimetric profile of $B_{d_{0,i}} (y_i,1)$ is $\psi(\ep,n)$-close to the Euclidean isoperimetric profile of a ball of radius one, where  $\psi(\ep,n)  $  is a constant depending only on $n$ and $\ep$ and  $\psi(\ep,n)  \to 0$ as $\ep \to 0$ for fixed $n\in \N$.

Thus Perelman's Pseudolocality theorem \cite[Theorem 10.1]{P1} implies that  
\begin{equation*}
|\Rm(g_i(t))|_{g_i(t)}\leq  \frac{\de^2}{t},\quad\text{  for all $t\in (0,L(n,\ti \ep))\cap [0,T_i]$ on $B_{d(g_i(t))}(y_i,R(n,\ti \ep)),$}
\end{equation*}
 and then on $B_{d_{0,i}}(y_i,\ti R(n,\ti \ep)),$ for all $t\in (0,\ti L(n,\ti \ep ))$ using the distance estimates from Lemma \ref{SiTopThm1}, if $\ep$ is chosen small enough. 
Scaling back by the small  $\frac{1}{C}$, we see that this  contradicts the fact that we cannot find a uniform radius and time  centered around any point in $B_{d_{0,i}}(x_{0,i},1)$ for which $|\Rm(g_i(t))|_{g_i(t)}\leq \de^2/t$  holds.


As for the last statement, let $\de>0$ and $x_0\in M$ such that $B_{d_0}(x_0,1) \Subset M.$ 
Then by what we have just proved applied to $\delta/\gamma(n)$, there exists an $x\in B_{d_0}(x_0,\delta)$ and $r=r(\de,v,n)>0$, $S=S(\de,v,n)>0$ such that
$$|\Rm(g(t)) |_{g(t)} \leq \frac{\de^2}{\gamma(n)^2t},\quad\text{ on $B_{d_0}(x,r)$ for all $t\in   [0,T]\cap[0,S]$.}$$
According to Lemma \ref{SiTopThm1} (and its notations), $B_{g(s)}\left(x,\frac{\gamma(n)}{\delta}\sqrt{s}\right)\subset B_{d_0}(x,r)$ provided $s\in[0,\ep(\de,v,n)^2]$ where $$\ep(\de,v,n)^2:=\min\left\{S,T,\left(\frac{ r}{\gamma(n)}\frac{1}{\sqrt{c_0}+\delta^{-1}}\right)^2\right\}.$$

\end{proof}

 \begin{proof}[Proof of Proposition \ref{prop:Hochard-splitting}]
 We will only give a proof for $m=1$. Without loss of generality, after rescaling, one can assume $T=1$. Let $\phi:(X \times \R, d_{X \times \R}) \to M$ be the homeomorphism which is an isometry at time $0$
$(M,d_t) \to (M,d_0)$ as $t \downto 0$ and $\phi:(X \times \R,d_{X \times \R}) \to (M,d_0)$ is an isometry so that $d_0((y,u),(y',u'))^2=d_X(y,y')^2+|u-u'|^2$ for all $(y,u)$ and $(y',u')$ in $X\times\R$.

 Let $\delta>0$ and $R>0$. Let $x_0:=(y_0,0)\in X\times\R$ and consider the points $p_{\pm R}:=\phi((y_0,\pm R))$. 
 
 Given $\delta >0$,  let $\ep= \ep(\de,n)>0$ be the constant   from  Lemma \ref{lemma-existence-static-pt}. 
 Applying 
   Lemma \ref{lemma-existence-static-pt}
   to the   rescaled solution $\ti g(t) =  \ep^2  g(\frac{1}{\ep^2} t)$  on  $B_{\ti d_0}(p_{+ R},1)$) respectively 
 $B_{\ti d_0}(p_{-R},1)$,     we obtain    $\delta$-static points   $p'_{ + R} \in B_{d_0}(p_{+ R},\delta{\ep}^{-1})$ 
  respectively  $p'_{ -R} \in B_{d_0}(p_{-R},\delta{\ep}^{-1})$ 
  for the solution
 $g(t)$ on the time interval $[0,1].$
 
  Let us define the following functions:
 \begin{equation*}
 \begin{split}
   b_{\pm,R}(x,t)&:=d_t(p'_{\pm R},x)-d_t(p'_{\pm R},x_0),\\
 \mathcal{E}_{R}(x,t)&:=d_t(p'_{+R},x)+d_t(p'_{-R},x)-d_t(p'_R,p'_{-R}),
  \end{split}
 \end{equation*}
 and observe that by the triangle inequality, one has for all $x\in M$ and $t>0$:
 \begin{equation}\label{excess-fct-controls-b}
 \mathcal{E}_{R}(x,t) \geq b_{+R}(x,t)+b_{-R}(x,t)\geq - \mathcal{E}_{R}(x_0,t).
\end{equation}
   Using  the fact that $p'_R,p'_{-R}$ are $\de$-static on $[0,1]$ (see \ref{defn-static-pt}), we may apply  
Perelman's distance estimates, Lemma 8.3 of \cite{P1},  to obtain that the Lipschitz function (in time) 
$d_{t}(p'_R,p'_{-R})$ satisfies 
  $\partt d_{t}(p'_R,p'_{-R}) \geq -\frac{\de}{\sqrt{t}}$ at all points $t\in [0,1]$  where the function is differentiable, and hence
  $d_t(p'_R,p'_{-R})\geq d_0(p'_R,p'_{-R})-\delta\sqrt{t}.$  Continuity of $d_t$ up to $t=0$ and non-negativity of the Ricci curvature of the flow gives that $d_t\leq d_0$ on $M$ for $t\in[0,1],$

 Therefore, 
 \begin{equation}\label{est-excess-t}
  0\leq \mathcal{E}_{R}(x,t)\leq \mathcal{E}_{R}(x,0)+\delta\sqrt{t},\quad\text{ for all $x\in M$ and $t\in[0,1]$.}
  \end{equation}
 
 Now, for $(x,t)\in M\times[0,1]$, the triangle inequality gives,
 \begin{equation*}
 \begin{split}
 d_t(p'_R,x)&\geq d_t(p'_R,p'_{-R})-d_t(p'_{-R},x)\\
 &\geq d_0(p'_{R},p'_{-R})-d_0(p'_{-R},x)-\delta\sqrt{t}\\
 &=d_0(p'_R,x)-\mathcal{E}_{R}(x,0)-\delta\sqrt{t},
 \end{split}
 \end{equation*}
 which thanks to $d_t(p'_R,x_0)\leq d_0(p'_R,x_0)$  leads to $$d_t(p'_R,x)-d_t(p'_R,x_0)\geq d_0(p'_R,x)-d_0(p'_R,x_0)-\mathcal{E}_{R}(x,0)-\delta\sqrt{t}.$$
 By reversing the roles of $x$ and $x_0$, one gets:
 $$d_t(p'_R,x)-d_t(p'_R,x_0)\leq d_0(p'_R,x)-d_0(p'_R,x_0)+\mathcal{E}_{R}(x_0,0)+\delta\sqrt{t},$$
 which gives:
 \begin{equation}\label{excess-busemann}
-\mathcal{E}_{R}(x,0)-\delta\sqrt{t}\leq b_{+,R}(x,t)-b_{+,R}(x,0)\leq \mathcal{E}_{R}(x_0,0)+\delta\sqrt{t}.
\end{equation}

By the product structure of $d_0$, one can check by a direct computation that there exists a universal constant $c>0$ such that for $x\in B_{d_0}(x_0,\delta^{-1})$ and $R\geq 20(\de^{-1} + \de\ep^{-1})$:
  \begin{equation}\label{hello-euclidean-geometry}
  \mathcal{E}_{R}(x,0)  \leq \frac{c}{R}\left (\frac{1}{\de^2} + \frac{\delta^2}{\ep^2}  \right).
  \end{equation}
We give here an explanation of this fact. \\

At time zero, we consider everything on $(X\times \R, d_{X \times \R}).$
Since the distance $d_0$ on $X \times \R  $ is given by the product structure, and since  $p'_{ \pm R}\in B_{d_0}((y_0, \pm R),\delta{\ep}^{-1})$, we see, using   the triangle inequality, that $p'_{ + R} =( y_{+},R+\al_{+} )$    and  $p'_{ - R} =(  y_{-},-R+\al_{-} )$ for some $ \al_{\pm} \in \R,$ and 
 \begin{equation}\label{hello-easy-est}
\max\{d_X(y_{+},y_0), d_X(y_{-},y_0),d_X(y_{+},y_{-}),   |\al_{\pm}|\}  \leq d_0(p_{\pm R},p'_{\pm R})\leq   \delta\varepsilon^{-1}.
 \end{equation}
Also, remembering $x_0 = ( y_0,0)$, we see that $x\in B_{d_0}(x_0,\delta^{-1})$
means $x = (y,\alpha)$ for some $\alpha\in \R$ and $y \in X$  where $|\alpha| \leq  \delta^{-1}  $ and $d_X( y, y_0) \leq   \delta^{-1} .$ 
In particular, $d_1 := d_X( y, y_{+}) \leq   \delta\varepsilon^{-1} +   \delta^{-1} ,$ 
and  $d_2 := d_X( y, y_{-}) \leq   \delta\varepsilon^{-1} +   \delta^{-1},$
and $d_3:= d_X(y_+,y_{-})  \leq 2\delta\varepsilon^{-1} +   2\delta^{-1}$ 
in view of the triangle inequality.  

We insert this information, and the definition of $d_0$  into the definition  of $\mathcal{E}_{R}(\cdot,0)$ to get: 
 \begin{equation*}
 \begin{split}
  \mathcal{E}_{R}(x,0) & = d_0( (y_{+} ,R+\al_{+}),(y,\alpha)) + d_0( (y_{-} ,-R+\al_{-}),(y,\alpha)) \\
 &\quad -d((y_{+} ,R+\al_{+}), (y_{-} ,-R+\al_{-})\\
&  = \sqrt{ (R + {\al_{+}} -\alpha)^2 +(d_X( y, y_{+}))^2 }  +  \sqrt{ (-R + {\al_{-}} -\alpha)^2 + (d_X( y, y_{-}))^2 } \\
&  \quad - \sqrt{   (2R +\al_{+} - \al_{-})^2 +d_X^2(y_+,y_{-})   }\\
& =   \sqrt{ (R + b_1)^2 +d_1^2 }  +  \sqrt{ (-R + b_2)^2 +d_2^2 }  - \sqrt{   (2R + (b_1 - b_2))^2 +d_3^2   },
\end{split}
 \end{equation*}
where $b_1:=  \al_{+} -\alpha$, $b_2 := \al_{-} -\alpha$, which implies
$b_1-b_2=  \al_{+}- \al_{-}.$  
We claim that  $$ \sqrt{ (R + b_1)^2 +d_1^2 }  +  \sqrt{ (-R + b_2)^2 +d_2^2 }  \leq \sqrt{   (2R + (b_1 - b_2))^2 +d_3^2   } + 
20  \frac{v}{R},$$ where $v := b_1^2 + b_2^2 + d_1^2 +d_2^2+d_3^2\leq 10(\de^2\ep^{-2}+\de^{-2}) $ under the  assumption that $ R \geq 20(\delta^{-1}+\delta\ep^{-1})$ which ensures in particular that $R\geq 10|b_1|+10|b_2|.$   This will imply  the desired result \eqref{hello-euclidean-geometry}. Now,
\begin{equation*} 
\begin{split}
 \sqrt{   (2R + (b_1 - b_2))^2 +d_3^2   } + 
20 \frac{v}{R}  
   & \geq 2R + (b_1 - b_2) +  
20  \frac{v}{R} \\
& = \Big( (R +  b_1) + 10 \frac{v}{R}\Big)  + \Big( (R - b_2)+ 10 \frac{v}{R} \Big).
\end{split}
\end{equation*}
The term $\Big( (R +  b_1) + 10 \frac{v}{R}\Big)$ may be estimated as follows:
\begin{equation*}
\begin{split} 
  (R +  b_1) + 10 \frac{v}{R} 
 & =
\sqrt{ \Big( (R +  b_1) + 10 \frac{v}{R}\Big) ^2} \\
&=  \sqrt{   (R +  b_1)^2  + 20  \frac{v(R+b_1)}{R} +  100 \frac{v^2}{R^2}}\\
& \geq \sqrt{   (R +  b_1)^2  + 10\,v }\\
& \geq \sqrt{   (R +  b_1)^2  + d_1^2}.
\end{split}
\end{equation*}
Similarly  $ (R - b_2) + 10 \frac{v}{R} \geq \sqrt{   (R -  b_2)^2  + d_2^2},$ so that, 
$$ 
  \mathcal{E}_{R}(x,0) \leq 20 \frac{v}{R}, $$ as required.

Let us choose now a sequence $(\delta_i)_i$ converging to $0$ and let $R_i:=c\delta_i^{-1}(\de_i^{-2}+\ep_i^{-2}\de_i^2).$ To points $p_{\pm R_i}$ correspond $\delta_i$-static points on $[0,1]$, $p^{\pm}_i:=p'_{\pm R_i}$ which lie in $B_{d_0}(p_{\pm R_i}, \ep_i^{-1}\delta_i)$. For each $i$ and $t\in(0,1]$, let $\gamma_{i,t}^{\pm}$ be a minimizing geodesic between $x_0$ and $p^{\pm}_i$. Up to extracting a subsequence, $\gamma_{i,t}^{\pm}$ can be assumed to converge to a ray $\gamma_{t}^{\pm}$ originating from $x_0$. 

We claim that the concatenation of $\gamma_t^-$ and $\gamma_t^+$ is a geodesic line which goes through $x_0$, which we denote $\ga_t: \R \to M,$ with respect to $g(t)$ for all $t\in(0,1]$. 

Indeed, for $t\in(0,1]$ and $s_{\pm}\geq0$, then for $i$ sufficiently large such that 
\begin{equation*}
\min\{d_t(x_0,p_{i}^{+}),d_t(x_0,p_{i}^{-})\}\geq \max\{s_+,s_{-}\},
\end{equation*}
one gets thanks to the fact that $\gamma_{i,t}^{\pm}$ are rays with respect to $g(t)$,
\begin{equation*}
\begin{split}
0&\leq d_t(\gamma_{i,t}^+(s_+),x_0)+d_t(x_0,\gamma_{i,t}^-(s_{-}))-d_t(\gamma_{i,t}^+(s_+),\gamma_{i,t}^-(s_{-}))\\
&=d_t(p_i^+,x_0)+d_t(x_0,p_i^{-})-d_t(p_i^+,p_i^{-})\\
&\quad +d_t(p_i^+,p_i^{-})+(s_+-d_t(p_i^+,x_0))+(s_{-}-d_t(x_0,p_i^{-}))-d_t(\gamma_{i,t}^+(s_+),\gamma_{i,t}^-(s_{-}))\\
&\leq \mathcal{E}_{R_i}(x_0,t)+d_t(p_i^+,\gamma_{i,t}^+(s_+))+\underbrace{(s_+-d_t(p_i^+,x_0))}_{=-d_t(p_i^+,\gamma_{i,t}^+(s_+))}\\
&\quad+d_t(\gamma_{i,t}^-(s_{-}),p_i^{-})+\underbrace{(s_{-}-d_t(x_0,p_i^{-}))}_{=-d_t(p_i^-,\gamma_{i,t}^+(s_-))}\\
&\leq \mathcal{E}_{R_i}(x_0,t)\leq \mathcal{E}_{R_i}(x_0,0)+\delta_i\sqrt{t},
\end{split}
\end{equation*}
where we have used the triangle inequality, the fact that 
 $\mathcal{E}_{R_i}(x_0,t)  = d_t(p_i^+,x_0)+d_t(x_0,p_i^{-})-d_t(p_i^+,p_i^{-})$ per definition, and   the  estimate \eqref{est-excess-t} in the last line.

Thanks to estimate \eqref{hello-euclidean-geometry} applied to $\delta_i$, $\varepsilon_i$ and $R_i$ as defined above, one gets the expected claim by letting $i$ tend to $\infty$. 

Now, the  functions 
$ b_{\pm,R_i}(x,t):=d_t(p'_{\pm R_i},x)-d_t(p'_{\pm R_i},x_0)$  
are Lipschitz with Lipschitz constant equal to $1$, and differentiable almost everywhere, 
and bounded from above and below, in view of the triangle inequality, 
\begin{equation*}
 b_{\pm,R_i}(x,t) = d_t(p'_{\pm R_i},x)-d_t(p'_{\pm R_i},x_0) \geq -d_t(x,x_0),\quad\text{ and}\quad
 b_{\pm,R_i}(x,t) \leq d_t(x,x_0).
\end{equation*}
Hence taking a limit up to a subsequence, we obtain two functions $b_{\pm,\,t}:M \to \R$ which are Lipschitz, with Lipschitz constant $1.$
Furthermore, for  points $x_i = \ga^{+}_{i,t}(s)$ we have  $b_{\pm,R_i}(\ga^{+}_{i,t}(s),t) 
= d_t(p'_{\pm R_i}, \ga^{+}_i(s) )-d_t(p'_{\pm R_i},x_0) = -s,$ which guarantees that 
$b_{\pm,\,t}(\ga_t(s))= \lim_{i\to \infty} b_{\pm,R_i}(\ga^{+}_{i,t}(s),t) =-s.$

Now, by mimicking the proof of Cheeger-Gromoll's Splitting Theorem, thanks to \eqref{excess-fct-controls-b}, \eqref{excess-busemann} and \eqref{hello-euclidean-geometry}, we infer that on $M$,
\begin{equation}
b_{+,\,t}+b_{-,\,t}=0.
\end{equation}
Since $\Delta_{g(t)}b_{\pm,\,t}\leq 0$ in the weak sense, we also get that $b_{+,\,t}$ is a weak harmonic function hence a smooth harmonic function by elliptic regularity. Furthermore, the Bochner formula applied to $b_{+,\,t}$ gives that $|\nabla^{g(t)}b_{+,\,t}|_{g(t)}^2$ is a subharmonic function since $\Rc(g(t))\geq 0$ by assumption:
\begin{equation*}
\Delta_{g(t)}|\nabla^{g(t)}b_{+,\,t}|^2_{g(t)}=2|\nabla^{g(t),2}b_{+,\,t}|^2_{g(t)}+2\,\Rc(g(t))( \nabla^{g(t)}b_{+,\,t},\nabla^{g(t)}b_{+,\,t})\geq 0.
\end{equation*}

As $|\nabla^{g(t)}b_{+,\,t}|^2_{g(t)}\leq 1$ on $M$ and since $|\nabla^{g(t)}b_{+,\,t}|^2_{g(t)}= 1$ on the line $\gamma_t$, the strong maximum principle shows that $|\nabla^{g(t)}b_{+,\,t}|^2_{g(t)}= 1$ on $M$.

Finally, the estimate \eqref{excess-busemann} above shows us that $b_{\pm,\,0}(x)  := \lim_{i\to \infty} b_{\pm,R_i}(x,0)=\lim_{i\to \infty} b_{\pm,R_i}(x,t) =: b_{\pm,\,t}(x).$

Therefore, there exists a smooth function $b_{+,\,0}$ independent of time $t\in(0,1]$ with gradient equal to $1$ and whose Hessian vanishes identically with respect to each metric $g(t)$, $t\in(0,1]$.

Let $(h(t))_{t\in(0,1]}$ be the family of metrics on the $(n-1)$-dimensional smooth manifold $N:=b_{+,\,0}^{-1}(\{0\})$ induced by $(g(t))_{t\in(0,1]}$. As $(N,h(t))$ is totally geodesic in $(M,g(t))$ for each $t\in(0,1]$, the family $(N,h(t))_{t\in(0,1]}$ is a solution to the Ricci flow.

Moreover, for each time $t\in(0,1]$, there exists a diffeomorphism $\varphi_t:N\times\R\rightarrow M$ such that $\varphi_t^*g(t)=dr^2\oplus h(t)$. More precisely, if $(\psi^t_{r})_{r\in\R}$ denotes the flow generated by $\nabla^{g(t)}b_{+,0}$ such that $\psi^t_{r=0}=\operatorname{Id}_M$ then $\varphi_t(y,r):=\psi^t_{r}(y)$ for $y\in N$. In particular, since $\Rc(\varphi_t^*g(t))(\partial_r,\cdot)=0$ and since the Ricci curvature is invariant under diffeomorphism, $\Rc(g(t))(\nabla^{g(t)}b_{+,0}\,,\cdot)=0$ which implies that the vector fields $\nabla^{g(t)}b_{+,0}$ are time-independent for $t\in(0,1]$ since $\partial_tg(t)=-2\Rc(g(t))$. Therefore, $\psi^t_{r}=\psi^1_{r}$ for all $r\in\R$ and $t\in(0,1]$ which implies $\varphi_t=\varphi_1$ for each $t\in(0,1]$ so that $\varphi_1^*g(t)=dr^2\oplus h(t)$ for $t\in(0,1]$. It also shows there is a common line $\gamma$ such that $\gamma(0)=x_0\in M$ with respect to each $g(t)$, $t\in(0,1]$.


Using very similar calculations to the ones used above  to show 
$\lim_{i\to \infty} \mathcal{E}_{R_i}(x ,0) = 0$, we see  $b_{+,\,0}( \phi(y,\alpha)) = -\alpha$ for all $(y,\alpha) \in X \times \R$.

Indeed, by a slight abuse of notation, examining the quantities on $X \times \R$, we have  
\begin{equation*}
\begin{split}
\left| b_{+,\,R}(( y,\al),0) +\alpha\right|&=\left| d_{X \times \R}((y_{+},R + \al_{+}), (y,\al)) - d_{X \times \R}((y_{+},R +\al_+), (y_0,0)) +\alpha\right|\\
&=\left|  \sqrt{ ( R + \al_{+}-\al)^2 + d_X(y_{+},y)^2} - \sqrt{ ( R + \al_{+} )^2 + d_X(y_{+},y_0)^2}+\alpha\right|
\\
&\leq   \left|(R + \al_{+}-\al ) - (R + \al_{+}) +\alpha\right| +10\frac{v}{R}  \\
&\leq  10\frac{v}{R}  \to 0,
\end{split}
\end{equation*}
 as $R= R_i \to \infty$ if we take a sequence  $R=R_i  $ as above. 

Since the geodesic $\ga: \R \to M$ with $\ga(0) = \phi(y_0,0)$ is a geodesic for all times $t\in(0,1]$, we see that it is also one at time $t=0$.
Writing $\hat \ga(s)= \phi^{-1}(\ga(s)) = (y(s),r(s))$, we obtain a geodesic in $(X \times \R, d_{X \times \R}).$
From the above $-s= b_{+,\,0}(\ga(s)) = \varphi^*b_{+,\,0}(\hat \ga(s)) = 
\varphi^*b_{+,\,0}(y(s),r(s)) = -r(s).$ That is $r(s)=s$.
As $s \to  \hat \ga(s) = (y(s),r(s)) =(y(s),s)$ is a   geodesic satisfying $d_{X \times \R} (\hat \ga(s), \hat \ga(s')  ) =|s-s'|$, for all $s$, $s'$ in $\R$, in the product space, we see that
$y(s)=y_0$.
Hence the geodesic line  $\ga$ at time $t=0$  is mapped onto the geodesic line  $s\to (y_0,s)$ in the product $(X\times \R)$ via $\phi^{-1}$ as required, and we see that $\ga(s) = \phi(y_0,s)$ for all $s \in \R.$
Since $(N,h(t))_{t\in (0,T)}$ also satisfies $\Rc(h(t)) \geq 0$ and $|\Rm(h(t))|\leq \frac{c_0}{t}$ we see that
$ d_{h(t)} \to  d_N$ as $t\downto 0$  for some metric $d_N$ on $N$ satisfying  $  d_{h(t)} -\ga(n)\sqrt{c_0t}   \geq d_N \geq d_{h(t)}.$
   Defining  $\Psi:N \to X$ by $\Psi(y) = \Pi_X(\phi^{-1}(\phi_1(y,0)))$ where $\Pi_X(x,s) =x, $
we see that $\Psi:N \to X$ is a homeomorphism which also satisfies that $\Psi: (N,d_N) \to (X,d_X) $ is an isometry, and the map  $\si:(N\times \R) \to M,$ $\si(y,r) = \phi(\Psi(y),r)=\varphi_1(y,r) $ is a diffeomorphism which also satisfies that
$\si : (N\times \R,d_{N\times \R})  \to  (M,d_0)$ is an isometry.
Furthermore
$\si^{*}g(t) = \varphi_1^*g(t) = h(t)\oplus dr^2,$ as required.

 \end{proof}
\end{appendix}
\newpage 
\bibliography{ham-conj}
\bibliographystyle{amsalpha}

\end{document}